\RequirePackage{fix-cm}
\documentclass[smallextended]{svjour3}       
\smartqed  
\usepackage{lineno}

\usepackage{amsmath,amssymb}
\usepackage{stmaryrd}
\usepackage{enumerate}
\usepackage{algorithm}
\usepackage[noend]{algpseudocode}
\usepackage{calrsfs}

\usepackage{physics}

\usepackage{graphicx}
\usepackage{amssymb}
\usepackage{array}
\usepackage{amsmath}
\usepackage{epstopdf}
\usepackage{epsfig}
\usepackage{subfigure}
\usepackage{pstricks}
\usepackage{fancyheadings}

\usepackage{stackengine}

\usepackage[margin=1.0in]{geometry}
\DeclareMathAlphabet{\pazocal}{OMS}{zplm}{m}{n}

\usepackage{graphicx}

\usepackage{stackengine}

\makeatletter
\def\BState{\State\hskip-\ALG@thistlm}
\makeatother

\newcommand{\lJump}{[\![}
\newcommand{\rJump}{]\!]}

\begin{document}

\title{A stable discontinuous Galerkin  method for the perfectly matched layer for elastodynamics in first order form
}

\titlerunning{A stable DG  method for the PML for elastodynamics in first order form}        

\author{Kenneth Duru         \and Leonhard Rannabauer \and Alice-Agnes Gabriel  \and  Gunilla Kreiss  \and  Michael Bader
}


\institute{Kenneth Duru \at
              Mathematical Sciences Institute, Australian National University, Canberra, Australia \\
              Tel.: +61 2 6125 4552\\
              \email{kenneth.duru@anu.edu.au}           
           \and
           Leonhard Rannabauer \at
              Technical University of Munich, Germany
              \and
              Alice-Agnes Gabriel \at
              Department of Geophysics, Ludwig-Maximilian University, Munich, Germany
               \and 
        Gunilla Kreiss \at
              Division of Scientific Computing, Department of Information Technology, Uppsala University, Sweden
               \and
           Michael Bader \at
              Technical University of Munich, Germany
}

\date{Received: date / Accepted: date}

\maketitle

\begin{abstract}
We present a stable  discontinuous Galerkin (DG)  method  with a perfectly matched layer (PML) for three and two space dimensional  linear elastodynamics, in velocity-stress formulation, subject to well-posed linear boundary conditions. 

 First, we consider the elastodynamics equation, in a  cuboidal domain, and derive an unsplit PML truncating the domain using complex coordinate stretching.
Leveraging the hyperbolic structure of the underlying system,  we construct continuous energy  estimates, in the time domain for the elastic wave equation, and in the Laplace space for a sequence of PML model problems, with variations in one, two and three space dimensions, respectively. They correspond to PMLs normal to  boundary faces, along edges and in corners.


Second, we develop a DG numerical method for the linear elastodynamics equation  using physically motivated numerical flux and penalty parameters, which are compatible with all well-posed, internal and external, boundary conditions. When the PML damping vanishes, by construction, our choice of penalty parameters yield an upwind scheme and a discrete energy estimate analogous to the continuous energy estimate. 

Third,  to ensure numerical stability of the discretization when PML damping is present, it is necessary to extend the  numerical DG fluxes, and the numerical  inter-element and boundary  procedures, to the PML auxiliary differential equations. This is crucial  for deriving discrete energy estimates analogous to the continuous energy estimates.   
Numerical solutions are evolved in time using the high order arbitrary  derivative (ADER)  time stepping  scheme of the same order of accuracy with the spatial discretization.  By combining the DG spatial approximation with the high order ADER  time stepping  scheme and the accuracy of the PML we obtain an arbitrarily  accurate wave propagation solver in the time domain.  

Numerical experiments are presented in two and three space dimensions corroborating the theoretical results.  
%
\end{abstract}

\section{Introduction}
\label{intro}
Computational strategies based on the discontinuous Galerkin method (DG method)  are  desirable for large scale numerical simulation of wave phenomena occurring in many applications \cite{Duru_exhype_2_2019,HesthavenWarburton2002,Carsten_EtAL2018,PeltiesdelaPuenteAmpueroBrietzkeKaser2012,DumbserKaser2006}.
However, one of the main  features of propagating waves is that they can propagate long distances relative to their characteristic dimension, the wavelength.
As an example, seismic waves generated by  events that occurred in one continent can be recorded,  far away, in a distant continent, at some later time.
For numerical simulations, it is precisely this essential feature of waves, the radiation of waves to far field, that leads to the greatest difficulties, \cite{Engquist77,Berenger1994,Givoli2004,Hagstrom2003,Duruthesis2012}.
In truncated computational domains, this manifests itself as spurious reflections of outgoing waves at artificial boundaries, which will travel back into the simulation domain and destroy the accuracy of numerical simulations everywhere. Therefore, in order to ensure the accuracy of DG methods, artificial boundaries introduced to limit the computational domain must be closed with reliable and accurate boundary conditions.

The effort to design efficient domain truncation schemes for propagating waves began about forty years ago \cite{Engquist77}  and has evolved over time to become an entire area of research \cite{Givoli2004,Duruthesis2012}. One of the major outcomes of this effort is the perfectly matched layer (PML), invented by \cite{Berenger1994}  in 1994 for Maxwell's equations in electromagnetics, and has been extended to many other  equations whose solutions are composed of waves.
The PML transforms the underlying hyperbolic partial differential equation (PDE) such that all spatially oscillating solutions decay exponentially in space. This allows the PML to effectively absorb all outgoing waves without reflections, independent of angle of incidence and frequency. This desirable property makes the PML so efficient and attractive to be used in  numerical simulation of absorption of waves. However, the PML transformation  has some important mathematical and numerical consequences which has limited its use in many practical computations. 

One, although   the PML solutions decay exponentially in space, for general systems, there is no guarantee that all solutions will decay in time. This is undesirable, since any growth in the PML can propagate into the simulation domain and pollute the numerical solution. For the past twenty years,  the mathematical analysis, of well-posedness and stability, of the PML has  been an area of active research \cite{Duruthesis2012,Be_etAl,SkAdCr,AppeloKreiss2006,DuruGabKreiss2019,HalpernPetit-BergezRauch2011,DiazJoly2006}. 
Most of these analyses are based on the use of  classical Fourier methods for the PML initial value problem (IVP).  An important stability result, {\em the geometric stability condition},  was introduced in \cite{Be_etAl} to characterize the temporal stability of  PML IVPs.
 If this condition is not satisfied, then there are modes of high spatial frequencies with temporally growing amplitudes. 
 
Two, further complicating the matter, the PML is derived in the continuous setting in the absence of any boundary conditions. However, in any practical setting the PML is a layer of finite thickness surrounding a truncated domain, and can in addition interact with boundary conditions. Thus, a more complete stability analysis of the PML in truncated domains must include the underlying boundary conditions present. 
 For IVPs that satisfy the geometric stability condition, initially, it was that the introduction of boundary conditions can ruin the stability or well-posedness of the PML  \cite{SkAdCr}. 
 We note, however, that the PML stability analysis has been extended to PML initial boundary value problems (IBVPs), using normal modes analysis \cite{Duruthesis2012,DuKrSIAM,KDuru2016,DuruKozdonKreiss2016}. The central result of these analyses is: as long as the underlying undamped IBVP (without the PML) is stable, the PML IBVP is stable if the IVP satisfies the geometric stability condition.

Three, for symmetric or symmetrizable hyperbolic PDEs, the PML is generally asymmetric. It becomes extremely difficult to derive energy estimates (in the time domain) that can be used to design stable numerical methods for the PML in truncated domains. Therefore, even when the geometric stability condition is satisfied, numerical experiments have also shown that the PML can support growth. 
For the DG method and PML for elastodynamics, the growth can be catastrophic \cite{XieKomatitschMartinMatzen}, destroying the accuracy of numerical simulations.
Some of our recent results \cite{KDuru2016,DuruKozdonKreiss2016} have revealed the impact of numerical boundary procedures on the stability of discrete PMLs, using high order summation-by-parts (SBP) finite difference method. For DG methods, the main difficulty is that classical $L_2$-energy norms traditionally used to develop stable DG numerical fluxes and approximations are not directly applicable when the PML is active.

For time-dependent problems, the PML is often derived in the transform (Laplace or Fourier in time) space, and then transformed back to the time domain. In the Laplace space, for the acoustic wave equation and Maxwell's equation, it is possible to derive energy estimates for the variable coefficients PML \cite{DuruGabKreiss2019,KDuru2016} in the continuous setting.  By mimicking this energy estimate it is also possible to design an energy stable DG method for the PML for acoustic wave equation \cite{DuruGabKreiss2019}. However, for more general  systems of hyperbolic PDEs  such as the linear elastodynamics equation, where
more than one wave type and wave speed are simultaneously present, the  theory can not be easily extended in  a straightforward manner. Therefore, we will require further assumptions and simplifications for such problems.

In this study we will take a more pragmatic approach, similar to \cite{DuruKozdonKreiss2016}. 
We aim to prove energy estimates for  model equations that are strong enough to allow us to design accurate and robust  DG  methods for the general 3D PML problem, for linear elastodynamics. 
We will derive energy estimates in the Laplace space for the PML  assuming  variations only in one space dimension,  two space dimension for the PML edge problem and three space dimensions  for a PML corner problem. By mimicking  these energy estimates we will design stable DG numerical flux and  method for the PML for 3D linear elastodynamics equation. If the discrete approximation of the PML is not provably stable for these model problems, there is no hope that the approximation will be stable for the general PML problem.

To begin, we develop a DG numerical method for the linear elastodynamics equation  using physically motivated numerical flux and penalty parameters, which are compatible with all well-posed, internal and external, boundary conditions \cite{DuruGabrielIgel2017,Duru_exhype_2_2019}. When the PML damping vanishes, by construction, our choice of penalty parameters yield an upwind scheme and a discrete energy estimate analogous to the continuous energy estimate. 
As in \cite{DuruGabKreiss2019}, to ensure numerical stability of the discretization when PML  is present, it is necessary to extend the  numerical DG fluxes, and the numerical  inter-element and boundary  procedures, to the PML auxiliary differential equations. This is crucial  for deriving discrete energy estimates analogous to the continuous energy estimates.   

We discretize in time using the high order arbitrary  derivative (ADER)  time stepping  scheme \cite{Duru_exhype_2_2019,DumbserKaser2006,DumbserPeshkovRomenski} of the same order of accuracy with the spatial discretization.  By combining the DG spatial approximation with the high order ADER  time stepping  scheme and the accuracy of the PML we obtain an arbitrarily  accurate wave propagation solver in the time domain.


The remainder of the paper will proceed as follows. In section 2 we introduce the equations of 3D linear elastodynamics, well-posed boundary conditions with interface conditions, and derive energy estimates.  In section 3 we derive the PML, and derive  energy estimates for the PML, in the Laplace space, in section 4.  In section 5, we present numerical approximations, and prove numerical stability in section 6. In section 7,  we present numerical experiments, in 2D and 3D, verifying the analysis, and demonstrating the effectiveness of the PML in seismological applications. We draw conclusions in section 8.
\section{Equations}
In this section, we present the equations of  linear elastodynamics in 3D, in velocity-stress formulation. 
We will introduce well-posed boundary conditions,  and  the interface conditions that will be used to patch DG elements together.
We will end the section by deriving energy estimates for the corresponding IBVP.
\subsection{Linear elastic wave equation}
Introduce the time variable   $t\ge  0$, and the spatial Cartesian coordinates $\left(x,y,z\right) \in \Omega \subset \mathbb{R}^3$.
Consider the  3D elastic wave equation, in a heterogeneous source-free medium, 
 in  first order form,
{
\begin{equation}\label{eq:linear_wave}
\begin{split}
\mathbf{P}^{-1} \frac{\partial{\mathbf{Q}}}{\partial t} = \sum_{\xi = x, y, z}\mathbf{A}_{\xi}\frac{\partial{\mathbf{Q}}}{\partial \xi}, 
  \end{split}
  \end{equation}
}
subject to the initial condition
\begin{align}\label{eq:initial_condition}
 \mathbf{Q}(x,y,z,0) = \mathbf{Q}_0(x,y,z) \in \mathrm{L}^2\left(\Omega\right).
\end{align}
The coefficient matrices are symmetric, $\mathbf{A}_{\xi} = \mathbf{A}_{\xi}^T$ and $\mathbf{P}= \mathbf{P}^T$,  with $ \mathbf{Q}^T\mathbf{P}\mathbf{Q} > 0$. In general the matrix $\mathbf{P}$ depends on the spatial coordinates $x, y,z$, and encodes the material parameters of the underlying medium. The constant coefficients and the non-dimensional  matrices $\mathbf{A}_{\xi}$ encapsulate the underlying linear conservation law and the corresponding linear constitutive relation. 

To describe wave propagation in elastic solids, we introduce the unknown wave fields 
\begin{align}\label{eq:velocity_stress}
\mathbf{Q}\left(x,y,z,t\right) = \begin{bmatrix}
\mathbf{v}(x,y, z,t) \\
\boldsymbol{\sigma}(x,y,z,t)
 \end{bmatrix},
\end{align}
with the particle velocity vector, $\mathbf{v} = \left[ v_x, v_y, v_z \right]^T$, and the stress vector, \\
$\boldsymbol{\sigma} = \left[ \sigma_{xx}, \sigma_{yy}, \sigma_{zz}, \sigma_{xy},  \sigma_{xz},  \sigma_{yz}\right]^T$.  
The symmetric constant coefficient matrices $\mathbf{A}_{\xi} $ describing the conservation of momentum and the constitutive relation, defined by Hooke's law, are given by
{

\begin{align}\label{eq:elastic_coeff}
\mathbf{A}_{\xi} = 
\begin{pmatrix}
\mathbf{0}_3 & \mathbf{a}_{\xi}\\
\mathbf{a}_{\xi}^T & \mathbf{0}_6
\end{pmatrix},
\quad
\mathbf{a}_x = 
\begin{pmatrix}
1& 0& 0& 0&0& 0\\
0& 0& 0& 1&0& 0\\
0& 0& 0& 0&1& 0
\end{pmatrix},
\quad
\mathbf{a}_y = 
\begin{pmatrix}
0& 0& 0& 1&0& 0\\
1& 0& 0& 0&0& 0\\
0& 0& 0& 0&0& 1
\end{pmatrix},
\quad
\mathbf{a}_z = 
\begin{pmatrix}
0& 0& 0& 0&1& 0\\
0& 0& 0& 0&0& 1\\
0& 0& 1& 0&0& 0\\
\end{pmatrix},
\end{align}
}
where $\mathbf{0}_3$ and $\mathbf{0}_6$ are  the $3$-by-$3$ and $6$-by-$6$ zero  matrices.

The symmetric positive definite material parameter matrix $\mathbf{P}$  is defined by
{
%
\begin{align}\label{eq:material_coeff}
\mathbf{P} = 
\begin{pmatrix}
\rho^{-1} \mathbf{1}  & \mathbf{0}\\
 \mathbf{0}^T  & \mathbf{C}
\end{pmatrix}
 ,
  \quad
  \mathbf{1} =   \begin{pmatrix}
  1 & 0 & 0 \\
  0 & 1 & 0 \\
  0 & 0 & 1 
  \end{pmatrix}
,
  \quad
  \mathbf{0} =   \begin{pmatrix}
  0 & 0 & 0& 0 & 0 & 0 \\
  0 & 0 & 0& 0 & 0 & 0 \\
  0 & 0 & 0& 0 & 0 & 0 
  \end{pmatrix}
  ,
\quad
\mathbf{C} =   \begin{pmatrix}
  c_{11} & c_{12} & c_{13} & c_{14} & c_{15} & c_{16}\\
  c_{12} & c_{22} & c_{23}& c_{24} & c_{25} & c_{26}\\
  c_{13} & c_{23} & c_{33}& c_{34} & c_{35} & c_{36}\\
  c_{14} & c_{24} & c_{34}& c_{44} & c_{45} & c_{46}\\
  c_{15} & c_{25} & c_{35}& c_{45} & c_{55} & c_{56}\\
  c_{16} & c_{26} & c_{36}& c_{46} & c_{56} & c_{66}
  \end{pmatrix},
\end{align}
}
where  $\rho(x,y,z) > 0$ is the  density of the medium, and  $\mathbf{C} = \mathbf{C}^{T} > 0$ is the symmetric positive definite  matrix of elastic constants.
Thus in \eqref{eq:linear_wave}, the first three equations are the conservation of momentum and the last six equations are the time derivatives of the constitutive relation, defined by Hooke's law, relating the  stress field to strains where the constant of proportionality is the stiffness matrix of elastic coefficients $\mathbf{C}$.

In a general anisotropic medium the stiffness matrix $\mathbf{C}$ is described by 21 independent elastic coefficients. 
However,  in the isotropic case, the medium is described by two independent elastic coefficients, the Lam\'e parameters $\mu > 0$, $\lambda > -\mu $.  We have
\begin{align}\label{eq:isotropic_stiffness_tensor}
\mathbf{C} =   \begin{pmatrix}
  2\mu + \lambda &  \lambda &  \lambda & 0 & 0 & 0\\
   \lambda & 2\mu + \lambda &  \lambda & 0  & 0 & 0\\
   \lambda &  \lambda & 2\mu + \lambda & 0 &  0 & 0\\
  0 & 0 & 0& \mu  & 0 & 0 \\
  0 & 0 & 0 & 0 & \mu & 0\\
  0 & 0 & 0 & 0 & 0 & \mu 
  \end{pmatrix}.
\end{align}
Here, $\lambda, \mu$ are the Lam\'e parameters, with $\mu > 0$ and $-\mu \le \lambda < \infty$. 
The elastic wave equation supports two families of wave, the P-wave and the S-wave, with wave speeds defined by
{
\begin{align}\label{eq:wave speed}
 c_p = \sqrt{\frac{2\mu + \lambda}{\rho}},\text{ and } \quad c_s = \sqrt{\frac{\mu }{\rho}}.
\end{align}
}
To analyse the effect of boundaries and boundary conditions we need to consider the eigenstructure 
of $\widetilde{\mathbf{A}}_\xi = {\mathbf{P}}{\mathbf{A}}_\xi$. For each $\xi=x,y,z$ there are 9 linearly independent eigenvectors and 6 
nontrivial eigenvalues, $\pm c_p,\pm c_s$, where the latter are double.
For later use we denote these eigenvalues by  $\pm c_\eta, \,\eta = x, y, z$, with
{
\begin{align}\label{eq:wave speed}
c_{\eta} = c_p~\text{if}~\eta = \xi,~\text{and} \quad c_{\eta} = c_s~\text{if}~\eta \ne \xi .
\end{align}
}

For real functions, we introduce the weighted $L^2$ inner product
\begin{align}\label{eq:weighted_scalar_product}
\left(\mathbf{Q}, \mathbf{F}\right)_P = \int_{\Omega}{\frac{1}{2}[\mathbf{Q}^T\mathbf{P}^{-1}\mathbf{F}] dxdydz},
\end{align}
and the corresponding norm
\begin{align}\label{eq:physical_energy}
\|\mathbf{Q}\left(\cdot, \cdot, \cdot, t\right)\|_P^2 = \left(\mathbf{Q}, \mathbf{Q}\right)_P = \int_{\Omega}\left(\sum_{\eta = x, y, z} \frac{\rho}{2} v_\eta^2  + \frac{1}{2}\boldsymbol{\sigma}^T{S}\boldsymbol{\sigma}\right)dxdydz. 
\end{align}
The weighted $L^2$-norm $\|\mathbf{Q}\left(\cdot, \cdot, \cdot, t\right)\|_P^2$ is the mechanical energy, which is the sum of the kinetic energy and the strain energy.

To begin with we consider the Cauchy problem with $\Omega=R^3$, and decay condition $|\mathbf{Q}| \to 0$ at $|(x,y,z)| \to \infty$.
%
%
To show that the  problem   is stable we multiply   \eqref{eq:linear_wave}  with $\boldsymbol{\phi}^T(x,y,z) $ 
from the left, where $\boldsymbol{\phi}(x,y,z) \in L^2\left(\Omega\right)$ is an arbitrary test function, and integrate over the whole spatial domain, $\Omega$, obtaining
{
\begin{align}\label{eq:product_1}
\int_{\Omega}\boldsymbol{\phi}^T\mathbf{P}^{-1}\frac{\partial \mathbf{Q}}{\partial t}dxdydz &= \int_{\Omega}\boldsymbol{\phi}^T\left(\sum_{\xi = x, y, z}\mathbf{A}_{\xi}\frac{\partial{\mathbf{Q}}}{\partial \xi}\right)dxdydz.
\end{align}
}
In the right hand side of \eqref{eq:product_1}, integrating-by-parts, and using the fact that the 
coefficient matrices are constant and symmetric, $\mathbf{A}_{\xi}  = \mathbf{A}_{\xi}^T$, gives
{
\begin{align}\label{eq:product_2}
\int_{\Omega}\boldsymbol{\phi}^T\mathbf{P}^{-1}\frac{\partial \mathbf{Q}}{\partial t}dxdydz  = \frac{1}{2}\int_{\Omega}\sum_{\xi = x,y,z} \left[\boldsymbol{\phi}^T\mathbf{A}_{\xi}\frac{\partial{\mathbf{Q}}}{\partial \xi} -\mathbf{Q}^T\mathbf{A}_{\xi}\frac{\partial{\boldsymbol{\phi}}}{\partial \xi}\right] dxdydz +  \frac{1}{2}\oint_{\partial \Omega}\boldsymbol{\phi}^T \left(\sum_{\xi = x, y, z}n_{\xi}\mathbf{A}_{\xi}\right)\mathbf{Q}  dS.
\end{align}
}
By the decay conditions there are no boundary terms.
Replacing $\boldsymbol{\phi}$ with $\mathbf{Q}$ in \eqref{eq:product_2}, in the right hand side  the volume terms vanish, having
{
\begin{align}\label{eq:product_3}
\int_{\Omega}\mathbf{Q}^T\mathbf{P}^{-1}\frac{\partial{\mathbf{Q}}}{\partial t}dxdydz =  0.
\end{align}
}
The energy equation follows
{
\begin{align}\label{eq:cauchy_estimate}
\frac{d}{dt}\|\mathbf{Q}\left(\cdot, \cdot, \cdot, t\right)\|_P^2 = 0, 
\end{align}
}
and thus the energy is conserved, $\|\mathbf{Q}\left(\cdot, \cdot, \cdot, t\right)\|_P^2 = \|\mathbf{Q}\left(\cdot, \cdot, \cdot, 0\right)\|_P^2$, for all $t\ge 0$. 
This analysis indicates that the Cauchy problem for \eqref{eq:linear_wave} is  well-posed and asymptotically stable. 
\begin{remark}
Depending on the coefficient matrices $\mathbf{P}, \mathbf{A}_{\xi}$ the  system \eqref{eq:linear_wave} can also describe acoustic waves, electromagnetic waves or elastic waves propagating in a heterogeneous medium. 
\end{remark}
\subsection{Boundary conditions}
Consider the 3D cuboidal domain
{
\begin{equation}\label{eq:physical_domain}
\Omega = \{(x,y,z): -1 \le x \le 1, \quad -1 \le y\le 1, \quad -1 \le z\le 1\}.
\end{equation}
}
 Stable and well-posed boundary conditions are needed to close the rectangular surfaces  of the boundaries of the cuboidal domain.
Boundary conditions are enforced by modifying the amplitude of the incoming characteristics \cite{GustafssonKreissOliger1995,DuruGabrielIgel2017,Duru_exhype_2_2019}. 
To analyse the effect of boundaries and boundary conditions   we need for each $\xi=x,y,z$ the traction vector at boundaries with
 $\xi = -1$, or $\xi =  1$. The traction there is
\begin{align}\label{eq:velocity_tractions}
%
\mathbf{T} =  \boldsymbol{a}_{\xi}\boldsymbol{\sigma}.
\end{align}
Note that 
\begin{align}
\mathbf{A}_{\xi}\mathbf{Q}  = \begin{pmatrix}
\boldsymbol{a}_{\xi}\boldsymbol{\sigma} \\
\boldsymbol{a}_{\xi}^T \mathbf{v} \\
\end{pmatrix},
~\
\text{and}
~\
\mathbf{Q} ^T\mathbf{A}_{\xi}\mathbf{Q}  =   \mathbf{v}^T\mathbf{T} + \mathbf{T}^T\mathbf{v} = 2\mathbf{v}^T\mathbf{T}.
\end{align}

In this case the energy method yields \eqref{eq:product_3} with a boundary term in the right hand side.
If we introduce the reference boundary surface
 \begin{align}
 \widetilde{\Gamma} = [-1, 1]\times[-1, 1].
 \end{align}
The boundary term is given by
\begin{align}\label{eq:product_BT}
  \mathrm{BT}\left(\mathbf{v}, \mathbf{T}\right) &\equiv \frac{1}{2}\oint_{\partial \Omega} \left[\mathbf{Q}^T \left(\sum_{\xi = x, y, z}n_{\xi}\mathbf{A}_{\xi}\right)\mathbf{Q}\right] dS \nonumber \\
   &= \sum_{\xi = x, y, z}\int_{\widetilde{\Gamma}}  \frac{1}{2}\left[ \mathbf{Q}^T\mathbf{A}_{\xi}\mathbf{Q}  \Big|_{\xi =1}- \mathbf{Q}^T\mathbf{A}_{\xi}\mathbf{Q}  \Big|_{\xi =-1}  \right]  \frac{dxdydz}{d\xi}\nonumber \\
 &= \sum_{\xi = x, y, z}\int_{\widetilde{\Gamma}}\left[ \mathbf{v}^T\mathbf{T}  \Big|_{\xi =1}- \mathbf{v}^T\mathbf{T}  \Big|_{\xi = -1}  \right]  \frac{dxdydz}{d\xi}.
\end{align}
Well-posed boundary conditions are needed to close the edges of the rectangular domain such that the boundary term   $\mathrm{BT}\left(\mathbf{v}, \mathbf{T}\right)$ defined in \eqref{eq:product_BT} is negative semi-definite.

For each $\xi=x,y,z$ there are 3 ingoing characteristic variables and 3 outgoing characteristic variables at each boundary $\xi=\pm 1$, corresponding to the nontrivial eigenvalues 
of $\tilde A_\xi$. These characteristic variables are
 \begin{align}\label{eq:characteristics}
{q}_\eta  = \frac{1}{2}\left({Z}_\eta {v}_\eta  + {T}_\eta \right), \quad {p}_\eta  = \frac{1}{2}\left({Z}_\eta {v}_\eta  - {T}_\eta \right), \quad Z_\eta  > 0, \quad \eta = x, y, z,
\end{align}
where we have introduced the impedance, $Z_{\eta} = \rho c_{\eta}$, and $c_\eta$ are the wave speeds defined in \eqref{eq:wave speed}.
At the boundary $\xi = 1$ $(\xi = -1)$,  ${q}_\eta $ (${p}_\eta $)  are the characteristic variables  going into the domain  and  ${p_\eta }$ (${q_\eta }$) the  
characteristic variables going out of the domain.
 For each $\xi = x, y, z$, we consider linear boundary conditions, 
{
\begin{equation}\label{eq:BC_General2}
\begin{split}
&{q}_\eta  - {\gamma_\eta }{p}_\eta  = 0 \iff \frac{Z_{\eta}}{2}\left({1-\gamma_\eta }\right){v}_\eta  -\frac{1+\gamma_\eta }{2} {T}_\eta  = 0,  \quad \xi = -1, \\
&{p}_\eta  - {\gamma_\eta }{q}_\eta  = 0 \iff \frac{Z_{\eta}}{2} \left({1-\gamma_\eta }\right){v}_\eta  + \frac{1+\gamma_\eta }{2}{T}_\eta  = 0,  \quad  \xi = 1,
 \end{split}
\end{equation}
}
where the reflection coefficients $\gamma_\eta $ are real numbers with $ 0 \le |\gamma_\eta |\le 1$. 
 Note that these boundary conditions do not allow the different families of characteristic variables to mix.
The boundary conditions \eqref{eq:BC_General2} specify the ingoing characteristics on the boundary in terms of the outgoing characteristics. 
The boundary condition \eqref{eq:BC_General2}, can describe several physical situations.
 We have a free-surface boundary condition  if $\gamma_\eta  = 1$, an absorbing boundary condition  if $\gamma_\eta  = 0$ and a clamped boundary condition  if $\gamma_\eta  = -1$.
We note that
  \begin{align}\label{eq:simplify_3}
& \text{at} \quad \xi = -1, \quad v_\eta T_\eta  =  \frac{Z_\eta \left(1-\gamma_\eta \right)}{\left(1+\gamma_\eta \right)}v_{\eta}^2 = \frac{\left(1+\gamma_\eta \right)}{Z_\eta \left(1-\gamma_\eta \right)}T_{\eta}^2  >0, \quad  |\gamma_\eta | < 1, \quad  v_\eta T_\eta  = 0, \quad  |\gamma_\eta | = 1, \nonumber
 \\
 & \text{at} \quad \xi = 1, \quad  v_\eta T_\eta  = - \frac{Z_\eta \left(1-\gamma_\eta \right)}{\left(1+\gamma_\eta \right)}v_{\eta}^2 = - \frac{\left(1+\gamma_\eta \right)}{Z_\eta \left(1-\gamma_\eta \right)}T_{\eta}^2  < 0, \quad  |\gamma_\eta | < 1, \quad  v_\eta T_\eta  = 0, \quad  |\gamma_\eta | = 1.
  \end{align}
If all reflection coefficients at all boundaries satisfy $|\gamma_\eta |\le 1$ then the boundary term defined in \eqref{eq:product_BT} is negative semi-definite, $\mathrm{BT}\le 0$,
and  we have
\begin{align}\label{eq:energy_conservation}
\frac{d}{dt}\|\mathbf{Q}\left(\cdot, \cdot, \cdot, t\right)\|_P^2 = \mathrm{BT} \le 0,
\end{align}
where the energy $\|\mathbf{Q}\left(\cdot, \cdot, \cdot, t\right)\|_P^2$ is defined by  \eqref{eq:physical_energy}.
\subsection{Interface conditions}
We introduce physical interface conditions  we will use to couple local  DG elements to the global domain.
These physical interface conditions will connect two adjacent elements.
Consider the Cauchy prolem with a planar interface at $x = 0$, and  denote the corresponding fields and material parameters in the positive/negative sides of the interface with the superscripts $+/-$.
We define the jumps in scalar or vector valued fields by $\lJump{{a} \rJump} =  a^{+} - a^{-}$.
The interface conditions that will connect two adjacent elements are force balance, and vanishing opening and slip velocities
\begin{equation}\label{eq:elastic_elastic_interface}
 {\mathbf{T}}^{+}  = {\mathbf{T}}^{-} = {\mathbf{T}}, \quad \lJump {\mathbf{v}} \rJump = 0.
\end{equation}


As before, using the energy method we have
 \begin{align}\label{eq:energy_estimate_fault_30}
\frac{d}{dt}\left(\|\mathbf{Q}^{-}\left(\cdot, \cdot, \cdot, t\right)\|_P^2+ \|\mathbf{Q}^{+}\left(\cdot, \cdot, \cdot, t\right)\|_P^2\right)   =   \mathrm{IT}\left({\mathbf{T}}, \lJump {\mathbf{v}} \rJump \right).
 \end{align}
Here the interface term is defined by
\begin{align}\label{eq:interface_term}
  \mathrm{IT}\left({\mathbf{T}}, \lJump {\mathbf{v}} \rJump\right) = -\int_{\widetilde{\Gamma}} \lJump {\mathbf{v}} \rJump ^T {\mathbf{T}} {dydz}.
\end{align}
Note that by \eqref{eq:elastic_elastic_interface} the interface term vanishes identically and the sum of the energy is conserved.

Recently in \cite{DuruGabrielIgel2017,Duru_exhype_2_2019}, we developed new physics based numerical flux suitable for patching DG elements together.   A key step taken in the construction of the physics based numerical flux is to reformulate the boundary condition  \eqref{eq:BC_General2} and  interface condition \eqref{eq:elastic_elastic_interface} by introducing transformed (hat-) variables, $\left(\widehat{\mathbf{T}}, \widehat{\mathbf{v}}\right)$,  so that we can simultaneously construct (numerical) boundary/interface data for the particle velocity vector $\mathbf{v}$ and the traction vector $\mathbf{T}$.   The hat-variables encode the solution of the IBVP on the boundary/interface. To be more specific, the hat-variables are solutions of the Riemann problem constrained against physical boundary/interface conditions  \eqref{eq:BC_General2} and \eqref{eq:elastic_elastic_interface}. Since the hat-variables are constructed to satisfy the boundary/interface conditions,  \eqref{eq:BC_General2}  and  \eqref{eq:elastic_elastic_interface} exactly, we must have
{
\begin{align}\label{eq:identity_bc}
\mathrm{BT}(\widehat{\mathbf{v}}, \widehat{\mathbf{T}})  \le 0,
\end{align}
\begin{align}\label{eq:identity_interface}
 \mathrm{IT}\left(\widehat{\mathbf{T}}, \lJump \widehat{\mathbf{v}}  \rJump \right)  = 0.
\end{align}
}
 The indentities \eqref{eq:identity_bc}-\eqref{eq:identity_interface}, will be used in proving numerical stability. We refer the reader to \cite{DuruGabrielIgel2017,Duru_exhype_2_2019} for more elaborate discussions.

\section{PML for first order wave equations}
Here, we will use the well known complex coordinate stretching technique \cite{Chew1994}, to construct a modal PML,  \cite{KDuru2016,DuruKozdonKreiss2016,AppeloKreiss2006} for the  system \eqref{eq:linear_wave}. As above, we consider the cuboidal domain \eqref{eq:physical_domain}.
To begin with, let the Laplace transform, in time, of $\mathbf{Q}\left(x,y, z, t\right)$ be defined by
{
\small
\begin{equation}
\widetilde{\mathbf{Q}}(x,y,z,s)  = \int_0^{\infty}e^{-st}{\mathbf{Q}}\left(x,y,z,t\right)\text{dt},  \quad s = a + ib, \quad \Re{s} = a > 0.
\end{equation}
}
We consider  a setup where the PML is included in all spatial coordinates.  Take the Laplace transform,  in time,  of    equation  \eqref{eq:linear_wave}. 
The PML can be constructed  in each coordinate,  $\xi = x, y, z $, using 
$\partial/\partial{\xi} \to 1/S_{\xi} \partial/\partial{\xi} $. 
Here
{
\small
\begin{align}\label{eq:PML_metric}
S_{\xi} = 1 +\frac{d_{\xi}\left(\xi\right)}{s + \alpha_{\xi}\left(\xi\right)},
\end{align}
}
 are the complex PML metrics, with $s$ denoting the Laplace dual time variable,  $d_{\xi} \ge 0$  are  the PML damping functions and  $\alpha_{\xi} \ge 0$ are the complex frequency shift (CFS) \cite{Kuzuoglu96}.
 
 Now, take the Laplace transform of \eqref{eq:linear_wave} in time, and introduce the complex change of variable $\partial/\partial{\xi} \to 1/S_{\xi} \partial/\partial{\xi} $. We have the PML in the Laplace space
 {
\begin{equation}\label{eq:linear_wave_PML_Laplace}
\begin{split}
\mathbf{P}^{-1} s \widetilde{\mathbf{Q}}(x,y,z,s) = \sum_{\xi = x, y, z}\mathbf{A}_{\xi}\frac{1}{S_{\xi} }\frac{\partial{\widetilde{\mathbf{Q}}(x,y,z,s)}}{\partial \xi}, 
  \end{split}
  \end{equation}
}
In order to localize the PML in time, we introduce the auxiliary variable
\begin{equation}\label{eq:auxiliary_PML_Laplace}
\begin{split}
 \widetilde{\mathbf{w}}_{\xi}(x,y,z,s) = \frac{1}{(s + \alpha_\xi)S_{\xi} }\mathbf{A}_{\xi}\frac{\partial{\widetilde{\mathbf{Q}}(x,y,z,s)}}{\partial \xi}, 
  \end{split}
  \end{equation}
  and invert the Laplace transform,
 having the time-dependent PML 
 {
\begin{equation}\label{eq:elastic_pml_1A}
\begin{split}
\mathbf{P}^{-1}\frac{\partial{\mathbf{Q}}}{\partial t} = \sum_{\xi = x, y, z}\left[\mathbf{A}_{\xi}\frac{\partial{\mathbf{Q}}}{\partial \xi}  - d_{\xi}\left(\xi\right)\mathbf{w}_{\xi}\right],
  \end{split}
  \end{equation}
\begin{equation}\label{eq:elastic_pml_2A}
\begin{split}
\frac{\partial{\mathbf{w}_{\xi}}}{\partial t} 
 = \mathbf{A}_{\xi}\frac{\partial{\mathbf{Q}}}{\partial \xi} -   \left(\alpha_{\xi}(\xi) + d_{\xi}\left(\xi\right)\right)\mathbf{w}_{\xi}.
  \end{split}
  \end{equation}
}
Since, we have used the well known PML metric \eqref{eq:PML_metric}, the PML \eqref{eq:elastic_pml_1A}-\eqref{eq:elastic_pml_2A} can be shown to be  analogous to other  PML models such as  \cite{KDuru2016,DuruKozdonKreiss2016,AppeloKreiss2006,Be_etAl}.     We will initialize the PML with zero initial data and  terminate the PML  \eqref{eq:elastic_pml_1A}-\eqref{eq:elastic_pml_2A}  with the boundary conditions  \eqref{eq:BC_General2}.
Note  that the PML absorption functions  $d_{\xi} \ge 0$ and auxiliary  functions $\mathbf{w}_{\xi}$ vanish almost everywhere except in the layers defining the PML.
For example, inside  the  PML strip truncating the  computational domain in the $x$-direction  only   the auxiliary variable $\mathbf{w}_{x}$ and the PML damping function $d_x\left(x\right) \ge 0$   are nonzero, that is  $\mathbf{w}_{\xi} =0$ and $d_{\xi}\left(\xi\right) =0$ for $\xi \ne x$.  There are  domain edges where one damping function vanishes and two damping functions are simultaneously active, for example in the $xy-$edge we have $d_z(z) = 0$ and  $d_{\xi}\left(\xi\right) \ge 0$ for $\xi = x, y$. There also   corner regions where all damping functions are simultaneously nonzero, $d_{\xi}\left(\xi\right) >0$.

 Without damping, $d_{\xi}\left(\xi\right)\equiv 0$,  we recover the  elastic wave equation  \eqref{eq:linear_wave},  which satisfies the energy estimate  \eqref{eq:energy_conservation}.
However, the energy estimate is not  applicable to the PML, \eqref{eq:elastic_pml_1A}-\eqref{eq:elastic_pml_2A} when $d_{\xi}\left(\xi\right) \ge 0$, for any $\eta = x, y, z$.  

In the coming section we will discuss the temporal  stability of the PML and derive energy estimates in the Laplace space.

\section{Energy estimates for the PML in the Laplace space}
The stability analysis of the PML has attracted considerable attention in the literature.  The usual analysis, see \cite{Duruthesis2012,Be_etAl,SkAdCr,AppeloKreiss2006,DuruGabKreiss2019},  focuses on the constant coefficient problem, where, all variable coefficients are frozen and the problem  is  subdivided into simpler  sub-problems, such as: a) the PML strip problem, where only one PML damping is nonzero,  eg. the $x$-dependent PML strip with  $d_x(x) > 0$ and $d_{\xi}\left(\xi\right) =0$ for $\xi \ne x$; b) the PML edge problem where one damping function vanishes and two damping functions are simultaneously active, eg. in the $xy-$PML-edge we have $d_z(z) = 0$ and  $d_{\xi}\left(\xi\right) \ge 0$ for $\xi = x, y$; c) a PML corner region with $d_{\xi}\left(\xi\right) >0$, for all $\xi = x, y, z$. 
The stability  analysis of each Cauchy PML layer can be performed  using classical Fourier methods. The analysis has been extended to accommodate boundary conditions, \cite{Duruthesis2012} using normal mode analysis for the IBVP.  The results of this theory are documented in \cite{Duruthesis2012,DuKrSIAM,KDuru2016,DuruKozdonKreiss2016}. The central result is: as long as the underlying undamped IBVP (without the PML) is stable, the PML IBVP is stable if there are no modes with phase and group velocities pointing in opposite direction.  Since, we have used the well known PML metric \eqref{eq:PML_metric}, this result is valid for the  PML \eqref{eq:elastic_pml_1A}-\eqref{eq:elastic_pml_2A}. We will not repeat this analysis here, for more elaborate discussion we refer the reader to above references.

Using the modal analysis outlined above, the continuous PML IBVP,  \eqref{eq:elastic_pml_1A}-\eqref{eq:elastic_pml_2A} with \eqref{eq:BC_General2},  can be proven stable. However,  extending this modal analysis to the discrete setting is possible, but non-trivial even in the most simplified cases. Therefore, even when the continuous PML IBVP can proven stable, numerical instabilities persist when PML are used in computations, particularly when boundary  and interface  conditions are included. A numerical method which is provably stable in the absence of the PML can become unstable when the PML is included \cite{ChongXiaMillerTso2011,Tago_etal,KDuru2016,DuruKozdonKreiss2016}. When a continuous PML is stable, the  primary difficulty now is how to ensure numerical stability when  the PML is discretized.  A main challenge is that the PML is generally asymmetric.  Therefore, classical energy norms that are traditionally used in developing stable DG, finite difference or finite element methods for hyperbolic PDEs  are not  sufficient when the PML is present.  

In this study we will take a more pragmatic approach. We will derive energy estimates in the Laplace space for some simplified  model problems, for a)   PML strip problem,  b)  PML edge problem, and c)  PML corner problem.  These energy estimates can be used to derive stable numerical methods for the PML IBVP,  \eqref{eq:elastic_pml_1A}-\eqref{eq:elastic_pml_2A} with \eqref{eq:BC_General2} and \eqref{eq:elastic_elastic_interface}. If the discrete approximation of the PML is not provably stable for these model problems, there is no hope that the approximation will be stable for the general PML problem.


As in \cite{DuruGabKreiss2019} we will reformulate the IBVP, \eqref{eq:elastic_pml_1A}-\eqref{eq:elastic_pml_2A} with \eqref{eq:BC_General2} and \eqref{eq:elastic_elastic_interface},  by introducing new variables by subtracting the product of the initial function and $e^{-t}$ from each of the unknown functions. The new unknown functions satisfy the same system as the original unknowns, \eqref{eq:elastic_pml_1A}-\eqref{eq:elastic_pml_2A}, but with zero initial data and nontrivial source functions in the equations. Denote the source function by $\mathbf{f}(x,y,z, t) = \left(\mathbf{F}_{Q}\left(x,y,z, t\right),  \mathbf{F}_{w_\xi}\left(x,y,z, t\right)\right)^T$,  and note that all these components decay exponentially in time, and depend on the initial data and derivatives thereof. 


Laplace transformation in time  of the PML equations \eqref{eq:elastic_pml_1A}-\eqref{eq:elastic_pml_2A},  yields

 {
\begin{equation}\label{eq:elastic_pml_1A_Laplace}
\begin{split}
\mathbf{P}^{-1} s\widetilde{\mathbf{Q}} = \sum_{\xi = x, y, z}\left[\mathbf{A}_{\xi}\frac{\partial{\widetilde{\mathbf{Q}}}}{\partial \xi}  - d_{\xi}\left(\xi\right)\widetilde{\mathbf{w}}_{\xi}\right] + \widetilde{\mathbf{F}}_{Q},
  \end{split}
  \end{equation}
\begin{equation}\label{eq:elastic_pml_2A_Laplace}
\begin{split}
s\widetilde{\mathbf{w}}_{\xi}
 = \mathbf{A}_{\xi}\frac{\partial{\widetilde{\mathbf{Q}}}}{\partial \xi} -   \left(\alpha_{\xi}\left(\xi\right) + d_{\xi}\left(\xi\right)\right)\widetilde{\mathbf{w}}_{\xi} +  \widetilde{\mathbf{F}}_{w_\xi},
  \end{split}
  \end{equation}
}
 Transforming the boundary conditions \eqref{eq:BC_General2} and the interface conditions  \eqref{eq:elastic_elastic_interface}  gives
{
\begin{equation}\label{eq:BC_General2_Laplace}
\begin{split}
& \frac{Z_{\eta}}{2}\left({1-\gamma_\eta }\right)\widetilde{v}_\eta  -\frac{1+\gamma_\eta }{2} \widetilde{T}_\eta  = 0,  \quad \xi = -1, \\
& \frac{Z_{\eta}}{2} \left({1-\gamma_\eta }\right)\widetilde{v}_\eta  + \frac{1+\gamma_\eta }{2}\widetilde{T}_\eta  = 0,  \quad  \xi = 1,
 \end{split}
\end{equation}
}
\begin{equation}\label{eq:elastic_elastic_interface_Laplace}
  \widetilde{\mathbf{T} }^{+}  =  \widetilde{\mathbf{T} }^{-} =  \widetilde{\mathbf{T} }, \quad \lJump{ \widetilde{\mathbf{v}}  \rJump} = 0.
\end{equation}
Here   $\widetilde{\mathbf{T} }$ is the traction vector at  the interface at hand, with components $\widetilde{T}_\eta$. Note that \eqref{eq:BC_General2_Laplace} is defined for $\eta=x,y,z$.
The boundary conditions \eqref{eq:BC_General2_Laplace} satisfy
 {
  \begin{align}\label{eq:simplify_3_Laplace}
   \widetilde{\mathbf{T} }^\dagger \widetilde{\mathbf{v} } =  \widetilde{\mathbf{v} }^\dagger \widetilde{\mathbf{T} },
   \quad 
   \widetilde{\mathbf{v} }^\dagger \widetilde{\mathbf{T} }\Big|_{\xi = -1}  \ge 0, 
   \quad
 \widetilde{\mathbf{v} }^\dagger \widetilde{\mathbf{T} }\Big|_{\xi = 1}   \le 0,
   \quad
   \forall |\gamma_\eta | \le 1,
  \end{align}
  }
  and the interface condition \eqref{eq:elastic_elastic_interface_Laplace} gives
   {
  \begin{align}\label{eq:simplify_4_Laplace}
   \widetilde{\mathbf{T} }^\dagger \lJump{\widetilde{\mathbf{v}} \rJump} =   \lJump{ \widetilde{\mathbf{v}}  \rJump}^\dagger  \widetilde{\mathbf{T} }  = 0.
  \end{align}
  }
Here, $\widetilde{\mathbf{T}}^{\dagger}$ denotes the complex conjugate transpose of $\widetilde{\mathbf{T}}$. 
From \eqref{eq:simplify_3_Laplace}, note the negative semi-definiteness of the boundary terms
\begin{align}\label{eq:BT_Negative}
 \frac{1}{2}\left[\widetilde{\mathbf{Q}}^{\dagger}\mathbf{A}_{\xi} \widetilde{\mathbf{Q}}\right]\Big|_{\xi =-1}^{\xi = 1} =     \widetilde{\mathbf{v} }^\dagger \widetilde{\mathbf{T} }\Big|_{\xi = 1}-  \widetilde{\mathbf{v} }^\dagger \widetilde{\mathbf{T} }\Big|_{\xi = -1} \le 0, \quad \forall |\gamma_{\eta}| \le 1.
\end{align}

Next, we use \eqref{eq:elastic_pml_2A_Laplace} and eliminate the auxiliary variables $\widetilde{\mathbf{w}}_{\xi}$ in \eqref{eq:elastic_pml_1A_Laplace}.  We obtain
\begin{equation}\label{eq:elastic_pml_3D_Laplace}
\begin{split}
\mathbf{P}^{-1} \left(sS_x\right)\widetilde{\mathbf{Q}} =   \sum_{\xi = x, y, z}\frac{S_x}{S_{\xi}} \mathbf{A}_{\xi}\frac{\partial{\widetilde{\mathbf{Q}} }}{\partial \xi}  + \mathbf{P}^{-1}\widetilde{\mathbf{F}},
  \end{split}
  \end{equation}
with  the boundary conditions \eqref{eq:BC_General2_Laplace} and interface condition \eqref{eq:elastic_elastic_interface_Laplace},  where
\begin{align}
\widetilde{\mathbf{F}} =     \mathbf{P}\left(S_x \widetilde{\mathbf{F}}_{Q}  - \sum_{\xi = x,y, z} \frac{d_{\xi}S_x }{s + \alpha_{\xi} + d_{\xi}} \widetilde{\mathbf{F}}_{w_\xi}\right).
\end{align}
%
For Laplace transformed systems, it is necessary to extend the weighted scalar product \eqref{eq:weighted_scalar_product} and the norm \eqref{eq:physical_energy} to complex functions.
 For complex functions, we define the weighted scalar product
 \begin{align}\label{eq:elastic_scalarproduct_laplacePML}
\left(\widetilde{\mathbf{Q}}, \widetilde{\mathbf{F}}\right)_{{P}} =  \int_{\Omega} \frac{1}{2}\left[\widetilde{\mathbf{Q}}^{\dagger}\mathbf{P}^{-1} \widetilde{\mathbf{F}}\right]   dxdydz,
 \end{align}
 and the corresponding  norm
\begin{align}\label{eq:elastic_energy_laplacePML}
\|\widetilde{\mathbf{Q}} (\cdot ,\cdot ,\cdot , s)\|_{P}^2 = \left(\widetilde{\mathbf{Q}}, \widetilde{\mathbf{Q}}\right)_{{P}}.
 \end{align}
Again, $\widetilde{\mathbf{Q}}^{\dagger}$ denotes the complex conjugate transpose of $\widetilde{\mathbf{Q}}$. 

Let   $s = a + i b$ and consider
 the complex number
\begin{align}\label{eq:scaled_variables_2}
 S_{x} =  \frac{\alpha + s + d }{\alpha + s} = \frac{\alpha + a + d + ib}{ \alpha + a + ib }.
\end{align} 
Note that  for $\alpha \ge 0$, $d \ge 0$ we have
{
\begin{align}\label{eq:scaled_variables_3}
\Re{\left(sS_x\right)} =  a + \left( \frac{a\left(a+\alpha\right) + b^2}{{|s + \alpha|^2}}\right)d \ge a.
\end{align}
}
Thus for any  $a  > 0$, we have
\begin{align}
0 < \frac{1}{\Re{\left(sS_x\right)}} = \frac{1}{a + \left( \frac{a\left(a+\alpha\right) + b^2}{{|s + \alpha|^2}}\right)d } \le \frac{1}{a}.
\end{align}

\subsection{The PML strip problem}
We will now focus on the $x$-dependent  PML strip problem, that is  \eqref{eq:elastic_pml_1A}-\eqref{eq:elastic_pml_2A}  with  $d_x \ge  0$,  and $\partial /\partial \xi =  0$, $d_\xi = 0$ for $\xi = y, z$.
This simplification results in the 1D PML problem
\begin{equation}\label{eq:weak_Laplace_1D}
\begin{split}
 \left(sS_x\right) \mathbf{P}^{-1} \widetilde{\mathbf{Q}} = {A}_x\frac{\partial{\widetilde{\mathbf{Q}} }}{\partial x} +  \mathbf{P}^{-1}\widetilde{\mathbf{F}}.
  \end{split}
  \end{equation}
  Equation \eqref{eq:weak_Laplace_1D} lives in a 3D domain define in \eqref{eq:physical_domain}, but have been simplified by restricting the initial data and the forcing $\widetilde{\mathbf{F}}$  to functions that vary only in 1D, the $x$-axis.
  We have
\begin{theorem}\label{theo:PML_1D_Strip_Laplace}
Consider the 1D   PML equation \eqref{eq:weak_Laplace_1D} in the Laplace space, with  piecewise constant $d_x(x) = d\ge 0$ and  $\alpha_x  = \alpha\ge 0$,  subject to the boundary conditions \eqref{eq:BC_General2_Laplace},  with $ |\gamma_\eta| \le 1$  at $x = \pm 1$, and the interface condition \eqref{eq:elastic_elastic_interface_Laplace} at discontinuities of $d_x(x)$. For all $s$ such that  $\Re(s) \ge a > 0$ we have
{
\begin{equation}\label{eq:energy_estimate_pml_laplace_strip}
\begin{split}
& \|\sqrt{\Re(s S_x)}\widetilde{\mathbf{Q}} (\cdot ,\cdot ,\cdot , s)\|_{P}^2 \le \|\widetilde{\mathbf{Q}} (\cdot ,\cdot ,\cdot , s)\|_{P} \|\widetilde{\mathbf{F}} (\cdot ,\cdot ,\cdot , s)\|_{P}^2 + \widetilde{\mathrm{BT}}, \quad  \Re(s S_x) \ge a > 0, \\
& \widetilde{\mathbf{F}} =     \mathbf{P}\left(S_x \widetilde{\mathbf{F}}_{Q}  -  \frac{d }{s + \alpha} \widetilde{\mathbf{F}}_{w_x}\right), \quad \widetilde{\mathrm{BT}} = \int_{\widetilde{\Gamma}}\frac{1}{2}\left[\widetilde{\mathbf{Q}}^\dagger{A}_x\widetilde{\mathbf{Q}}\right]\Big|_{x = -1}^{x =1}dydz  \le 0, \quad \forall |\gamma_\eta| \le 1.
\end{split}
\end{equation}
}
\end{theorem}
\begin{proof}
 It suffices to consider constant PML damping $d_{x} \ge 0$ coefficients in the entire domain. If $d_{x} = d \ge 0$ is piecewise constant we split the domain, and the integral, along the discontinuities,  with the jump condition \eqref{eq:elastic_elastic_interface_Laplace} the interface terms vanish.
  Multiply  \eqref{eq:weak_Laplace_1D} with $\widetilde{\mathbf{Q}}^{\dagger}$ from the left and integrate over the whole spatial domain. Integration-by-parts gives
  {
  \small
  \begin{equation}\label{eq:weak_Laplace_1D_1}
\begin{split}
\int_{\Omega}\left( \left(sS_x\right)\widetilde{\mathbf{Q}}^{\dagger} \mathbf{P}^{-1} \widetilde{\mathbf{Q}} \right)dxdy dz &= \frac{1}{2}\int_{\Omega}\left(\left[\widetilde{\mathbf{Q}}^\dagger\left(\mathbf{A}_{x}\frac{\partial{\widetilde{\mathbf{Q}}}}{\partial x}\right)-\left(\mathbf{A}_{x}\frac{\partial{\widetilde{\mathbf{Q}}}}{\partial x}\right)^\dagger \mathbf{Q}\right] \right)dxdydz +  \int_{\widetilde{\Gamma}}\left(\frac{1}{2}\left[\widetilde{\mathbf{Q}}^{\dagger}{A}_x \widetilde{\mathbf{Q}}\right] \Big|_{-1}^{1}\right)dydz \\
&+ \int_{\Omega}\left(\widetilde{\mathbf{Q}}^{\dagger} \mathbf{P}^{-1}\widetilde{\mathbf{F}} \right)dxdydz.
  \end{split}
  \end{equation}
  }
  Adding the product \eqref{eq:weak_Laplace_1D_1} to its  complex conjugate,  the spatial derivative vanishes, yielding
  {
  \begin{equation}\label{eq:weak_Laplace_1D_2}
\begin{split}
\|\sqrt{\Re(s S_x)}\widetilde{\mathbf{Q}} (\cdot ,\cdot ,\cdot , s)\|_{P}^2  = \frac{1}{2}\left[\left(\widetilde{\mathbf{Q}},\widetilde{\mathbf{F}} \right)_{{P}}  + \left(\widetilde{\mathbf{F}},\widetilde{\mathbf{Q}} \right)_{{P}}\right] + \frac{1}{2}\int_{\widetilde{\Gamma}}\left(\left[\widetilde{\mathbf{Q}}^{\dagger}{A}_x \widetilde{\mathbf{Q}}\right] \Big|_{-1}^{1}\right)dydz.
  \end{split}
  \end{equation}
  }
   From \eqref{eq:BT_Negative}, note that the boundary terms are never positive, $\widetilde{\mathrm{BT}} \le 0$.
  Using Cauchy-Schwarz inequality and the fact \eqref{eq:BT_Negative}  in the right hand side of \eqref{eq:weak_Laplace_1D_2} completes the proof.
 \end{proof}

\begin{remark}\label{remark:pml_strip}
It is important to note that   similar energy  estimates \eqref{eq:energy_estimate_pml_laplace} are also valid for the  PML strip problems in the $y$-axis or $z$-axis, that is  when $ d_y = d>0, d_x = d_z=0$ or  $ d_z = d>0, d_x = d_y=0$.
\end{remark}

\subsection{The PML  edge problem}
We consider specifically the  $xy$--edge PML  problem, that is \eqref{eq:elastic_pml_1A}-\eqref{eq:elastic_pml_2A} with $d_x  = d_y = d>0$ and $d_z = 0$, $\partial /\partial z =0$.  The analysis can be extended to $xz$--edge and $yz$--edge PML problems. 
If  $d_x  = d_y = d>0$ and $d_z = 0$, then we have $S_y = S_x$ and $S_z = 1$.
This simplification results in  the 2D $xy$-edge PML problem, 
\begin{equation}\label{eq:weak_Laplace_2D}
\begin{split}
 \left(sS_x\right) \mathbf{P}^{-1} \widetilde{\mathbf{Q}}  =  \sum_{\xi = x, y}\mathbf{A}_{\xi}\frac{\partial{\widetilde{\mathbf{Q}} }}{\partial \xi}  + \mathbf{P}^{-1}\widetilde{\mathbf{F}}.
  \end{split}
  \end{equation}
Note also that equation \eqref{eq:weak_Laplace_2D} lives in a 3D domain define in \eqref{eq:physical_domain}, but have been simplified by restricting the initial data and the forcing $\widetilde{\mathbf{F}}$  to functions that vary only in 2D, the $xy$-plane.
\begin{theorem}\label{theo:PML_2D_Edge_Laplace}
Consider the 2D   PML equation \eqref{eq:weak_Laplace_2D} in the Laplace space, with  piecewise constant $d_x(x) = d_y(y) = d\ge 0$ and  $\alpha_x = \alpha_y = \alpha\ge 0$,  subject to the boundary conditions \eqref{eq:BC_General2_Laplace},  with $ |\gamma_\eta| \le 1$  at $x = \pm 1$ and at $y = \pm 1$, and the interface condition \eqref{eq:elastic_elastic_interface_Laplace} at discontinuities of $d_x(x)$, $d_y(y)$. For all $s$ such that  $\Re(s) \ge a> 0$ we have
{
\begin{equation}\label{eq:energy_estimate_pml_laplace_edge}
\begin{split}
&\|\sqrt{\Re(s S_x)}\widetilde{\mathbf{Q}} (\cdot ,\cdot ,\cdot , s)\|_{P}^2 \le \|\widetilde{\mathbf{Q}} (\cdot ,\cdot ,\cdot , s)\|_{P} \|\widetilde{\mathbf{F}} (\cdot ,\cdot ,\cdot , s)\|_{P}^2 + \widetilde{\mathrm{BT}},  \quad \Re(s S_x) \ge a > 0,\\
&  \widetilde{\mathbf{F}} =     \mathbf{P}\left(S_x \widetilde{\mathbf{F}}_{Q}  -  \sum_{\xi = x, y}\frac{d }{s + \alpha} \widetilde{\mathbf{F}}_{w_\xi}\right) , 
\quad
\widetilde{\mathrm{BT}} = \sum_{\xi = x, y}\int_{\widetilde{\Gamma}} \frac{1}{2}\left[\widetilde{\mathbf{Q}}^\dagger\mathbf{A}_{\xi} \widetilde{\mathbf{Q}}\right]\Big|_{-1}^{1}\frac{dxdydz}{d\xi}  \le 0.
\end{split}
\end{equation}
}
\end{theorem}
The proof  of Theorem \ref{theo:PML_2D_Edge_Laplace} has been moved to  Appendix \ref{sec:proof_theorem2}.

   \begin{remark}\label{remark:pml_edge}
It is also noteworthy   that  similar energy  estimates \eqref{eq:energy_estimate_pml_laplace_edge} are valid for the  PML edge problems in the $xz$-edge or $yz$-edge, that is  when $ d_x = d_z = d>0, d_y=0$ or  $ d_y = d_z = d>0, d_x =0$.
\end{remark}

\subsection{The PML  corner problem}

Consider the corresponding PML corner problem,  \eqref{eq:elastic_pml_1A}-\eqref{eq:elastic_pml_2A}   with $d_x  = d_y = d_z = d>0$ and $\alpha_x  = \alpha_y = \alpha_z = \alpha >0$.
Then, all PML metrics are identical $S_y = S_x = S_z$. We have $S_x/S_\xi = 1$ and 
\begin{equation}\label{eq:weak_Laplace_Corner}
\begin{split}
\left(sS_x\right) \mathbf{P}^{-1} \widetilde{\mathbf{Q}}  =  \sum_{\xi = x, y,z}\mathbf{A}_{\xi}\frac{\partial{\widetilde{\mathbf{Q}} }}{\partial \xi}  +  \mathbf{P}^{-1}\widetilde{\mathbf{F}} .
  \end{split}
  \end{equation}
\begin{theorem}\label{theo:PML_3D_Corner_Laplace}
Consider the 3D   PML equation \eqref{eq:weak_Laplace_Corner} in the Laplace space, with  piecewise constant $d_x(x) = d_y(y) = d_z(z) =d\ge 0$ and  $\alpha_x = \alpha_y =  \alpha_z =  \alpha\ge 0$,  subject to the boundary conditions \eqref{eq:BC_General2_Laplace},  with $ |\gamma_\eta| \le 1$  at $x = \pm 1$,  $y = \pm 1$ and $z = \pm 1$, and the interface condition \eqref{eq:elastic_elastic_interface_Laplace} at discontinuities of $d_x(x)$, $d_y(y)$, $d_z(z)$. For all $s$ such that  $\Re(s) \ge a > 0$ we have
%
\begin{equation}\label{eq:energy_estimate_pml_laplace_corner}
\begin{split}
&\|\sqrt{\Re(s S_x)}\widetilde{\mathbf{Q}} (\cdot ,\cdot ,\cdot , s)\|_{P}^2 \le \|\widetilde{\mathbf{Q}} (\cdot ,\cdot ,\cdot , s)\|_{P} \|\widetilde{\mathbf{F}} (\cdot ,\cdot ,\cdot , s)\|_{P} + \widetilde{\mathrm{BT}},  \quad \Re(s S_x)\ge a > 0,\\
&  \widetilde{\mathbf{F}} =     \mathbf{P}\left(S_x \widetilde{\mathbf{F}}_{Q}  -   \sum_{\xi = x, y,z}\frac{d }{s + \alpha}\widetilde{\mathbf{F}}_{w_\xi}\right) ,  
\quad
\widetilde{\mathrm{BT}} = \sum_{\xi = x, y, z}\int_{\widetilde{\Gamma}} \frac{1}{2}\left[\widetilde{\mathbf{Q}}^\dagger\mathbf{A}_{\xi} \widetilde{\mathbf{Q}}\right]\Big|_{-1}^{1}\frac{dxdydz}{d\xi}  \le 0. 
\end{split}
\end{equation}
\end{theorem}
The proof of Theorem \ref{theo:PML_3D_Corner_Laplace} have been moved to   Appendix \ref{sec:proof_theorem3}.


We introduce the energy norms  in the physical space 
\begin{equation}\label{eq:elastic_pml_3D_energy_physical}
\begin{split}
 \|{\mathbf{Q}} (\cdot ,\cdot ,\cdot , t)\|_{P}^2  = \|\mathcal{L}^{-1}\widetilde{\mathbf{Q}} (\cdot ,\cdot ,\cdot , s)\|_{P}^2, \quad
  \|{\mathbf{F}} (\cdot ,\cdot ,\cdot , t)\|_{P}^2  = \|\mathcal{L}^{-1}\widetilde{\mathbf{F}} (\cdot ,\cdot ,\cdot , s)\|_{P}^2 ,
  \end{split}
  \end{equation}
    where
    \begin{align}
   {\mathbf{F}} (\cdot ,\cdot ,\cdot , t) =    \left(\delta(t) + {d}_x e^{-\alpha t} \right) * \mathbf{P}{\mathbf{F}}_{Q}\left(x,y,z, t\right) - \sum_{\xi = x, y, z} d_{\xi}   e^{-\alpha t} * \mathbf{P}{\mathbf{F}}_{w_\xi} \left(x,y,z, t\right).
  \end{align}
   The derivation of the source term $ {\mathbf{F}} (\cdot ,\cdot ,\cdot , t) $ can be found in Appendix \ref{sec:source_terms}.
Here, $*$ is the convolution operator, and $\delta(t) $ is the Dirac delta distribution.
\begin{theorem}\label{Theo:estimate_PML}
Consider the energy estimate in the Laplace space
\begin{equation}\label{eq:energy_estimate_pml_laplace}
\begin{split}
& \|\widetilde{\mathbf{Q}} (\cdot ,\cdot ,\cdot , s)\|_{P}^2 \le \frac{1}{a }\|\widetilde{\mathbf{Q}} (\cdot ,\cdot ,\cdot , s)\|_{P} \|\widetilde{\mathbf{F}} (\cdot ,\cdot ,\cdot , s)\|_{P}, \quad \Re(s) \ge a  > 0.
\end{split}
\end{equation}
For any $a>0$ and $T>0$ we have
  \begin{align}\label{eq:estimate_0_0}
  \int_0^{T} e^{-2at}  \|{\mathbf{Q}} (\cdot ,\cdot ,\cdot , t)\|_{P}^2 dt \le  \frac{1}{a ^2}\int_0^{T} e^{-2at} \left( \| {\mathbf{F}} (\cdot ,\cdot ,\cdot , t)\|_{P}^2\right) dt.
  \end{align}
\end{theorem}
\begin{proof}
Squaring both sides of \eqref{eq:energy_estimate_pml_laplace} having
 \begin{equation}\label{eq:energy_estimate_pml_laplace_corner_cont_0}
\begin{split}
  \|\widetilde{\mathbf{Q}} (\cdot ,\cdot ,\cdot , s)\|_{P} ^2  \le   \frac{1}{a  ^2}  \|\widetilde{\mathbf{F}} (\cdot ,\cdot ,\cdot , s)\|_{P}^2.
  \end{split}
\end{equation}
We will use the identity \cite{GustafssonKreissOliger1995} 
 \[
  \int_0^{+\infty} e^{-2at}\| f(\cdot,\cdot,\cdot,t)\|_{P}^2dt= \frac{1}{2\pi}\int_{-\infty}^{+\infty} \|\widetilde{f}\left(\cdot,\cdot,\cdot, a+ib\right)\|_{P}^2db,
  \]
 to obtain
  \begin{align}\label{eq:estimate_0_0}
    \int_0^{\infty} e^{-2at}  \|{\mathbf{Q}} (\cdot ,\cdot ,\cdot , t)\|_{P}^2 dt \le  \frac{1}{a ^2}  \int_0^{\infty} e^{-2at} \left( \| {\mathbf{F}} (\cdot ,\cdot ,\cdot , t)\|_{P}^2\right) dt.
  \end{align}
 Finally, the result follows by the usual argument that the future cannot influence the past.
\end{proof}

Since the boundary terms $\widetilde{\mathrm{BT}}$ in \eqref{eq:energy_estimate_pml_laplace_corner}, \eqref{eq:energy_estimate_pml_laplace_edge} and \eqref{eq:energy_estimate_pml_laplace_strip} are never positive, we can use Theorem \ref{Theo:estimate_PML} to invert the the estimates in Theorem \ref{theo:PML_1D_Strip_Laplace}, \ref{theo:PML_2D_Edge_Laplace} and \ref{theo:PML_3D_Corner_Laplace} to get an energy estimate in the physical space.
The estimate seemingly allows the energy of the solution to grow  exponentially with time for general data. However, by choosing $a>0$ in relation to the length 
of the time interval of interest, a bound  involving only algebraic growth in time follows.
\section{The discontinuous Galerkin spectral element method}
In this section, we present the DG approximations for the undamped system \eqref{eq:linear_wave} and the PML  \eqref{eq:elastic_pml_1A}-\eqref{eq:elastic_pml_2A}, subject to the boundary conditions  \eqref{eq:BC_General2} and the interface condition \eqref{eq:elastic_elastic_interface}. We will  use the physically motivated numerical fluxes develop in \cite{DuruGabrielIgel2017,Duru_exhype_2_2019} to patch DG elements into the global domain. The physically motivated numerical flux is upwind by construction and gives an energy estimate analogous to \eqref{eq:energy_conservation}.  The boundary and inter-element procedure will begin with the integral form \eqref{eq:product_1}.
As we will see later, the procedure and analysis carries over when numerical approximations are introduced.

\subsection{Weak boundary and inter-element procedures, and the energy identity}
We begin by discretizing the domain $(x,y,z)  \in \Omega = [-1, 1]\times[-1, 1]\times[-1, 1]$ into $K\times L\times M$ elements denoting the $klm$-th element by $\Omega_{klm} = [x_k, x_{k+1}]\times [y_l, y_{l+1}]\times [z_m, z_{m+1}]$, where $k = 1, 2, \dots, K$, $l = 1, 2, \dots, L$, $m = 1, 2, \dots, M$ with $x_1 = -1$,  $y_1 = -1$, $z_1 = -1$ and $x_{K+1} = 1$, $y_{L+1} = 1$, $z_{M+1} = 1$.

We map the element $\Omega_{lmn} = [x_k, x_{k+1}]\times [y_l, y_{l+1}]\times [z_m, z_{m+1}]$ to a reference element $(q, r, s) \in \widetilde{\Omega} = [-1, 1]^{3}$ by the linear transformation 
{
\begin{align}\label{eq:transf}
x = x_k + \frac{\Delta{x}_k}{2}\left(1 + q \right),  \quad
y = y_l + \frac{\Delta{y}_l}{2}\left(1 + r \right), \quad 
z = z_m + \frac{\Delta{z}_m}{2}\left(1 + s \right),  
\end{align}
}
with
{
\[
\quad \Delta{x}_k = x_{k+1} - x_k, \quad  \Delta{y}_l = y_{l+1} - y_l, \quad  \Delta{z}_m = y_{m+1} - y_m.
\]
}
The Jacobian of the transformation is
\[
 \quad J = \frac{\Delta{x}_k}{2}\frac{\Delta{y}_l }{2}\frac{\Delta{z}_m}{2} > 0.
\]
Note that the non-zero metric derivatives are
\[
q_x = \frac{2}{\Delta{x}_k}, \quad r_y = \frac{2}{\Delta{y}_l}, \quad s_z = \frac{2}{\Delta{z}_m}.
\]
All other metric derivatives vanish identically.
The volume integral yields
\begin{align}
\int_{\Omega} f(x,y,z) dxdydz = \sum_{k=1}^{K}\sum_{l=1}^{L}\sum_{m=1}^{M}\int_{\Omega_{klm}} f(x,y,z) dxdydz = \sum_{k=1}^{K}\sum_{l=1}^{L}\sum_{m=1}^{M}\int_{\widetilde{\Omega}} f(q,r,s) Jdqdrds,
\end{align}
and the transformed gradient operator is
\[
{\grad}  := \left(\frac{\partial}{\partial x}, \frac{\partial}{\partial y}, \frac{\partial}{\partial z}\right)^T= \left(q_x \frac{\partial}{\partial q}, r_y \frac{\partial}{\partial r}, s_z \frac{\partial}{\partial s}\right)^T.
\]
We introduce
\begin{align}\label{eq:rename_var}
& \mathbf{w}_{q} = \mathbf{w}_{x}, \quad \mathbf{w}_{r} = \mathbf{w}_{y}, \quad \mathbf{w}_{s} = \mathbf{w}_{z}, \quad \mathbf{A}_{q} = q_x \mathbf{A}_{x} , \quad \mathbf{A}_{r} = r_y \mathbf{A}_{y} , \quad \mathbf{A}_{s} = s_z \mathbf{A}_{z}, \nonumber \\
& {d}_{q}(q) =  {d}_{x}(x) , \quad {d}_{r}(r) =  {d}_{y}(y) , \quad {d}_{s}(s) =  {d}_{z}(z), \quad {\alpha}_{q}(q) =  {\alpha}_{x}(x) , \quad {\alpha}_{r}(r) =  {\alpha}_{y}(y) , \quad {\alpha}_{s}(s) =  {\alpha}_{z}(z).
\end{align}
Therefore, the  elemental integral form reads
\begin{align}\label{eq:weak_0}
\int_{\widetilde{\Omega}}\left[\boldsymbol{\phi}^T\mathbf{P}^{-1}\frac{\partial \mathbf{Q}}{\partial t}\right]Jdqdrds &= \sum_{\xi = q, r, s}\int_{\widetilde{\Omega}}\boldsymbol{\phi}^T\left(\left[\mathbf{A}_{\xi}\frac{\partial{\mathbf{Q}}}{\partial \xi}  - d_{\xi}\left(\xi\right)\mathbf{w}_{\xi}\right]\right)Jdqdrds, 
\end{align}
\begin{equation}\label{eq:weak_1}
\begin{split}
\int_{\widetilde{\Omega}}\left[\boldsymbol{\phi}^T\frac{\partial{\mathbf{w}_{\xi}}}{\partial t}\right]Jdqdrds 
& =  \int_{\widetilde{\Omega}}\boldsymbol{\phi}^T\left(\mathbf{A}_{\xi}\frac{\partial{\mathbf{Q}}}{\partial \xi} -   \left(\alpha_{\xi}(\xi) + d_{\xi}\left(\xi\right)\right)\mathbf{w}_{\xi}\right)Jdqdrds .
  \end{split}
  \end{equation}
Next   we consider the element boundaries, $\xi = -1, 1$,  and generate boundary and interface data $\widehat{\mathbf{T}}$, $ \widehat{\mathbf{v}}$. The hat-variables encode the tractions and particle velocities at the element boundaries. Please see  \cite{DuruGabrielIgel2017,Duru_exhype_2_2019} for more elaborate discussions.
The next step is to construct fluctuations by penalizing data, that is hat-variables, against the ingoing characteristics only.
  That is, for each $\eta = x, y, z$ with $Z_\eta > 0$, then
    \begin{align*}
  &{G}_\eta = \frac{1}{2} {Z}_\eta \left({v}_\eta - \widehat{{v}}_\eta \right)+ \frac{1}{2}\left({T}_\eta  - \widehat{{T}}_\eta \right)\Big|_{\xi = 1},   
  \quad 
  &{G}_\eta  = \frac{1}{2} {Z}_\eta\left({v}_\eta  - \widehat{{v}}_\eta \right)- \frac{1}{2}\left({T}_\eta  - \widehat{{T}}_\eta \right)\Big|_{\xi = -1}.
  \end{align*}
Next we define the flux fluctuation vectors obeying the eigen-structure of the elastic wave equation.
  \begin{align}\label{eq:num_flux_fluctuations}
  \mathbf{FL}_{\xi}
  =
  \begin{pmatrix}
\mathbf{G} \\
-\boldsymbol{a}_{\xi}^T \mathbf{Z}^{-1} {\mathbf{G}} 
\end{pmatrix},
\quad
 \mathbf{FR}_{\xi}
  =
  \begin{pmatrix}
\mathbf{G} \\
\boldsymbol{a}_{\xi}^T  \mathbf{Z}^{-1} {\mathbf{G}} 
\end{pmatrix},
\quad
\mathbf{Z}
  =
  \begin{pmatrix}
Z_x & 0 & 0\\
0 & Z_y & 0\\
0 & 0 & Z_z\\
\end{pmatrix}.
  \end{align}
    To patch the elements together, we append the flux fluctuation vectors $\mathbf{FL}_{\xi}$ and $\mathbf{FR}_{\xi}$ to \eqref{eq:weak_0}-\eqref{eq:weak_1} and integrate over the element face, we have
   \begin{align}\label{eq:weak_01}
\int_{\widetilde{\Omega}}\left[\boldsymbol{\phi}^T\mathbf{P}^{-1}\frac{\partial \mathbf{Q}}{\partial t}\right]Jdqdrds &= \sum_{\xi = x, y, z}\int_{\widetilde{\Omega}}\boldsymbol{\phi}^T\left(\left[\mathbf{A}_{\xi}\frac{\partial{\mathbf{Q}}}{\partial \xi}  - d_{\xi}\left(\xi\right)\mathbf{w}_{\xi}\right]\right)Jdqdrds \nonumber \\
&+ \sum_{\xi = x, y, z}\int_{\widetilde{\Gamma}}\left(\boldsymbol{\phi}^T(-1)\mathbf{FL}_{\xi} + \boldsymbol{\phi}^T(1)\mathbf{FR}_{\xi} \right)J\frac{dqdrds}{d\xi},
\end{align}
\begin{equation}\label{eq:weak_11}
\begin{split}
\int_{\widetilde{\Omega}}\left[\boldsymbol{\phi}^T\frac{\partial{\mathbf{w}_{\xi}}}{\partial t}\right]Jdqdrds 
& =  \int_{\widetilde{\Omega}}\boldsymbol{\phi}^T\left(\mathbf{A}_{\xi}\frac{\partial{\mathbf{Q}}}{\partial \xi} -   \left(\alpha_{\xi}(\xi) + d_{\xi}\left(\xi\right)\right)\mathbf{w}_{\xi}\right)Jdqdrds \\
 &+ \underbrace{ \theta_\xi \int_{\widetilde{\Gamma}}\left(\boldsymbol{\phi}^T(-1)\mathbf{FL}_{\xi} + \boldsymbol{\phi}^T(1)\mathbf{FR}_{\xi} \right)J\frac{dqdrds}{d\xi}}_{\text{PML stabilizing flux fluctuation}}.
  \end{split}
  \end{equation}
Here, $\theta_\xi$ are PML stabilizing parameters to be determined by requiring stability of the discrete PML, in the sense corresponding to Theorems \ref{theo:PML_1D_Strip_Laplace}, \ref{theo:PML_2D_Edge_Laplace}, \ref{theo:PML_3D_Corner_Laplace} and \ref{Theo:estimate_PML}.

We will now demonstrate the stability of the numerical flux fluctuations $\mathbf{FL}_{\xi}$ and $\mathbf{FR}_{\xi}$. We note, however, that  since we have not introduced any numerical approximation the fluctuations vanish identically, that is $\mathbf{G} = \widetilde{\mathbf{G}} = 0$ and $\mathbf{FL}_{\xi} =\mathbf{FR}_{\xi} = 0$. We will see later that the analysis will follow through when numerical approximations are introduced.

We will begin by noting the identities
\begin{align}
\mathbf{Q}^T\mathbf{FL}_{\xi} = \mathbf{v}^T\mathbf{G} - \mathbf{T}^T{\mathbf{Z}}^{-1}{\mathbf{G}}, \quad \mathbf{Q}^T\mathbf{FR}_{\xi} = \mathbf{v}^T\mathbf{G} + \mathbf{T}^T{\mathbf{Z}}^{-1}{\mathbf{G}},
\end{align}
and 
{
\small
\begin{equation}\label{eq:identity_pen}
\begin{split}
 \left(\mathbf{v}^T \mathbf{G} - \mathbf{T}^T {\mathbf{Z}}^{-1}{\mathbf{G}} + \mathbf{v}^T\mathbf{T}\right)\Big|_{\xi = -1} &= \sum_{\eta = x, y, z} \left( v_\eta G_\eta    - \frac{1}{Z_\eta }T_\eta   G_\eta + v_\eta  T_\eta \right)\Big|_{\xi = -1} 
= \sum_{\eta = x, y, z} \left(\frac{1}{Z_\eta  }|G_\eta |^2  + \widehat{T}_\eta \widehat{v}_\eta\right)\Big|_{\xi = -1},\\
\left(\mathbf{v}^T \mathbf{G} + \mathbf{T}^T {\mathbf{Z}}^{-1}{\mathbf{G}} - \mathbf{v}^T\mathbf{T}\right)\Big|_{\xi = 1}  &= \sum_{\eta = x, y, z} \left(v_\eta  G_\eta   + \frac{1}{Z_\eta }T_\eta  G_\eta - v_\eta  T_\eta \right)\Big|_{\xi = 1} 
= \sum_{\eta = x,y,z} \left(\frac{1}{Z_\eta  }|G_\eta |^2  - \widehat{T}_\eta \widehat{v}_\eta \right)\Big|_{\xi = 1}.
\end{split}
\end{equation}
}

To simplify the analysis we will consider a DG approximation with  two elements only, separated at $x = 0$. We will denote the corresponding fields, material parameters and energies in the positive/negative sides of the interface with the superscripts $+/-$. Each of the two elements have 5 external boundaries where the boundary conditions \eqref{eq:BC_General2}  are imposed, and one internal boundary connecting the two elements where the interface condition \eqref{eq:elastic_elastic_interface} is imposed.  The boundary and interface conditions are implemented using the numerical flux fluctuations \eqref{eq:num_flux_fluctuations}. It is also noteworthy that the analysis can be extended to a  DG approximation consisting of many elements.

Introduce the fluctuation term
  \begin{align}\label{eq:weak_fluctuate}
  F_{luc}\left({G},  Z\right) = - \sum_{\xi = q, r, s}\sum_{\eta = x,y,z}\int_{\widetilde{\Gamma}}\sqrt{\xi_x^2 + \xi_y^2 + \xi_z^2}\left(\left(\left(\frac{1}{Z_{\eta}}|{G}_\eta |^2\right)\Big|_{\xi = -1} + \left(\frac{1}{Z_{\eta}}|{G}_\eta |^2\right)\Big|_{\xi = 1}  \right)\right)J\frac{dqdrds}{d\xi} \le 0.
  \end{align}
   For the two elements model we introduce the external boundary terms 
 \begin{align*}
\mathrm{BTs}\left(\widehat{v}_{\eta}^{-}, \widehat{T}_{\eta}^{-}\right)  = & - \sum_{\xi = r, s}\int_{\widetilde{\Gamma}}\sum_{\eta = x,y,z}\left(\left(\sqrt{\xi_x^2 + \xi_y^2 + \xi_z^2}J \widehat{T}_\eta^{-} \widehat{v}_\eta^{-} \right)\Big|_{\xi = -1} - \left(\sqrt{\xi_x^2 + \xi_y^2 + \xi_z^2}J \widehat{T}_\eta^{-} \widehat{v}_\eta^{-}  \right)\Big|_{\xi = 1}\right)\frac{dqdrds}{d\xi} \\
&- \int_{\widetilde{\Gamma}}\sum_{\eta = x,y,z}\left(\left(\sqrt{q_x^2 + q_y^2 + q_z^2}J \widehat{T}_\eta^{-} \widehat{v}_\eta^{-}  \right)\Big|_{q = -1}\right){drds} \le 0,
 \end{align*}
 \begin{align*}
\mathrm{BTs}\left(\widehat{v}_{\eta}^{+}, \widehat{T}_{\eta}^{+}\right)  = & - \sum_{\xi = r, s}\int_{\widetilde{\Gamma}}\sum_{\eta = x,y,z}\left(\left(\sqrt{\xi_x^2 + \xi_y^2 + \xi_z^2}J \widehat{T}_\eta^{+} \widehat{v}_\eta^{+} \right)\Big|_{\xi = -1} - \left(\sqrt{\xi_x^2 + \xi_y^2 + \xi_z^2}J \widehat{T}_\eta^{+} \widehat{v}_\eta^{+}  \right)\Big|_{\xi = 1}\right)\frac{dqdrds}{d\xi} \\
&+ \int_{\widetilde{\Gamma}}\sum_{\eta = x,y,z}\left(\left(\sqrt{q_x^2 + q_y^2 + q_z^2}J \widehat{T}_\eta^{+} \widehat{v}_\eta^{+}  \right)\Big|_{q = 1}\right){drds} \le 0,
 \end{align*}
 and the interface term
  \begin{align*}
  \mathrm{IT}_s(\widehat{v}^{\pm}, \widehat{T}^{\pm}) = - \sum_{\eta = x,y,z}\int_{\widetilde{\Gamma}}\sqrt{q_x^2 + q_y^2 + q_z^2}{\widehat{T}}_\eta  \lJump{{\widehat{v}}_\eta \rJump} J {drds} \equiv 0.
\end{align*}
   We can now prove the theorem.
 \begin{theorem}\label{theo:weak_form_statbility}
Consider the elemental weak form \eqref{eq:weak_1}. When all PML damping vanish $d_\xi = 0$, then  we have the energy identity
   {
   \small
 \begin{equation}\label{eq:weak_energy}
\begin{split}
\frac{d}{dt} \left(\|\mathbf{Q}^{-}\left(\cdot, \cdot, \cdot, t\right)\|^2_P + \|\mathbf{Q}^{+}\left(\cdot, \cdot, \cdot, t\right)\|^2_P \right) &=   \mathrm{IT}_s\left(\widehat{v}^{\pm},  \widehat{T}^{\pm}\right)  +  BT_s\left(\widehat{v}^{-}, \widehat{T}^{-}\right) + BT_s\left(\widehat{v}^{+}, \widehat{T}^{+}\right) \\
&+   F_{luc}\left({G} ^{-},  Z^{-}\right) +   F_{luc}\left({G} ^{+},  Z^{+}\right) \le 0.
\end{split}
\end{equation}
}
\end{theorem}
Note that the interface term vanishes identically $\mathrm{IT}_s\left(\widehat{v}^{\pm},  \widehat{T}^{\pm}\right) \equiv  0$,  and since we have not introduced any numerical approximation, we have ${\mathbf{v}} = \widehat{\mathbf{v}}$ and ${\mathbf{T}} = \widehat{\mathbf{T}}$, and the fluctuation terms  also vanish $F_{luc}\left({G} ^{-},  Z^{-}\right) = 0$, $F_{luc}\left({G} ^{+},  Z^{+}\right)  = 0$.
 Thus, the energy estimate \eqref{eq:weak_energy} is completely analogous to the physical energy estimate \eqref{eq:energy_conservation}. However, when numerical approximations are introduced numerical solutions can  be discontinuous across the inter-element boundaries and the surface terms will dissipate energy. The artificial dissipation will vanish though in the limit of mesh refinement.

As we are yet to introduce any numerical approximation, for the PML with $d_\xi \ge 0$, Theorems \ref{theo:PML_1D_Strip_Laplace}, \ref{theo:PML_2D_Edge_Laplace}, \ref{theo:PML_3D_Corner_Laplace} and \ref{Theo:estimate_PML} are also valid for \eqref{eq:weak_1}, for any $\theta_{\xi}$. However, we will see later in the coming sections that this is not generally true when numerical approximations are introduced.

\subsection{The Galerkin approximation}
Inside the reference element $(q, r, s) \in \widetilde{\Omega} = [-1, 1]^{3}$, approximate the elemental solution by a polynomial interpolant $\mathbf{u}(q,r,s, t)$,  and write 
\begin{equation}\label{eq:variables_elemental}
\mathbf{u}\left(q,r,s, t\right) =  \sum_{i = 1}^{P+1} \sum_{j = 1}^{P+1}\sum_{k = 1}^{P+1}\mathbf{u}_{ijk}(t) \boldsymbol{\phi}_{ijk}\left(q,r,s\right),
\end{equation}
where $\mathbf{u}_{ijk}(t)$, are the elemental degrees of freedom to be determined, 
 and $ \boldsymbol{\phi}_{ijk}(q,r,s)$ are the $ijk$-th interpolating polynomials. We consider tensor products of  nodal basis with $ \boldsymbol{\phi}_{ijk}(q,r,s) = \mathcal{L}_i(q)\mathcal{L}_j(r)\mathcal{L}_k(s)$,  where $\mathcal{L}_i(q)$, $\mathcal{L}_j(r)$, $\mathcal{L}_k(s)$, are one dimensional nodal interpolating Lagrange polynomials of degree $P$, with $ \mathcal{L}_i(q_m) = \delta_{im}$. Here, $\delta_{im}$ is the Kronecker delta,  with $\delta_{im} = 1$ if $i = m$, and $\delta_{im} = 0$ if $i \ne m$. The interpolating nodes $q_m$, $m = 1, 2, \dots, P+1$, are the nodes of a Gauss quadrature with
\begin{equation}\label{eq:quad_rule_3D}
 \sum_{i = 1}^{P+1} \sum_{j = 1}^{P+1}  \sum_{k = 1}^{P+1}f(q_i, r_j, s_k)h_ih_jh_k \approx \int_{\widetilde{\Omega}}f(q, r, s) dq dr ds,
\end{equation}
where $h_i > 0$, $h_j>0$, $h_k>0$, are the quadrature weights.
We will only use quadrature rules such  that for all polynomial integrand $f(\xi)$ of degree $\le 2P-1$, the corresponding one dimensional rule is exact, 
 $\sum_{m = 1}^{P+1} f(\xi_m)h_m = \int_{-1}^{1}f(\xi) d\xi.$
  Admissible  candidates can be Gauss-Legendre-Lobatto quadrature rule with GLL nodes, Gauss-Legendre quadrature rule with GL nodes and Gauss-Legendre-Radau quadrature rule with GLR nodes. While both endpoints, $\xi = -1, 1$, are part of   GLL quadrature nodes,  the   GLR quadrature contains only the first endpoint $\xi = -1$ as a node. Lastly,  for the GL quadrature, both endpoints, $\xi= -1, 1$, are not quadrature nodes. Note that when an endpoint is not a quadrature node, $\xi_1 \ne -1$ or $\xi_{P+1} \ne 1$, extrapolation is needed to compute numerical fluxes at the element boundary, $\xi = -1$ or $\xi= 1$. We also remark that the GLL quadrature rule is exact for polynomial integrand of degree $2P-1$, GLR quadrature rule is exact for polynomial integrand of degree $2P$, and GL quadrature rule is exact for polynomial integrand of  degree $2P+1$.

 
 \subsection{The spectral difference approximation}
 Introduce the   matrices $ H, Q \in \mathbb{R}^{(P+1) \times (P+1)}$, defined by
{
\begin{equation}
 H = \mathrm{diag}[h_1, h_2, \cdots, h_{P+1}], \quad Q_{ij} = \sum_{m = 1}^{P+1} h_m \mathcal{L}_i(q_m)  {\mathcal{L}_j^{\prime}(q_m)} = \int_{-1}^{1}\mathcal{L}_i(q)  {\mathcal{L}_j^{\prime}(q)} dq.
\end{equation}
}
Note that the matrix
\begin{equation}\label{eq:derivative_operator_1d}
D = H^{-1} Q \approx \frac{\partial}{\partial q},
\end{equation}
is a one space dimensional spectral difference approximation of the first derivative, in the transformed coordinate $q$.

Using the fact that the quadrature rule is exact  for all polynomial  integrand of degree $\le 2P-1$ 
  implies that
\begin{equation}\label{eq:sbp_property_a}
Q + Q^T =  B,
\quad
B_{ij} = \mathcal{L}_i(1)  \mathcal{L}_j (1)- \mathcal{L}_i(-1)  \mathcal{L}_j (-1).
\end{equation}
Equations \eqref{eq:derivative_operator_1d} and \eqref{eq:sbp_property_a} are the discrete equivalence of the integration-by-parts property, and the matrix $B$ projects the nodal degrees of freedom to element faces.
If boundary points $q = -1,1$ are quadrature nodes and we consider nodal bases  with $ \mathcal{L}_j (q_i) = \delta_{ij}$ then we have $B = \text{diag}[-1, 0,0, \dots, 0, 1].$ 

The one space dimensional derivative operator \eqref{eq:derivative_operator_1d} can be extended to higher space dimensions using the Kronecker products $\otimes$, having
\begin{equation}\label{eq:derivative_operator_3d}
\mathbf{D}_x = \frac{2}{\Delta{x}}\left(I_{9} \otimes D\otimes I\otimes I\right), \quad \mathbf{D}_y = \frac{2}{\Delta{y}}\left(I_{9} \otimes  I\otimes D\otimes I\right), \quad \mathbf{D}_z = \frac{2}{\Delta{z}}\left(I_{9} \otimes  I\otimes I\otimes D\right),
\end{equation}
{\small
\begin{equation}\label{eq:norm_operator_3d}
\mathbf{H}_x = \frac{\Delta{x}}{2}\left(I_{9} \otimes  H\otimes I\otimes I\right), \quad \mathbf{H}_y = \frac{\Delta{y}}{2}\left(I_{9} \otimes  I\otimes H\otimes I\right), \quad \mathbf{H}_z = \frac{\Delta{z}}{2}\left(I_{9} \otimes  I\otimes I\otimes H\right), \quad \mathbf{H} = \mathbf{H}_x\mathbf{H}_y\mathbf{H}_z.
\end{equation}
}
\begin{equation}\label{eq:norm_operator_3d}
\mathbf{P} = \left({P} \otimes  I\otimes I\otimes I\right), \quad \mathbf{A}_{\xi} = \left({A}_{\xi} \otimes  I\otimes I\otimes I\right)
\end{equation}
Here, $I$ is the $(P+1)\times(P+1)$ identity matrix and $I_{9} $ the $9 \times 9$ identity matrix. Note that the matrix product $\mathbf{H} $ commutes, that is $\mathbf{H} = \mathbf{H}_x\mathbf{H}_y\mathbf{H}_z = \mathbf{H}_x\mathbf{H}_z\mathbf{H}_y  = \mathbf{H}_y\mathbf{H}_z\mathbf{H}_x$.
We also introduce the projection matrices
\begin{align*}
&\mathbf{e}_x(\eta) = \left(I_{9} \otimes  \boldsymbol{e}(\eta)\otimes I\otimes I\right), \quad \mathbf{e}_y(\eta) = \left(I_{9} \otimes  I\otimes \boldsymbol{e}(\eta)\otimes I\right), 
\\
&\mathbf{e}_z(\eta) = \left(I_{9} \otimes  I\otimes I\otimes \boldsymbol{e}(\eta)\right), \quad  \mathbf{B}_{\eta}(\psi, \xi) = \mathbf{e}_{\eta}(\psi) \mathbf{e}_{\eta}^T(\xi),
\end{align*}
where
\[
  \boldsymbol{e}(\eta) = [\mathcal{L}_i(\eta), \mathcal{L}_i(\eta), \cdots, \mathcal{L}_{P+1}(\eta)]^T.
\]
Using  \eqref{eq:sbp_property_a} and \eqref {eq:derivative_operator_1d} we can rewrite \eqref{eq:derivative_operator_3d} as
\begin{align}\label{eq:disc_sbp}
 \mathbf{D}_\xi = -\mathbf{H}_\xi^{-1}\mathbf{D}_\xi^T\mathbf{H}_\xi + \mathbf{H}_\xi^{-1}\left(\mathbf{B}_\xi\left(1,1\right)-\mathbf{B}_\xi\left(-1,-1\right)\right), \quad \xi = x, y, z.
\end{align}

 \subsection{The DG approximation of the PML}
 We now make a classical  Galerkin approximation by choosing test functions  in the same space as the basis functions, so that the residual is orthogonal to the space of test functions. By rearranging the elemental degrees of freedom $[\mathbf{Q}_{ijk}(t) ]$ row-wise as a vector, $\mathbf{Q}(t) $,  of length $9(P+1)^d$ where $d = 3$ is the number of space dimensions, we have the semi-discrete approximation 
{
\begin{equation}\label{eq:disc_elastic_pml_1A}
\begin{split}
\mathbf{P}^{-1}\frac{ d {\mathbf{Q}}}{d t} = \sum_{\xi = x, y, z}\left[\mathbf{A}_{\xi}\mathbf{D}_{\xi} \mathbf{Q}  - \mathbf{d}_{\xi}\mathbf{w}_{\xi} - \mathbf{H}_{\xi}^{-1}{\left({\mathbf{e}_{\xi}(-1)} \mathbf{FL}_\xi +  {\mathbf{e}_{\xi}(1)} \mathbf{FR}_\xi  \right)}\right],
  \end{split}
  \end{equation}
\begin{equation}\label{eq:disc_elastic_pml_3A}
\begin{split}
\frac{d {\mathbf{w}_{\xi}}}{ d t} 
 = \mathbf{A}_{\xi}\mathbf{D}_{\xi} \mathbf{Q}  - \left(\mathbf{d}_{\xi} + \boldsymbol{\alpha}_{\xi}\right)\mathbf{w}_{\xi} - \underbrace{\theta_{\xi}\mathbf{H}_{\xi}^{-1}{\left({\mathbf{e}_{\xi}(-1)} \mathbf{FL}_\xi +  {\mathbf{e}_{\xi}(1)} \mathbf{FR}_\xi  \right)}}_{\text{PML stabilizing flux fluctuations}}.
  \end{split}
  \end{equation}
}

Note the close similarity between the semi-discrete approximation \eqref{eq:disc_elastic_pml_1A}-\eqref{eq:disc_elastic_pml_3A} and the continuous analogue \eqref{eq:elastic_pml_1A}-\eqref{eq:elastic_pml_2A}. The discrete operator  $\left(\mathbf{D}_x, \mathbf{D}_y, \mathbf{D}_z\right)^T$ is also analogous to the continuous gradient operator $ \left(\partial/\partial x, \partial/\partial y, \partial/\partial z\right)^T$, where we have replaced the continuous derivative operators  with their discrete counterparts given in \eqref{eq:derivative_operator_3d}, 
\[
\frac{\partial}{\partial x} \to \mathbf{D}_x, \quad  \frac{\partial}{\partial y} \to \mathbf{D}_y, \quad \frac{\partial}{\partial z} \to \mathbf{D}_z.
\]
%
\section{Numerical stability}
In this section, we will prove numerical stability. 
As before, to simplify the analysis we will consider a DG approximation with  two elements only, separated at $x = 0$.  The analysis can be extended to a  DG approximation consisting of many elements.  
We will begin with the discrete undamped problem, that is when all PML absorption functions vanish $d_{\xi} = 0$.
We will proceed later to prove numerical stability of the semi-discrete approximation when the PML absorption functions are present $d_{\xi} \ge 0$.
\subsection{Stability analysis for the undamped discrete problem}
We will begin with the undamped case, $d_\xi \equiv 0$.  With $d_\xi\equiv 0$, then the auxiliary variables, $\mathbf{w}_\xi$ decouple completely from the modified discrete wave equation \eqref{eq:disc_elastic_pml_1A}. Note that this is equivalent to considering \eqref{eq:disc_elastic_pml_1A} by itself.
We approximate the integrals in \eqref{eq:physical_energy} by the corresponding quadrature rule \eqref{eq:quad_rule_3D}, having
{
\small
\begin{align}\label{eq:disc_elastic_energy}
\|\mathbf{Q}\|^2_{hP} := \frac{1}{2}{\mathbf{Q}} ^T {\mathbf{H}} {\mathbf{P}}^{-1}  {\mathbf{Q}} =  \frac{\Delta{x}}{2}\frac{\Delta{y}}{2}\frac{\Delta{z}}{2}\sum_{i=1}^{P+1}\sum_{j=1}^{P+1}\sum_{k=1}^{P+1} \frac{1}{2}{\mathbf{Q}}_{ijk} ^T  {\mathbf{P}}_{ijk}^{-1}   {\mathbf{Q}}_{ijk} h_ih_jh_k > 0.
\end{align}
}
We also approximate the surface integrals in the boundary and interface terms.

For the two elements model we introduce the external boundary terms 
 \begin{align*}
 \mathcal{BT}_s\left(\widehat{v}_{\eta}^{-}, \widehat{T}_{\eta}^{-}\right)  = &  \sum_{\xi = r, s}\sum_{i = 1}^{P+1}\sum_{k = 1}^{P+1}\sum_{\eta = x,y,z}\left(\left(\sqrt{\xi_x^2 + \xi_y^2 + \xi_z^2}J \widehat{T}_\eta^{-} \widehat{v}_\eta^{-} \right)\Big|_{\xi = 1} - \left(\sqrt{\xi_x^2 + \xi_y^2 + \xi_z^2}J \widehat{T}_\eta^{-} \widehat{v}_\eta^{-}  \right)\Big|_{\xi = -1}\right)_{i k} h_{i} h_{k} \\
&- \sum_{i = 1}^{P+1}\sum_{k = 1}^{P+1}\sum_{\eta = x,y,z}\left(\left(\sqrt{q_x^2 + q_y^2 + q_z^2}J \widehat{T}_\eta^{-} \widehat{v}_\eta^{-}  \right)\Big|_{q = -1}\right)_{i k} h_{i} h_{k} \le 0,
 \end{align*}
 \begin{align*}
\mathcal{BT}_s\left(\widehat{v}_{\eta}^{+}, \widehat{T}_{\eta}^{+}\right)  = &  \sum_{\xi = r, s}\sum_{i = 1}^{P+1}\sum_{k = 1}^{P+1}\sum_{\eta = x,y,z}\left(\left(\sqrt{\xi_x^2 + \xi_y^2 + \xi_z^2}J \widehat{T}_\eta^{+} \widehat{v}_\eta^{+} \right)\Big|_{\xi = 1} - \left(\sqrt{\xi_x^2 + \xi_y^2 + \xi_z^2}J \widehat{T}_\eta^{+} \widehat{v}_\eta^{+}  \right)\Big|_{\xi = -1}\right)_{i k} h_{i} h_{k}  \\
&+\sum_{i = 1}^{P+1}\sum_{k = 1}^{P+1}\sum_{\eta = x,y,z}\left(\left(\sqrt{q_x^2 + q_y^2 + q_z^2}J \widehat{T}_\eta^{+} \widehat{v}_\eta^{+}  \right)\Big|_{q = 1}\right)_{i k} h_{i} h_{k}  \le 0,
 \end{align*}
  the interface term
  \begin{align*}
  \mathcal{IT}_s\left(\widehat{v}^{\pm}, \widehat{T}^{\pm} \right) = - \sum_{i = 1}^{P+1}\sum_{k = 1}^{P+1}\sum_{\eta = x,y,z}\left(\sqrt{q_x^2 + q_y^2 + q_z^2}{\widehat{T}}_\eta  \lJump{{\widehat{v}}_\eta \rJump} J\right)_{i k} h_{i} h_{k} \equiv 0,
\end{align*}
   and the fluctuation term
    \begin{align*}
  \mathcal{F}_{luc}\left({G},  Z\right) = - \sum_{\xi = q, r, s}\sum_{\eta = x, y, z}\sum_{i = 1}^{P+1}\sum_{k = 1}^{P+1}\left(\sqrt{\xi_x^2 + \xi_y^2 + \xi_z^2}\left(\left(\left(\frac{1}{Z_{\eta}}|{G}_\eta |^2\right)\Big|_{\xi = -1} + \left(\frac{1}{Z_{\eta}}|{G}_\eta |^2\right)\Big|_{\xi = 1}  \right)\right)J\right)_{i, k} \le 0.
  \end{align*}

We have 
 \begin{theorem}\label{theo:discrete_case_no_damping}
 Consider the  semi-discrete  approximation   \eqref{eq:disc_elastic_pml_1A}-\eqref{eq:disc_elastic_pml_3A}. When all PML absorption functions vanish,  $d_{\xi} = 0$, the solution of the semi-discrete approximation satisfies the energy identity
 {\small
 \begin{equation}\label{eq:discrete_energy_estimate_elastic_3D}
\frac{d}{dt} \left(\|\mathbf{Q}^{-}\|^2_{hP} + \|\mathbf{Q}^{+}\|^2_{hP} \right)=   \mathcal{IT}_s\left(\widehat{v}^{\pm},  \widehat{{T}}^{\pm}\right)  +  \mathcal{BT}_s\left(\widehat{v}^{-}, \widehat{T}^{-}\right) + BT_s\left(\widehat{v}^{+}, \widehat{T}^{+}\right) +   \mathcal{F}_{luc}\left({G} ^{-},  Z^{-}\right) +   \mathcal{F}_{luc}\left({G} ^{+},  Z^{+}\right) \le 0.
\end{equation}
}
 \end{theorem}
 The proof of Theorem \ref{theo:discrete_case_no_damping} can be adapted from \cite{DuruGabrielIgel2017,Duru_exhype_2_2019}. We will not repeat it here.
 
 By theorem \ref{theo:discrete_case_no_damping} above, in the absence of the PML, $d_\xi = 0$, the semi-discrete approximation  \eqref{eq:disc_elastic_pml_1A}-\eqref{eq:disc_elastic_pml_3A} is asymptotically stable. However, this energy estimate \eqref{eq:discrete_energy_estimate_elastic_3D} is not valid when the PML is active, that is for  $d_\xi \ge 0$ and any $\xi = x, y, z$.
In order to demonstrate numerical stability when the PML is present $d_\xi \ne 0$, we will simplify further and prove the discrete versions of Theorems \ref{theo:PML_1D_Strip_Laplace}, \ref{theo:PML_2D_Edge_Laplace}, \ref{theo:PML_3D_Corner_Laplace} and \ref{Theo:estimate_PML}. If the numerical approximation is not provably stable for these model problems, there is no chance that it will  be stable for the general case.
 \subsection{Stability analysis for the discrete PML  problem}
As in the continuous case we are unable to derive  discrete energy estimates for the time dependent discrete PML problem \eqref{eq:disc_elastic_pml_1A}-\eqref{eq:disc_elastic_pml_3A}, when the damping is present $d_\xi \ge 0$. When damping is present $d_\xi \ge 0$, we will instead take the Laplace transform in time and derive  discrete energy estimates, analogous to \eqref{eq:energy_estimate_pml_laplace_strip}, \eqref{eq:energy_estimate_pml_laplace_edge}, \eqref{eq:energy_estimate_pml_laplace_corner},  and \eqref{eq:estimate_0_0}.  In particular, we will demonstrate the significant role played by the PML stabilizing parameter $\theta_{\xi}$ when the PML is present.

Taking the Laplace transform, in time, of the semi-discrete  problem  \eqref{eq:disc_elastic_pml_1A}-\eqref{eq:disc_elastic_pml_3A}  gives
{
\begin{equation}\label{eq:disc_elastic_pml_1A_Laplace}
\begin{split}
\mathbf{P}^{-1} s{  \widetilde{\mathbf{Q}}} = \sum_{\xi = x, y, z}\left[\mathbf{A}_{\xi}\mathbf{D}_{\xi} \widetilde{\mathbf{Q}}  - \mathbf{d}_{\xi}\mathbf{w}_{\xi} - \mathbf{H}_{\xi}^{-1}{\left({\mathbf{e}_{\xi}(-1)} \widetilde{\mathbf{FL}}_{\xi}+  {\mathbf{e}_{\xi}(1)} \widetilde{\mathbf{FR}}_{\xi} \right)}\right] + \widetilde{\mathbf{F}}_{Q} ,
  \end{split}
  \end{equation}
\begin{equation}\label{eq:disc_elastic_pml_3A_Laplace}
\begin{split}
 s\widetilde{\mathbf{w}}_{\xi}
 = \mathbf{A}_{\xi}\mathbf{D}_{\xi} \widetilde{\mathbf{Q}}  - \left(\mathbf{d}_{\xi} + \boldsymbol{\alpha}_{\xi}\right)\widetilde{\mathbf{w}}_{\xi}- \underbrace{\theta_{\xi}\mathbf{H}_{\xi}^{-1}{\left({\mathbf{e}_{\xi}(-1)} \widetilde{\mathbf{FL}}_{\xi}+  {\mathbf{e}_{\xi}(1)} \widetilde{\mathbf{FR}}_{\xi}  \right)}}_{\text{PML stabilizing flux fluctuations}} +   \widetilde{\mathbf{F}}_{w_\xi}.
  \end{split}
  \end{equation}
}

Next, we use \eqref{eq:disc_elastic_pml_3A_Laplace} and  eliminate the auxiliary variable $\widetilde{\mathbf{w}}_{\xi}$ from  \eqref{eq:disc_elastic_pml_1A_Laplace}. We have
{
\begin{equation}\label{eq:Laplace_disc_elastic_pml_1A}
\begin{split}
\mathbf{P}^{-1} \left(sS_x\right) \widetilde{\mathbf{Q}} = & \sum_{\xi = x, y, z} \frac{S_x}{S_\xi}\left[\frac{1}{2}\left(\mathbf{A}_{\xi}\mathbf{D}_{\xi}  - \mathbf{H}_{\xi}^{-1}\mathbf{A}_{\xi}\mathbf{D}_{\xi}^{T}\mathbf{H}_{\xi}    + \mathbf{H}_{\xi}^{-1}\mathbf{A}_{\xi}\left(\mathbf{B}_{\xi}\left(1,1\right)\right) -  \mathbf{B}_{\xi}\left(-1,-1\right)\right)\widetilde{\mathbf{Q}} \right]\\
- &\sum_{\xi = x, y, z} \frac{S_x}{S_\xi}\left[\mathbf{H}_{\xi}^{-1}{\left({\mathbf{e}_{\xi}(-1)} \widetilde{\mathbf{FL}}_{\xi} +  {\mathbf{e}_{\xi}(1)} \widetilde{\mathbf{FR}}_{\xi}  \right)} \right] + \mathbf{P}^{-1} \widetilde{\mathbf{F}} \\
+& \underbrace{\sum_{\xi = x, y, z} \frac{\mathbf{d}_\xi S_x \left(1-\theta_{\xi}\right)}{S_\xi \left(s + \alpha_{\xi}\right)}\mathbf{H}_{\xi}^{-1}{\left({\mathbf{e}_{\xi}(-1)} \widetilde{\mathbf{FL}}_{\xi} +  {\mathbf{e}_{\xi}(1)} \widetilde{\mathbf{FR}}_{\xi}  \right)}}_{\text{destabilizing PML  flux term}}. 
  \end{split}
  \end{equation}
}
Note that in \eqref{eq:Laplace_disc_elastic_pml_1A}, we  have made use of  the discrete integration property \eqref{eq:disc_sbp} of the spatial derivative operator $\mathbf{D}_{\xi}$.
The last term in the right hand side of \eqref{eq:Laplace_disc_elastic_pml_1A} is a destabilizing PML  flux term at element faces. The DG  is an element based method, and  these destabilizing flux terms appear almost everywhere in the PML, including at internal and external faces. These destabilizing terms can be eliminated using the parameter $\theta_\xi $.
 We set the penalty parameter $\theta_\xi =1$, extending the numerical implementation of the boundary conditions and inter-element conditions to the auxiliary differential equations. 
 The last  term in  the right hand side of \eqref{eq:Laplace_disc_elastic_pml_1A} vanishes, having
 {
 \begin{equation}\label{eq:Laplace_disc_elastic_pml_1A_0}
\begin{split}
\mathbf{P}^{-1} \left(sS_x\right) \widetilde{\mathbf{Q}} = & \sum_{\xi = x, y, z} \frac{S_x}{S_\xi}\left[\frac{1}{2}\left(\mathbf{A}_{\xi}\mathbf{D}_{\xi}  - \mathbf{H}_{\xi}^{-1}\mathbf{A}_{\xi}\mathbf{D}_{\xi}^{T}\mathbf{H}_{\xi}    + \mathbf{H}_{\xi}^{-1}\mathbf{A}_{\xi}\left(\mathbf{B}_{\xi}\left(1,1\right)\right) -  \mathbf{B}_{\xi}\left(-1,-1\right)\right)\widetilde{\mathbf{Q}} \right]\\
- &\sum_{\xi = x, y, z} \frac{S_x}{S_\xi}\left[\mathbf{H}_{\xi}^{-1}{\left({\mathbf{e}_{\xi}(-1)} \widetilde{\mathbf{FL}}_{\xi} +  {\mathbf{e}_{\xi}(1)} \widetilde{\mathbf{FR}}_{\xi}  \right)} \right] + \mathbf{P}^{-1} \widetilde{\mathbf{F}}.
  \end{split}
  \end{equation}
  }
In \eqref{eq:Laplace_disc_elastic_pml_1A_0}, note the uniform treatments of all element faces,  including internal and external element faces.

We approximate the scalar product \eqref{eq:elastic_scalarproduct_laplacePML} and the physical energy \eqref{eq:elastic_energy_laplacePML} by the quadrature rule \eqref{eq:quad_rule_3D}, having
\begin{align}\label{eq:disc_scalarproduct_energy_pml}
\Big\langle{\widetilde{\mathbf{Q}}, \widetilde{\mathbf{F}}}\Big\rangle_{h{P}} =  \frac{1}{2}\widetilde{\mathbf{Q}} ^{\dagger} {\mathbf{H}} {\mathbf{P}}^{-1}  \widetilde{\mathbf{F}} =  \frac{\Delta{x}}{2}\frac{\Delta{y}}{2}\frac{\Delta{z}}{2}\sum_{i=1}^{P+1}\sum_{j=1}^{P+1}\sum_{k=1}^{P+1} \frac{1}{2} \widetilde{\mathbf{Q}}_{ijk} ^{\dagger}  {\mathbf{P}}_{ijk}^{-1}    \widetilde{\mathbf{F}}_{ijk} h_ih_jh_k,
 \end{align}
 \begin{align}\label{eq:disc_elastic_energy_pml}
\|\widetilde{\mathbf{Q}}\left(\cdot, \cdot, \cdot, s\right)\|^2_{hP} = \Big\langle{\widetilde{\mathbf{Q}}, \widetilde{\mathbf{Q}}}\Big\rangle_{h{P}} =  \frac{\Delta{x}}{2}\frac{\Delta{y}}{2}\frac{\Delta{z}}{2}\sum_{i=1}^{P+1}\sum_{j=1}^{P+1}\sum_{k=1}^{P+1} \frac{1}{2} \widetilde{\mathbf{Q}}_{ijk} ^{\dagger}  {\mathbf{P}}_{ijk}^{-1}    \widetilde{\mathbf{Q}}_{ijk} h_ih_jh_k> 0. 
 \end{align}
As in the continuous setting,  we will consider first the 1D  PML strip problem, and proceed later to the edge and the corner regions.
 \subsubsection{The discrete PML strip problem }
 In \eqref{eq:Laplace_disc_elastic_pml_1A_0}, let  $d_x = d\ge 0$, $d_\xi \equiv 0$, $\mathbf{D}_\xi = \widetilde{\mathbf{FL}}_{\xi} = \widetilde{\mathbf{FR}}_{\xi}= 0$  for $\xi = y, z$. 
 We arrive at
 {
 \begin{equation}\label{eq:Laplace_disc_elastic_pml_1A_1D}
\begin{split}
\mathbf{P}^{-1} \left(sS_x\right) \widetilde{\mathbf{Q}} = &  \left[\frac{1}{2}\left(\mathbf{A}_{x}\mathbf{D}_{x}  - \mathbf{H}_{x}^{-1}\mathbf{A}_{x}\mathbf{D}_{x}^{T}\mathbf{H}_{x}    + \mathbf{H}_{x}^{-1}\mathbf{A}_{x}\left(\mathbf{B}_{\xi}\left(1,1\right)\right) -  \mathbf{B}_{x}\left(-1,-1\right)\right)\widetilde{\mathbf{Q}} \right]\\
- &\left[\mathbf{H}_{x}^{-1}{\left({\mathbf{e}_{x}(-1)} \widetilde{\mathbf{FL}}_{x} +  {\mathbf{e}_{\xi}(1)} \widetilde{\mathbf{FR}}_{x}  \right)} \right] + \mathbf{P}^{-1} \widetilde{\mathbf{F}}.
  \end{split}
  \end{equation}
  }
For the two elements model we introduce the external boundary terms 
 \begin{align*}
 \mathcal{BT}_s\left(\widehat{\widetilde{v}}_{\eta}^{-}, \widehat{\widetilde{T}}_{\eta}^{-}\right)  =  \frac{\Delta{y}}{2}  \frac{\Delta{z}}{2} \sum_{i = 1}^{P+1}\sum_{k = 1}^{P+1}\sum_{\eta = x,y,z}\left(\left(  \widehat{\widetilde{T}}_\eta^{-*} \widehat{\widetilde{v}}_\eta^{-}  \right)\Big|_{-1}\right)_{i k} h_{i} h_{k} \le 0,
 \end{align*}
 \begin{align*}
\mathcal{BT}_s\left(\widehat{\widetilde{v}}_{\eta}^{+}, \widehat{\widetilde{T}}_{\eta}^{+}\right)  = \frac{\Delta{y}}{2}  \frac{\Delta{z}}{2} \sum_{i = 1}^{P+1}\sum_{k = 1}^{P+1}\sum_{\eta = x,y,z}\left(\left( \widehat{\widetilde{T}}_\eta^{+*} \widehat{\widetilde{v}}_\eta^{+}  \right)\Big|_{ 1}\right)_{i k} h_{i} h_{k}  \le 0,
 \end{align*}
  the interface term
  \begin{align*}
  \mathcal{IT}_s\left(\widehat{\widetilde{v}}^{\pm}, \widehat{\widetilde{T}}^{\pm} \right) = - \frac{\Delta{y}}{2}  \frac{\Delta{z}}{2} \sum_{i = 1}^{P+1}\sum_{k = 1}^{P+1}\sum_{\eta = x,y,z}\left( {\widehat{\widetilde{T}}}_\eta^*  \lJump{{\widehat{\widetilde{v}}}_\eta \rJump} \right)_{i k} h_{i} h_{k} \equiv 0,
\end{align*}
   and the fluctuation term
    \begin{align*}
  \mathcal{F}_{luc}\left({\widetilde{G}},  Z\right) = - \frac{\Delta{y}}{2}  \frac{\Delta{z}}{2} \sum_{\eta = x, y, z}\sum_{i = 1}^{P+1}\sum_{k = 1}^{P+1}\left(\left(\frac{1}{Z_{\eta}}|{\widetilde{G}}_\eta |^2\right)\Big|_{-1} + \left(\frac{1}{Z_{\eta}}|{\widetilde{G}}_\eta |^2\right)\Big|_{ 1}  \right)_{i, k} h_{i} h_{k}  \le 0.
  \end{align*}
Here $v^*$ denotes the complex conjugate of $v$, and the surface integrals have been approximated by quadrature rules.
For the surface terms, we will also use the identity
{
\small
\begin{align}\label{eq:identity_001}
&\widetilde{\mathbf{Q}}^\dagger\mathbf{H}\mathbf{H}_{\xi}^{-1}\left(\frac{1}{2}\mathbf{A}_{\xi}\left( \mathbf{B}_{\xi}\left(1,1\right) -  \mathbf{B}_{\xi}\left(-1,-1\right)\right)\widetilde{\mathbf{Q}} -\left({\mathbf{e}_{\xi}(-1)} \widetilde{\mathbf{FL}}_{\xi} +  {\mathbf{e}_{\xi}(1)} \widetilde{\mathbf{FR}}_{\xi}  \right) \right) \\
 &=- \frac{\Delta{x}}{2} \frac{\Delta{y}}{2}  \frac{\Delta{z}}{2} \frac{2}{\Delta{\xi}} \sum_{i=1}^{P+1}\sum_{k= 1}^{P+1} \left(\sum_{\eta = x,y,z} \left(\frac{1}{Z_\eta  }|\widetilde{G}_\eta |^2  - \widehat{\widetilde{T}}_\eta^* \widehat{\widetilde{v}}_\eta \right)\Big|_{ 1} +   \sum_{\eta = x,y,z} \left(\frac{1}{Z_\eta  }|\widetilde{G}_\eta |^2  + \widehat{\widetilde{T}}_\eta^* \widehat{\widetilde{v}}_\eta \right)\Big|_{-1} \right)_{ik} h_ih_k. \nonumber
  \end{align}
  }

We can prove the discrete equivalence of Theorem \ref{theo:PML_1D_Strip_Laplace}.
 \begin{theorem}
 Consider the  semi-discrete PML equation  in the Laplace space \eqref{eq:Laplace_disc_elastic_pml_1A_1D}, with $d_x = d\ge 0$, $\alpha_x = \alpha \ge 0$, and $\Re(s) \ge a> 0$.  We have 
 \begin{equation}
 \begin{split}
& \|\sqrt{\Re(s S_x)}\widetilde{\mathbf{Q}}^{-}\left(s\right)\|^2_{hP} + \|\sqrt{\Re(s S_x)}\widetilde{\mathbf{Q}}^{+}\left(s\right)\|^2_{hP}  \le  \|\widetilde{\mathbf{Q}}^{-}\left(s\right)\|_{hP} \|\widetilde{\mathbf{F}}^{-}\left( s\right)\|_{hP}  + \|\widetilde{\mathbf{Q}}^{+}\left(s\right)\|_{hP} \|\widetilde{\mathbf{F}}^{+}\left(s\right)\|_{hP}+  \widetilde{\mathrm{BT}}_h, \\
& \widetilde{\mathrm{BT}}_h =  \mathcal{BT}_s\left(\widehat{\widetilde{v}}_{\eta}^{-}, \widehat{\widetilde{T}}_{\eta}^{-}\right) + \mathcal{BT}_s\left(\widehat{\widetilde{v}}_{\eta}^{+}, \widehat{\widetilde{T}}_{\eta}^{+}\right)  +  \mathcal{IT}_s\left(\widehat{\widetilde{v}}^{\pm}, \widehat{\widetilde{T}}^{\pm} \right) +  \mathcal{F}_{luc}\left({\widetilde{G}}^{-},  Z^{-}\right)  + \mathcal{F}_{luc}\left({\widetilde{G}}^{+},  Z^{+}\right) \le 0,
\\
& \Re(s S_x) \ge a > 0,
\end{split}
\end{equation}
where  the surface terms $\widetilde{\mathrm{BT}}_h$ are negative semi-definite.
 \end{theorem}
\begin{proof}
From the left multiply \eqref{eq:Laplace_disc_elastic_pml_1A_1D} with $\widetilde{\mathbf{Q}}^\dagger {\mathbf{H}}$, we have
  {
 \small
 \begin{equation}\label{eq:Laplace_disc_elastic_pml_1A_1D_proof_1}
\begin{split}
\left(sS_x\right)\widetilde{\mathbf{Q}}^\dagger {\mathbf{H}}\mathbf{P}^{-1}  \widetilde{\mathbf{Q}} = &  \frac{1}{2}\widetilde{\mathbf{Q}}^\dagger \left(\mathbf{A}_{x}\left(\mathbf{H}\mathbf{D}_{x}\right) - \left(\mathbf{H}\mathbf{D}_{x}\right)^T\mathbf{A}_{x} \right)\widetilde{\mathbf{Q}}  + \widetilde{\mathbf{Q}}^\dagger {\mathbf{H}}\mathbf{P}^{-1}  \widetilde{\mathbf{F}} \\
+ & \widetilde{\mathbf{Q}}^\dagger\mathbf{H}\mathbf{H}_{x}^{-1}\left(\frac{1}{2}\mathbf{A}_{x}\left( \mathbf{B}_{x}\left(1,1\right) -  \mathbf{B}_{x}\left(-1,-1\right)\right)\widetilde{\mathbf{Q}} -\left({\mathbf{e}_{x}(-1)} \widetilde{\mathbf{FL}}_{x} +  {\mathbf{e}_{x}(1)} \widetilde{\mathbf{FR}}_{x}  \right) \right)
%
  \end{split}
  \end{equation}
  }
Using the identity \eqref{eq:identity_001} gives 
  {
 \small
 \begin{equation}\label{eq:Laplace_disc_elastic_pml_1A_1D_proof_2}
\begin{split}
&\left(sS_x\right)\widetilde{\mathbf{Q}}^\dagger {\mathbf{H}}\mathbf{P}^{-1}  \widetilde{\mathbf{Q}} =   \frac{1}{2}\widetilde{\mathbf{Q}}^\dagger \left(\mathbf{A}_{x}\left(\mathbf{H}\mathbf{D}_{x}\right) - \left(\mathbf{H}\mathbf{D}_{x}\right)^T\mathbf{A}_{x} \right)\widetilde{\mathbf{Q}} +  \widetilde{\mathbf{Q}}^\dagger {\mathbf{H}}\mathbf{P}^{-1}  \widetilde{\mathbf{F}}  \\
- &\frac{\Delta{y}}{2}  \frac{\Delta{z}}{2} \sum_{i=1}^{P+1}\sum_{k= 1}^{P+1}\  \left(\sum_{\eta = x,y,z} \left(\frac{1}{Z_\eta  }|\widetilde{G}_\eta |^2  - \widehat{\widetilde{T}}_\eta^* \widehat{\widetilde{v}}_\eta \right)\Big|_{ 1} +   \sum_{\eta = x,y,z} \left(\frac{1}{Z_\eta  }|\widetilde{G}_\eta |^2  + \widehat{\widetilde{T}}_\eta^* \widehat{\widetilde{v}}_\eta \right)\Big|_{-1}\right)_{ik} h_ih_k,
  \end{split}
  \end{equation}
  }
  Next we add \eqref{eq:Laplace_disc_elastic_pml_1A_1D_proof_2} to its complex conjugate, and the spatial derivative terms vanish, we have
   {
 \small
 \begin{equation}\label{eq:Laplace_disc_elastic_pml_1A_1D_proof_3}
\begin{split}
\Re\left(sS_x\right)\widetilde{\mathbf{Q}}^\dagger {\mathbf{H}}\mathbf{P}^{-1}  \widetilde{\mathbf{Q}} &= - \frac{\Delta{y}}{2}  \frac{\Delta{z}}{2} \sum_{i=1}^{P+1}\sum_{k= 1}^{P+1}\left(\sum_{\eta = x,y,z} \left(\frac{1}{Z_\eta  }|\widetilde{G}_\eta |^2  - \widehat{\widetilde{T}}_\eta^* \widehat{\widetilde{v}}_\eta \right)\Big|_{ 1} +   \sum_{\eta = x,y,z} \left(\frac{1}{Z_\eta  }|\widetilde{G}_\eta |^2  + \widehat{\widetilde{T}}_\eta^* \widehat{\widetilde{v}}_\eta \right)\Big|_{-1}\right)_{ik} h_ih_k \\
&+ \frac{1}{2}\left(  \widetilde{\mathbf{Q}}^\dagger {\mathbf{H}}\mathbf{P}^{-1}  \widetilde{\mathbf{F}}  +  \widetilde{\mathbf{F}}^\dagger {\mathbf{H}}\mathbf{P}^{-1}  \widetilde{\mathbf{Q}} \right).
%
  \end{split}
  \end{equation}
  }
  Using Cauchy-Schwarz inequality for the source term in \eqref{eq:Laplace_disc_elastic_pml_1A_1D_proof_3} and collecting contributions from both sides of the elements completes the proof.
\end{proof}
\subsubsection{The discrete PML  edge problem}
Consider now   the discrete  PML  \eqref{eq:Laplace_disc_elastic_pml_1A_0} in the $xy$-edge region.
 In \eqref{eq:Laplace_disc_elastic_pml_1A_0}, let  $d_x = d_y = d\ge >0$, $d_z \equiv 0$, $\mathbf{D}_z = \widetilde{\mathbf{FL}}_{z} = 0$. 
 {
 \begin{equation}\label{eq:Laplace_disc_elastic_pml_1A_2D}
\begin{split}
\mathbf{P}^{-1} \left(sS_x\right) \widetilde{\mathbf{Q}} = &\sum_{\xi = x, y} \left[\frac{1}{2}\left(\mathbf{A}_{\xi}\mathbf{D}_{\xi}  - \mathbf{H}_{\xi}^{-1}\mathbf{A}_{\xi}\mathbf{D}_{\xi}^{T}\mathbf{H}_{\xi}    + \mathbf{H}_{\xi}^{-1}\mathbf{A}_{\xi}\left(\mathbf{B}_{\xi}\left(1,1\right)\right) -  \mathbf{B}_{\xi}\left(-1,-1\right)\right)\widetilde{\mathbf{Q}} \right]\\
- &\sum_{\xi = x, y} \left[\mathbf{H}_{\xi}^{-1}{\left({\mathbf{e}_{\xi}(-1)} \widetilde{\mathbf{FL}}_{\xi} +  {\mathbf{e}_{\xi}(1)} \widetilde{\mathbf{FR}}_{\xi}  \right)} \right] + \mathbf{P}^{-1} \widetilde{\mathbf{F}}
  \end{split}
  \end{equation}
  }

We formulate the discrete equivalence of Theorem \ref{theo:PML_2D_Edge_Laplace}.
 \begin{theorem}\label{theo:PML_2D_Edge_Laplace_Disc}
 Consider the  semi-discrete PML equation  in the Laplace space \eqref{eq:Laplace_disc_elastic_pml_1A_2D}, with $d_x = d_y = d \ge 0$, $\alpha_x = \alpha_y = \alpha \ge 0$ and $\Re(s) \ge a> 0$.  We have 
 {
 \small
 \begin{equation}
 \begin{split}
& \|\sqrt{\Re(s S_x)}\widetilde{\mathbf{Q}}^{-}\left(s\right)\|^2_{hP} + \|\sqrt{\Re(s S_x)}\widetilde{\mathbf{Q}}^{+}\left(s\right)\|^2_{hP}  \le  \|\widetilde{\mathbf{Q}}^{-}\left(s\right)\|_{hP} \|\widetilde{\mathbf{F}}^{-}\left( s\right)\|_{hP}  + \|\widetilde{\mathbf{Q}}^{+}\left(s\right)\|_{hP} \|\widetilde{\mathbf{F}}^{+}\left(s\right)\|_{hP}+  \widetilde{\mathrm{BT}}_h, \quad \Re(s S_x) \ge a \\
& \widetilde{\mathrm{BT}}_h =  \mathcal{BT}_s\left(\widehat{\widetilde{v}}_{\eta}^{-}, \widehat{\widetilde{T}}_{\eta}^{-}\right) + \mathcal{BT}_s\left(\widehat{\widetilde{v}}_{\eta}^{+}, \widehat{\widetilde{T}}_{\eta}^{+}\right)  +  \mathcal{IT}_s\left(\widehat{\widetilde{v}}^{\pm}, \widehat{\widetilde{T}}^{\pm} \right) +  \mathcal{F}_{luc}\left({\widetilde{G}}^{-},  Z^{-}\right)  + \mathcal{F}_{luc}\left({\widetilde{G}}^{+},  Z^{+}\right) \le 0,
\end{split}
\end{equation}
}
where the surface terms $\widetilde{\mathrm{BT}}_h$ are negative semi-definite.
%
 \end{theorem}
The proof of Theorem \ref {theo:PML_2D_Edge_Laplace_Disc} has been moved to  Appendix \ref{sec:proof_theorem8}.
 \subsubsection{The discrete PML  corner problem}
Consider now  a the discrete  PML   \eqref{eq:Laplace_disc_elastic_pml_1A_0} in the corner region, where all damping functions are nonzero, $d_\xi  = d \ge 0$ for all $\xi = x, y,z$.
For this case the  PML complex metrics are identical $S_y = S_x = S_z$.  Thus $S_x/S_\xi = 1$, and from \eqref{eq:Laplace_disc_elastic_pml_1A_0}  we have
 {
 \begin{equation}\label{eq:Laplace_disc_elastic_pml_1A_3D}
\begin{split}
\mathbf{P}^{-1} \left(sS_x\right) \widetilde{\mathbf{Q}} = &\sum_{\xi = x, y, z} \left[\frac{1}{2}\left(\mathbf{A}_{\xi}\mathbf{D}_{\xi}  - \mathbf{H}_{\xi}^{-1}\mathbf{A}_{\xi}\mathbf{D}_{\xi}^{T}\mathbf{H}_{\xi}    + \mathbf{H}_{\xi}^{-1}\mathbf{A}_{\xi}\left(\mathbf{B}_{\xi}\left(1,1\right)\right) -  \mathbf{B}_{\xi}\left(-1,-1\right)\right)\widetilde{\mathbf{Q}} \right]\\
- &\sum_{\xi = x, y, z} \left[\mathbf{H}_{\xi}^{-1}{\left({\mathbf{e}_{\xi}(-1)} \widetilde{\mathbf{FL}}_{\xi} +  {\mathbf{e}_{\xi}(1)} \widetilde{\mathbf{FR}}_{\xi}  \right)} \right] + \mathbf{P}^{-1} \widetilde{\mathbf{F}}.
  \end{split}
  \end{equation}
  }

We will now state the discrete equivalence of  Theorem \ref{theo:PML_3D_Corner_Laplace}.
 \begin{theorem}\label{theo:PML_3D_Corner_Laplace_Disc}
 Consider the  semi-discrete PML equation  in the Laplace space \eqref{eq:Laplace_disc_elastic_pml_1A_3D}, with $d_x = d_y = d_z = d\ge 0$,  $\alpha_x = \alpha_y = \alpha_z = \alpha  \ge 0$ and $\Re(s) \ge a > 0$.  We have 
 $
 \Re(s S_x)\ge a > 0
 $
 and
 {
 \small
 \begin{equation}
 \begin{split}
& \|\sqrt{\Re(s S_x)}\widetilde{\mathbf{Q}}^{-}\left(s\right)\|^2_{hP} + \|\sqrt{\Re(s S_x)}\widetilde{\mathbf{Q}}^{+}\left(s\right)\|^2_{hP}  \le  \|\widetilde{\mathbf{Q}}^{-}\left(s\right)\|_{hP} \|\widetilde{\mathbf{F}}^{-}\left( s\right)\|_{hP}  + \|\widetilde{\mathbf{Q}}^{+}\left(s\right)\|_{hP} \|\widetilde{\mathbf{F}}^{+}\left(s\right)\|_{hP}+  \widetilde{\mathrm{BT}}_h,  \\
& \widetilde{\mathrm{BT}}_h =  \mathcal{BT}_s\left(\widehat{\widetilde{v}}_{\eta}^{-}, \widehat{\widetilde{T}}_{\eta}^{-}\right) + \mathcal{BT}_s\left(\widehat{\widetilde{v}}_{\eta}^{+}, \widehat{\widetilde{T}}_{\eta}^{+}\right)  +  \mathcal{IT}_s\left(\widehat{\widetilde{v}}^{\pm}, \widehat{\widetilde{T}}^{\pm} \right) +  \mathcal{F}_{luc}\left({\widetilde{G}}^{-},  Z^{-}\right)  + \mathcal{F}_{luc}\left({\widetilde{G}}^{+},  Z^{+}\right) \le 0,
\end{split}
\end{equation}
}
where the surface terms are negative semi-definite.
 \end{theorem}
The proof of Theorem \ref {theo:PML_3D_Corner_Laplace_Disc} has been moved to Appendix \ref{sec:proof_theorem9}.
 
 We will now formulate a discrete equivalence of Theorem \ref{Theo:estimate_PML} as the corollary.
\begin{corollary}\label{corollary:estimate_PML}
Consider the semi-discrete energy estimate in the Laplace space
\begin{equation}\label{eq:disc_energy_estimate_pml_laplace_0}
\begin{split}
& \|\widetilde{\mathbf{Q}} (\cdot ,\cdot ,\cdot , s)\|_{hP}^2 \le \frac{1}{a }\|\widetilde{\mathbf{Q}} (\cdot ,\cdot ,\cdot , s)\|_{hP} \|\widetilde{\mathbf{F}} (\cdot ,\cdot ,\cdot , s)\|_{P}, \quad \Re(s)\ge a>0.
\end{split}
\end{equation}
For any $a> 0$ and $T>0$ we have
  \begin{align}\label{eq:estimate_0_0_disc}
  \int_0^{T} e^{-2at}  \|{\mathbf{Q}} (\cdot ,\cdot ,\cdot , t)\|_{hP}^2 dt \le  \frac{1}{a^2}\int_0^{T} e^{-2at} \left( \| {\mathbf{F}} (\cdot ,\cdot ,\cdot , t)\|_{hP}^2\right) dt.
  \end{align}
\end{corollary}
The proof of Corollary \ref{corollary:estimate_PML} can be easily adapted from the proof of Theorem \ref{Theo:estimate_PML}.

 Numerical experiments presented in the next section will demonstrate  further the  importance  of the analysis and theory derived in this section.

 \section{Numerical experiments}
In this section, we present numerical experiments, verifying the theoretical results and exploring further the numerical properties of the discrete PML. In particular,  we will quantify numerical errors introduced by discretizing the PML as well as verify the stability analysis of the last section. We will consider first a 2D problem and perform detailed numerical experiments demonstrating numerical stability, and showing convergence  for h- and p-refinement.   We will  proceed later to  3D numerical simulations of linear elastic waves in an unbounded and semi-bounded domain.

\subsection{PML for 2D  elastodynamics}
We consider the 2D problem with $\partial / \partial z = 0$ and $d_z = 0$. Note that  the auxiliary variable $\mathbf{w}_z = 0$, vanishes identically.
The domain is rectangular isotropic elastic solid  with the dimension $\left(x, y\right) = [-50~\ \text{km}, 50~\ \text{km}]\times[0, 50~\ \text{km}]$. 
We consider a homogeneous isotropic elastic medium with $\rho = 2.7$ g/cm$^3$, $c_p = 6.0 $ km/s, and $c_s = 3.464 $ km/s, where $c_p$ is the p-wave speed and $c_s$ is the shear wave speed.
At the top  surface of the elastic solid, at $y = 0$, we set the free surface boundary conditions, with  $\gamma_\eta = 1$ in \eqref{eq:BC_General2}. 

For the velocity fields,  we set the smooth initial condition
\begin{align}
v_x = v_y = e^{-\log\left({2}\right)\frac{x^2 + (y-25)^2}{9}},
\end{align}
 and use zero initial condition for the stress fields and all auxiliary variables.
The damping profile is a cubic monomial
\begin{equation}\label{eq:damping_func}
\begin{split}
&d\left(x\right) = \left \{
\begin{array}{rl}
0 \quad {}  \quad {}& \text{if} \quad |x| \le 50,\\
d_0\Big(\frac{|x|-50}{\delta}\Big)^3  & \text{if}  \quad |x| \ge 50,
\end{array} \right.
\end{split}
\end{equation}
where $d_0\ge 0$ is the damping strength. We will use 
\begin{equation}\label{eq:damping_coef}
d_0 = \frac{4c_p}{2\delta}\ln{\frac{1}{{tol}}},
\end{equation}
where $c_p$ denotes the P-wave speed, and  ${tol}$ is the magnitude of the relative PML error \cite{SjPe}. 

We use nodal Lagrange polynomial basis of  degree $ P \le 12 $, and GLL quadrature nodes, which gives $P+1$ optimal order of accuracy.  Numerical solutions are evolved in time using the high order  ADER  scheme \cite{PeltiesdelaPuenteAmpueroBrietzkeKaser2012,DumbserKaser2006}  of the same order of accuracy as the spatial discretization. The time step  is given by  

\begin{align}
\Delta{t} = \frac{\mathrm{CFL}}{\sqrt{c_p^2 + c_s^2}\left(2P+1\right)}\min\left(\Delta{x}, \Delta{y} \right)
\end{align}
 with $\mathrm{CFL} = 0.9$.
We will consider the  vertical PML strip problem and demonstrate numerical stability. 
We will chose parameters such that the total PML error converges to zero as we perform both $p$-refinement and $h$-refinement. Later  we simulate a half-plane problem surrounded by the PML. This  situation involves both the vertical (with $d_x(x) \ge 0$, $d_y(y) = 0$) and horizontal (with $d_x(x) = 0$, $d_y(y) \ge 0$) PML layers, and  PML corners  (with $d_x(x) \ge 0$, $d_y(y) \ge 0$) where both layers are simultaneously active.

\subsubsection{The vertical strip PML problem}
We consider the vertical layer  with $d_x(x) \ge 0$, $d_y(y) = 0$.
That is, in the $x$-direction we  introduce two additional layers, having $50\le |x| \le 50 + \delta$ in which the PML equations are solved. 
We terminate PML edges with the characteristic boundary condition,  with vanishing reflection coefficients $\gamma_\eta = 0$ in \eqref{eq:BC_General2}.
The bottom  boundary of the solid, at $y = 50$ km, is a fictional wall with   $\gamma_\eta = 0$ in \eqref{eq:BC_General2}, that is the absorbing boundary condition. 

We will investigate numerically the stability of the PML.    In the coming experiments we set the magnitude of the relative PML error $\mathrm{tol} = 10^{-6}$. We will use the complex frequency shift  $\alpha = 0.15$ in this experiment.  In a practical setup the PML width is considerably thin, we therefore use $\delta = 10$ km,  leading to the damping strength $d_0 = 16.58$. The width of the PML $\delta =10$ km corresponds to $10\%$ of the width of the computational domain. To verify numerical stability we will  consider separately $\theta_x = 0$ and $\theta_x =1$  in the experiments.

We set a uniform spatial step $\Delta{x} = \Delta{y} = 5$ km in both $x$ and $y$ axes and consider polynomial degree $p = 5$.  We run the simulation until $t = 100$ s.
Snapshots of the solution at the initial times are shown in Figure \ref{fig:snapshots_strip} showing how the initial Gaussian perturbation spreads, its interactions with the free-surface boundary conditions and the absorption properties of the PML. In  Figure \ref{fig:vel_t100} the snapshots  at the final time are shown for $\theta_x = 0$ and $\theta_x = 1$. Note that without the PML stabilizing parameter, $\theta_x = 0$, at $t = 50$ s the solutions start grow in the PML. The numerical growth is catastrophic, see Figure \ref{fig:Time_series_strip}. By the final time $t = 100$ s, the growth has spread everywhere in the domain polluting numerical simulations. However, with the PML stabilizing term $\theta_x = 1$, we regain numerical stability, see also  Figures \ref{fig:vel_t100} and \ref{fig:Time_series_strip}. This is consistent with the theory developed in this paper and the results presented in \cite{DuruKozdonKreiss2016,DuruGabKreiss2019} for the SBP finite difference method. We conclude that the PML stabilizing term $\theta_x = 1$ is necessary for the numerical stability of DG approximations of the  PML.

\begin{figure} [h!]
 \centering
\begin{subfigure}
    \centering
{\includegraphics[width=0.49\textwidth]{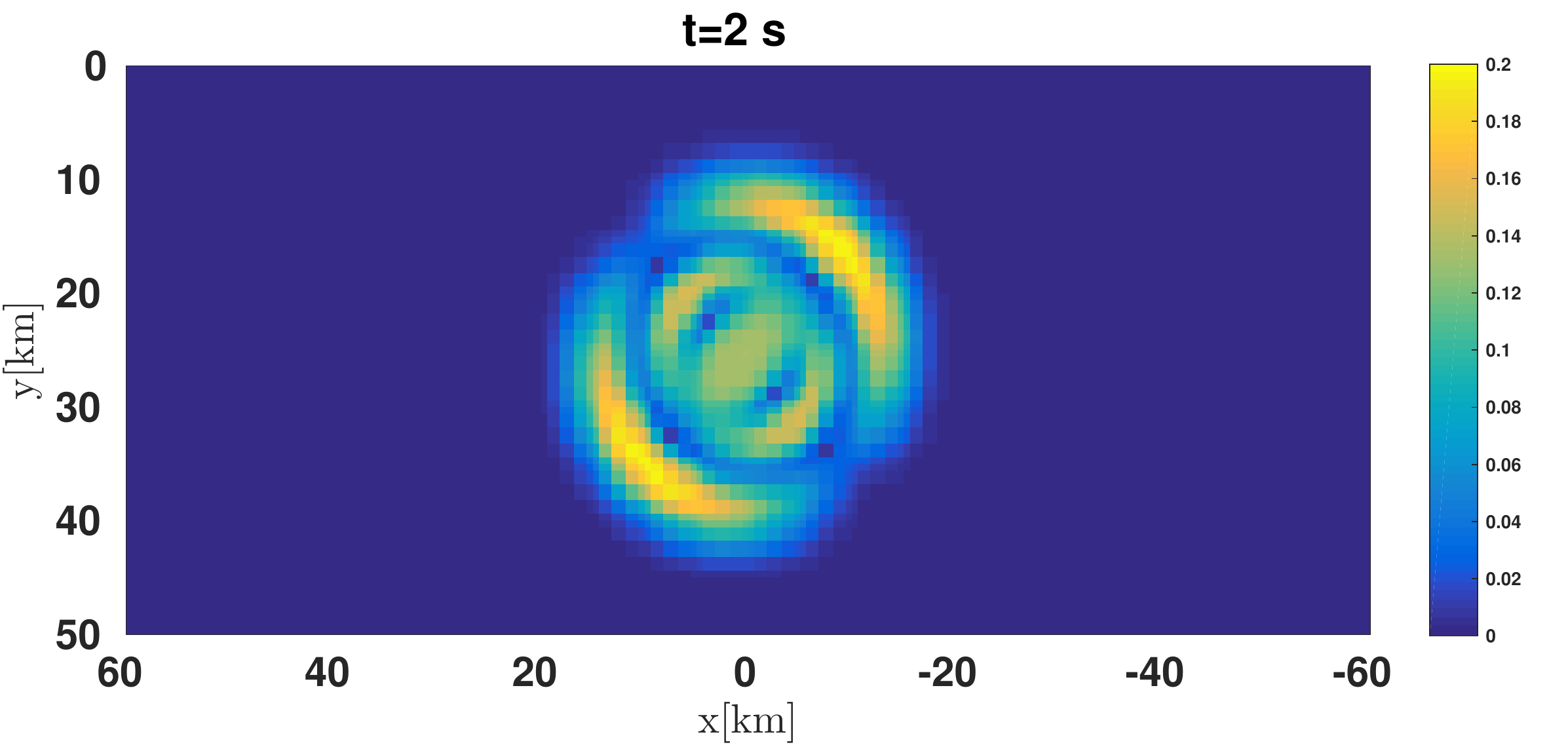}}%
\hspace{0.0cm}%
{\includegraphics[width=0.49\textwidth]{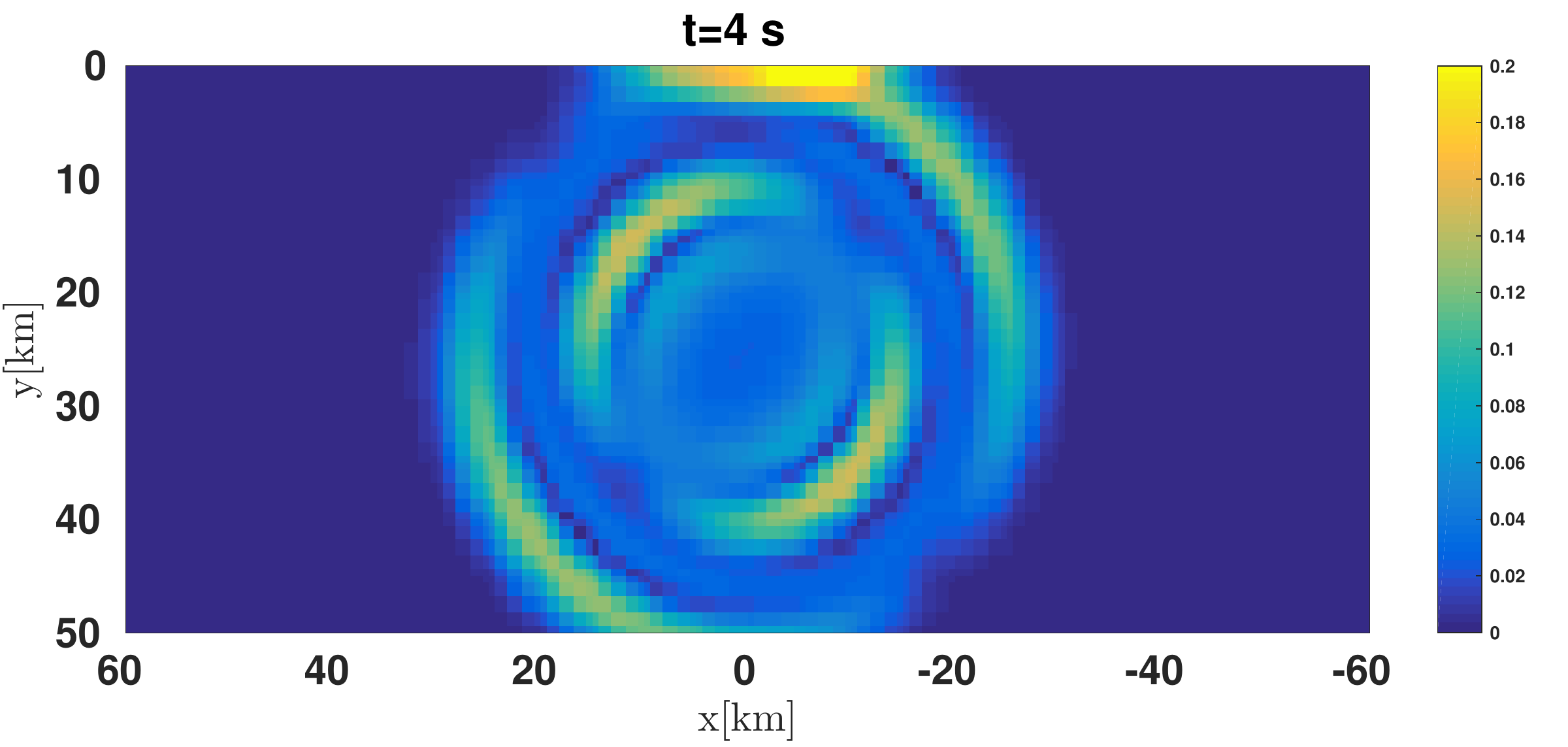}}%
\hspace{0.0cm}%
{\includegraphics[width=0.49\textwidth]{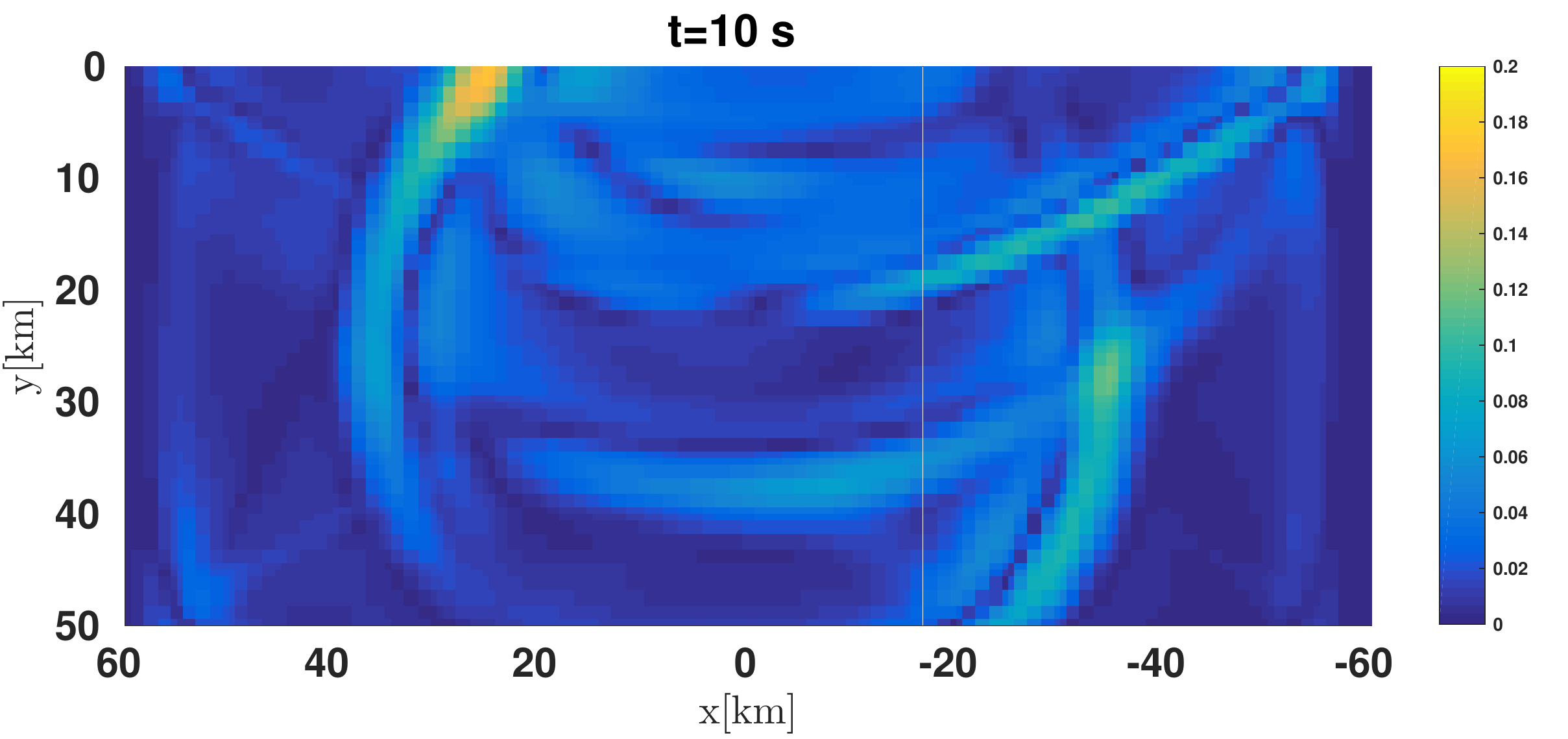}}%
\hspace{0.0cm}%
{\includegraphics[width=0.49\textwidth]{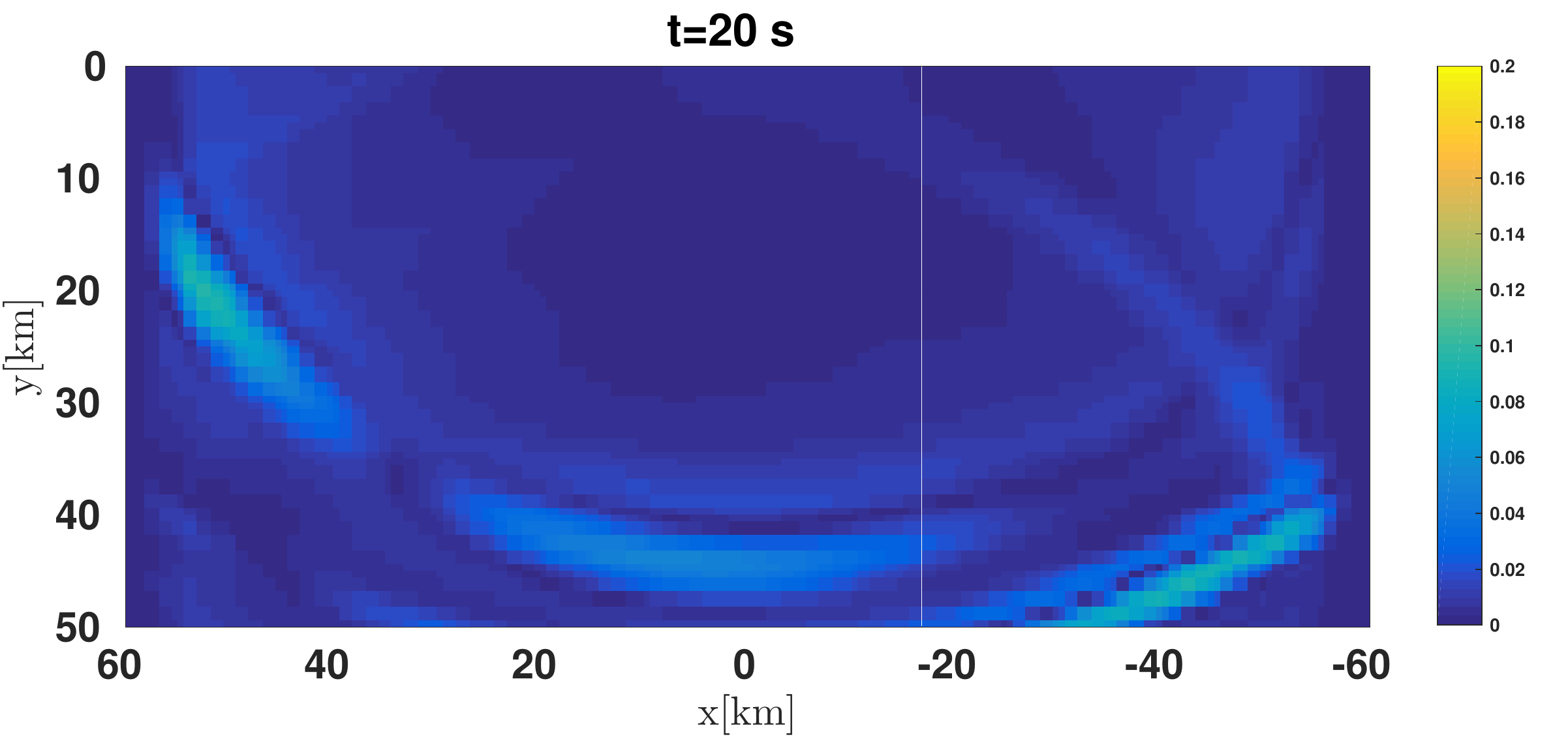}}%
     \end{subfigure}
    \caption{Snapshots of the absolute velocity $\sqrt{v_x^2 + v_y^2}$ at $t = 2, 4, 10$ and $20$ s.}
   \label{fig:snapshots_strip}
\end{figure}

\begin{figure}[h!]
\begin{subfigure}
    \centering
\stackunder[5pt]{\includegraphics[width=0.49\textwidth]{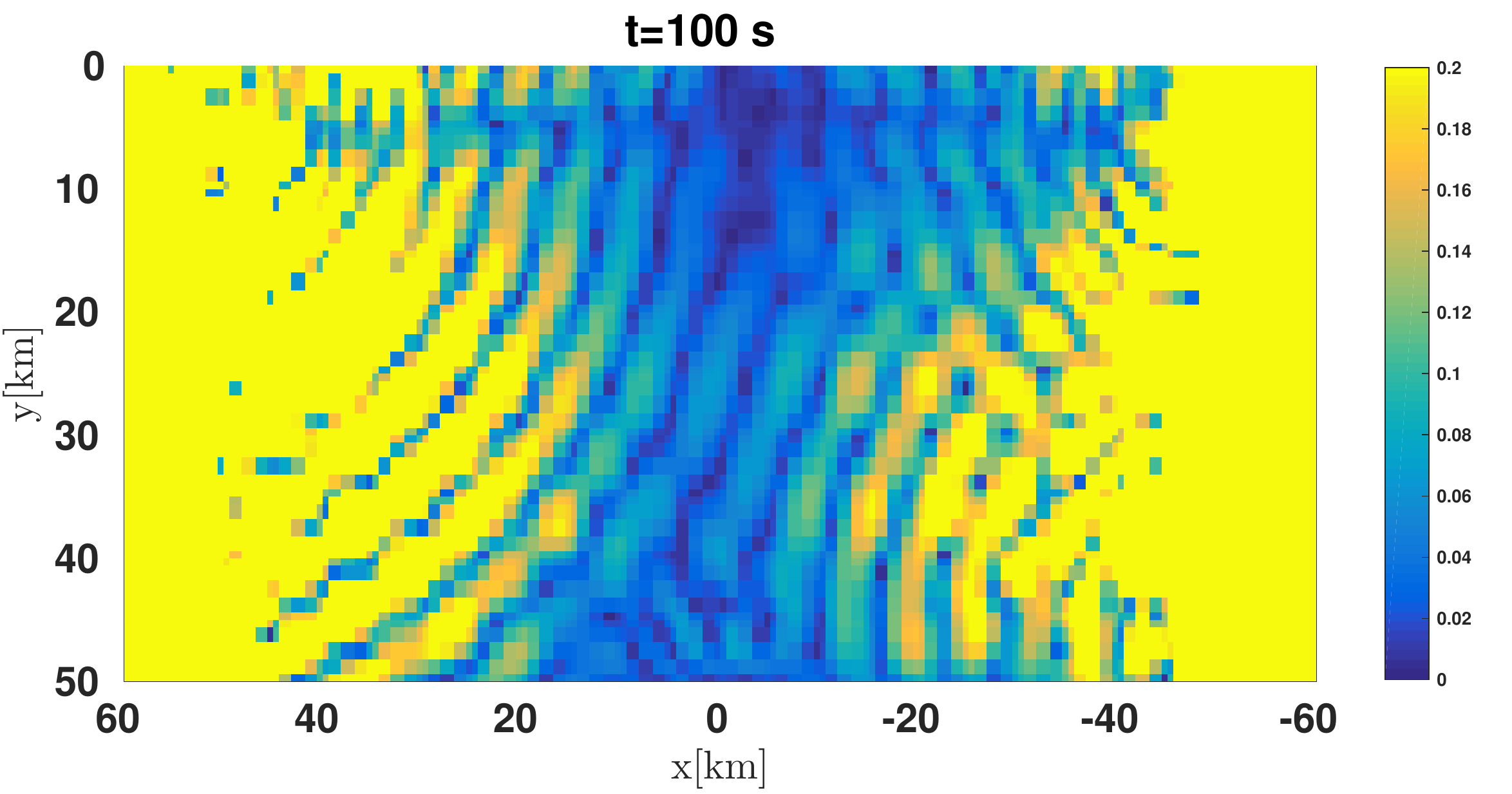}}{$\theta_x = 0$.}%
\hspace{0.0cm}%
\stackunder[5pt]{\includegraphics[width=0.49\textwidth]{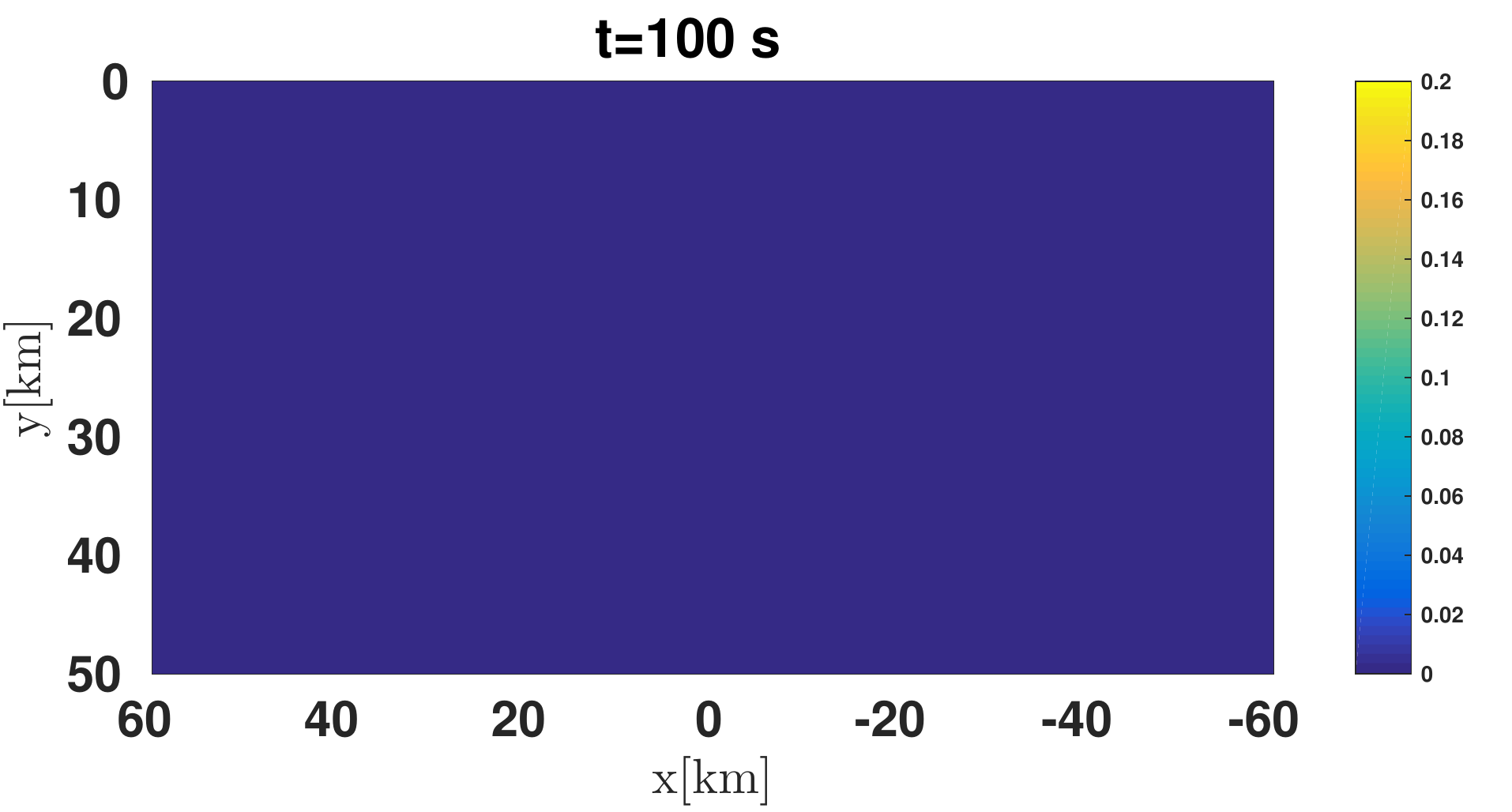}}{$\theta_x = 1$.}%
     \end{subfigure}
    \caption{Snapshots of the absolute velocity $\sqrt{v_x^2 + v_y^2}$ at $t = 100$ s.}
    \label{fig:vel_t100}
\end{figure}

\begin{figure} [h!]
 \centering
\includegraphics[width=0.49\linewidth, height=0.475\linewidth]{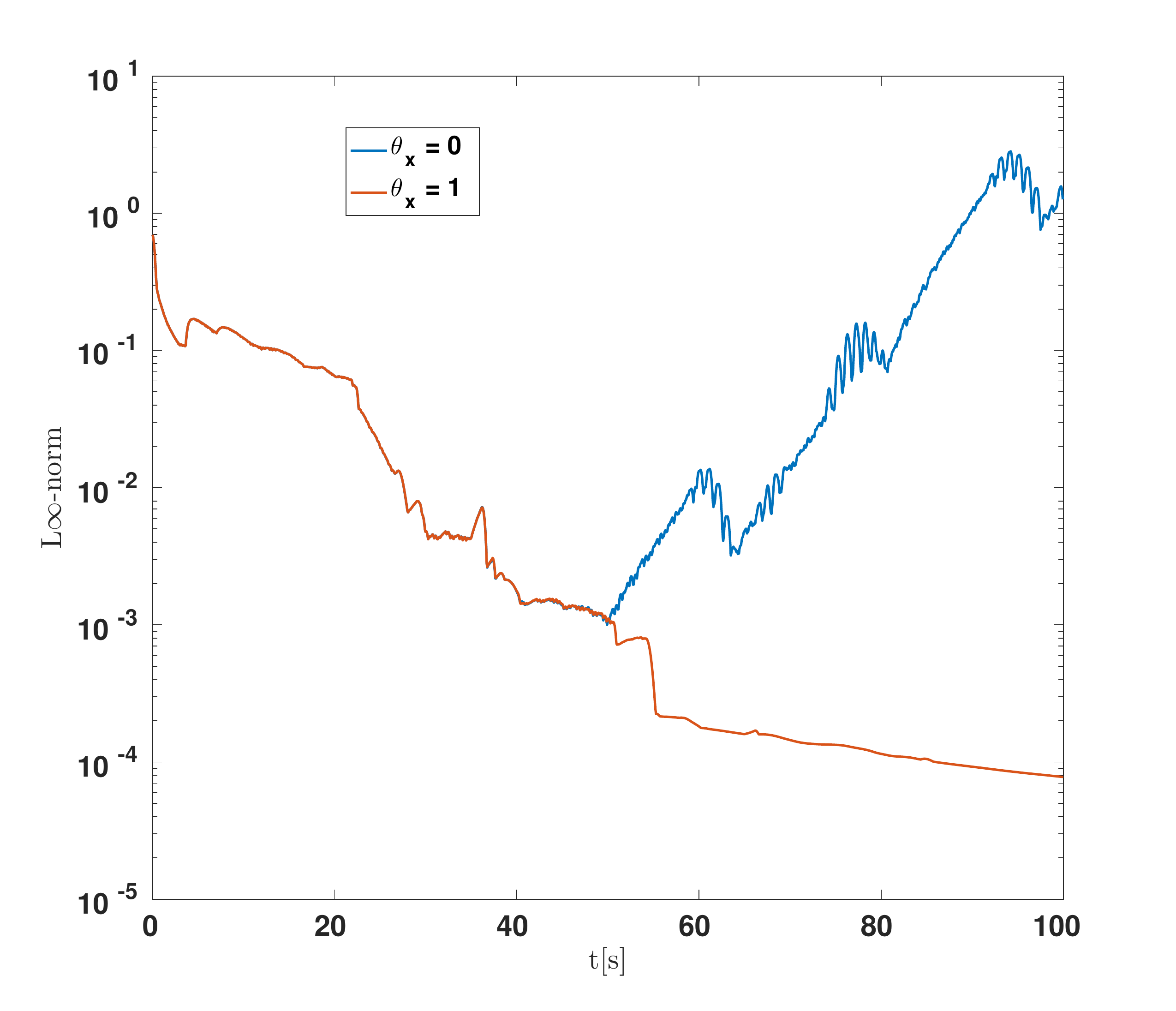}
 \caption{\textit{Time series of the L$_{\infty}$ norm of the velocity vector for the PML.}}
 \label{fig:Time_series_strip}
\end{figure}

 Next we verify numerical accuracy. We will show that we can choose PML parameters such that, for a fixed PML width $\delta = 10 $ km, the total PML error converges to zero at the rate of the underlying numerical method. We evolve the solutions until $t = 20$ s so that  the $p$-wave and the $s$-wave have traveled to through the PML and their reflections from  the external edge of the PML felt inside the simulation domain, $|x| \le 50$.  In order to evaluate PML errors, we compute a reference solution in  larger domain without the PML. We have chosen the reference domain to be large enough so that artificial reflections are not yet felt inside the simulation domain. By comparing the reference solution and the PML solution, measure in the maximum norm, in the interior of the  domain we obtain accurate measurement of the total  PML errors. 

We begin with $p$-refinement,  that is we fix the element size $\Delta{x} = \Delta{y}$ and increase polynomial degree, $p = 2, 4, 6, 8, 10$. The magnitude of the relative PML error is $\emph{tol} = \left[50(P+1)/\Delta{x}\right]^{-(P+1)}$, where $50(P+1)/\Delta{x}$ is the number of degrees of freedom spanning the width of the computational domain, $0 \le y \le 50$.  Note that  the magnitude of the relative PML error $\emph{tol}$ converges spectrally to zero as the polynomial degree $P$ increases.  All other discretization parameters are fixed with the PML stabilizing term $\theta_x = 1$ and element size $\Delta{x} = 5$ km. The PML spans  only 2 DG elements.   Numerical results are shown in Figure \ref{fig:Time_series_error_strip}. Note that PML errors converge spectrally to zero.

\begin{figure}[h!]
\begin{subfigure}
    \centering
\stackunder[5pt]{\includegraphics[width=0.49\linewidth, height=0.45\linewidth]{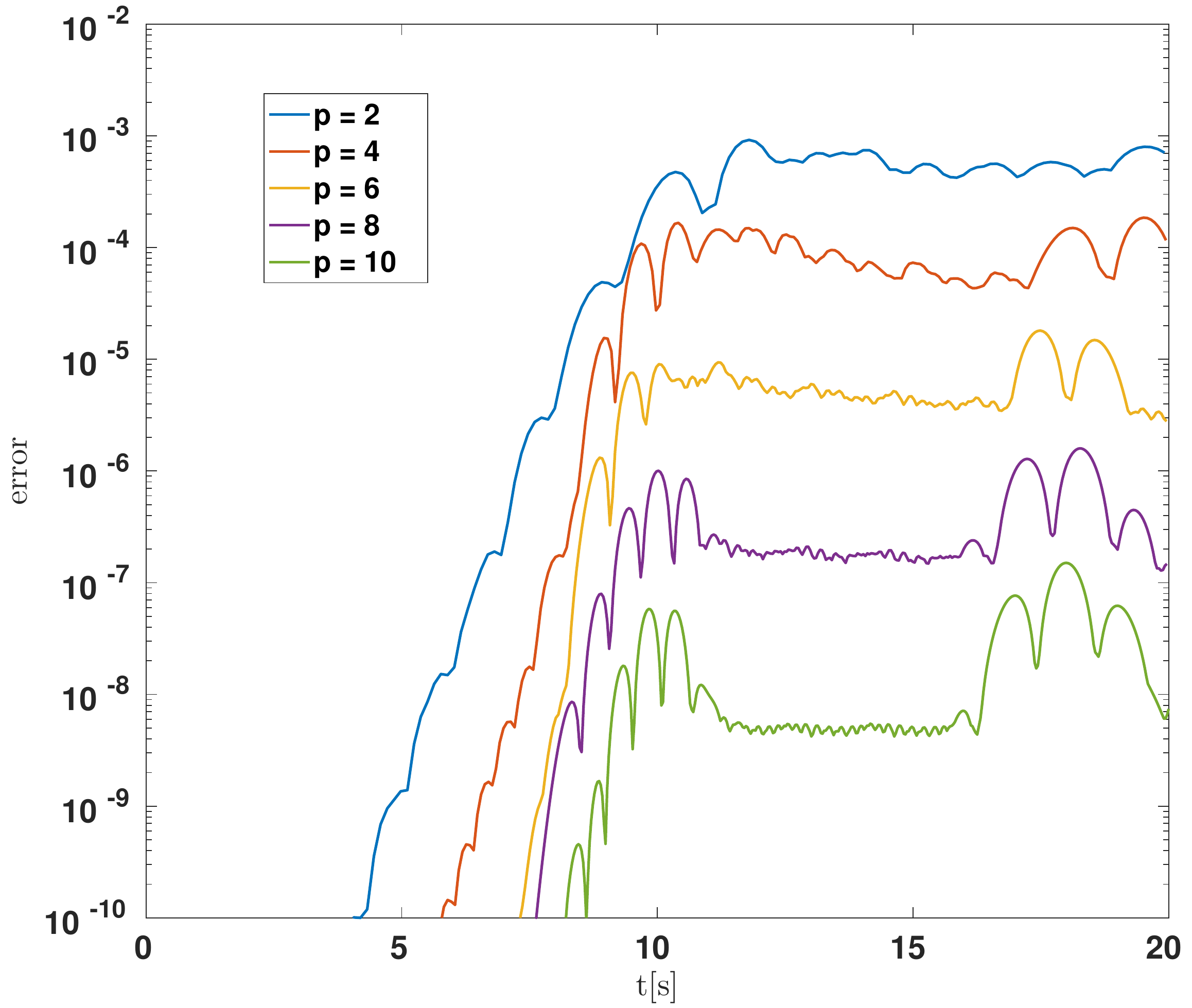}}{$p$-refinement.}%
\hspace{0.0cm}%
\stackunder[5pt]{\includegraphics[width=0.49\linewidth, height=0.45\linewidth]{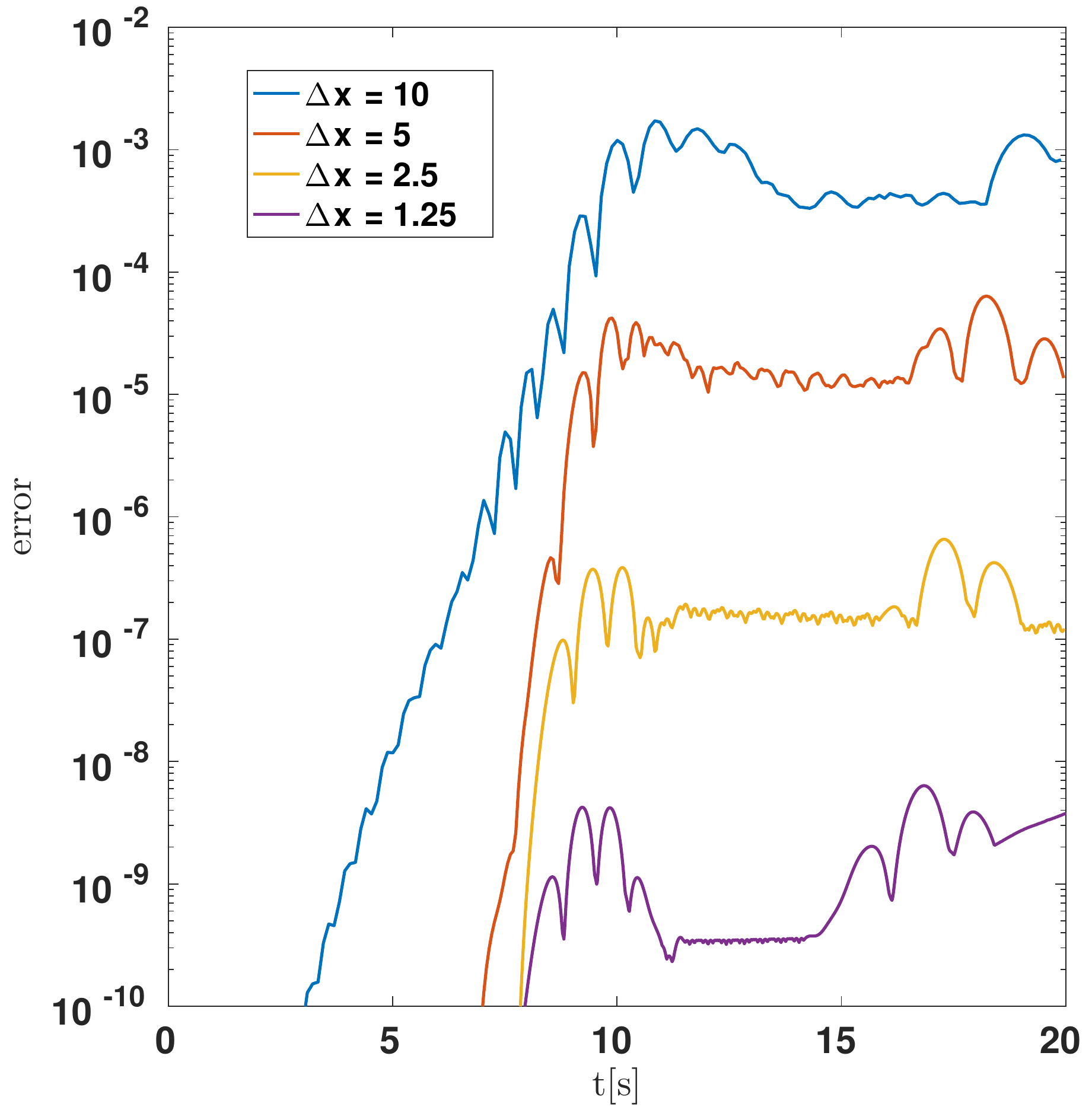}}{$h$-refinement.}%
     \end{subfigure}
    \caption{High order accurate convergence PML errors.}
    \label{fig:Time_series_error_strip}
\end{figure}

We consider $h$-refinement,  by  fixing the polynomial degree $p $ and decreasing the element size by a factor of 2, that is $\Delta{x} = \Delta{y} = 10, 5, 2.5, 1.25 $ km. The magnitude of the relative PML error is $\emph{tol} = \left[50\left(P+1\right)/\Delta{x}\right]^{-\left(P+1\right)}$.  Here  the magnitude of the relative PML error $\emph{tol}$ converges  to zero at the rate $\sim O\left(\Delta{x}^{(p+1)}\right)$. However, the total PML error will converge to zero rate $\sim O\left(\Delta{x}^{(P+1)}\ln(\Delta{x})\right)$, as the element size approaches zero, $\Delta{x} \to 0$, see \cite{DuruGabKreiss2019}.  All other discretization parameters are fixed with the PML stabilizing term $\theta_x = 1$ and the polynomial degree $P = 5$.   Numerical results are shown in Figure \ref{fig:Time_series_error_strip}. Note that PML errors converge spectrally to zero.

\begin{table}[h!]
\centering
\caption{ \textit{A vertical strip PML problem: PML errors.}}
\label{tab:SixthOrderFourthOrderError}
\begin{tabular}{c |  c | c  }
 $\Delta{x}$&$\textbf{error}$&\textbf{rate}\\
\hline
10&8.2513$\times 10^{-4}$&-\\
5&1.3602$\times 10^{-5}$&5.9228\\
2.5&1.1745$\times 10^{-7}$&6.8556\\
1.25&3.7712$\times 10^{-9}$&4.9608\\
\hline
\end{tabular}
\end{table}


%
\subsubsection{The half-space problem with PML corners}
We consider an isotropic linear elastic  half-space, $-\infty < x < \infty$ and $-\infty < y \le 0$,  with the  free-surface boundary condition at $y = 0$. The domain is unbounded at depth and in the  tangential directions. To perform numerical simulations we truncate the computational domain, by introducing the bounded domain $-50~\text{km}  < x < 50~\text{km}$ and $0 \le y \le 50$,  and surround the artificial boundaries with PML, having $ |x| \le 50 + \delta$, $0\le y \le 50 + \delta$. The setup involves two  vertical PML layers closing the left edge at $x = -50 $,  the right edge at $x = 50$,  and horizontal PML layer closing the bottom edge at $y = 50$,  of the elastic block. There are also corner regions where the horizontal and vertical layers are both active. We will now show that the PML stabilizing parameter $\theta_\xi = 1$ ensures robustness for  corners. 

To do this,  we will investigate stability and accuracy by repeating the experiments of the last subsection.
The initial conditions and the PML parameters are similar as before. The external PML boundaries closed with the classical  ABC, by setting the incoming characteristic variables at the boundaries to zero.
\begin{figure} [h!]
 \centering
\includegraphics[width=0.49\linewidth, height=0.475\linewidth]{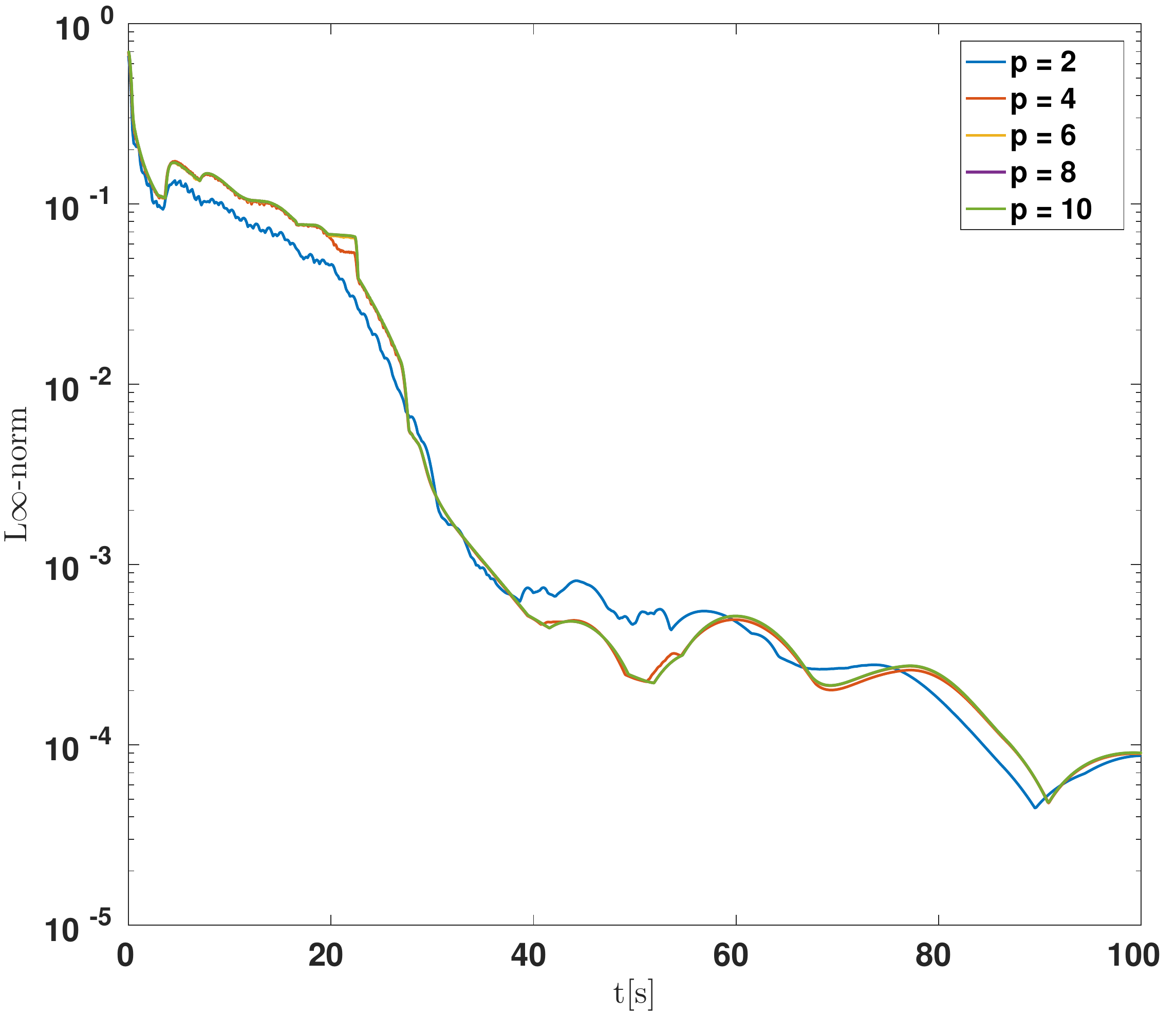}
 \caption{\textit{Time series of the L$_\infty$ norm of the velocity vector for the PML with $\theta_\xi = 1$ in a half-space.}}
 \label{fig:Time_series_HalfPlane}
\end{figure}

Figure \ref{fig:Time_series_HalfPlane} shows asymptotic numerical stability of the PML for polynomial approximations of degree $P \le 10$.
Numerical errors and convergence are shown in Figure \ref{fig:Time_series_error__HalfPlane}.
\begin{figure}[h!]
\begin{subfigure}
    \centering
\stackunder[5pt]{\includegraphics[width=0.49\linewidth, height=0.45\linewidth]{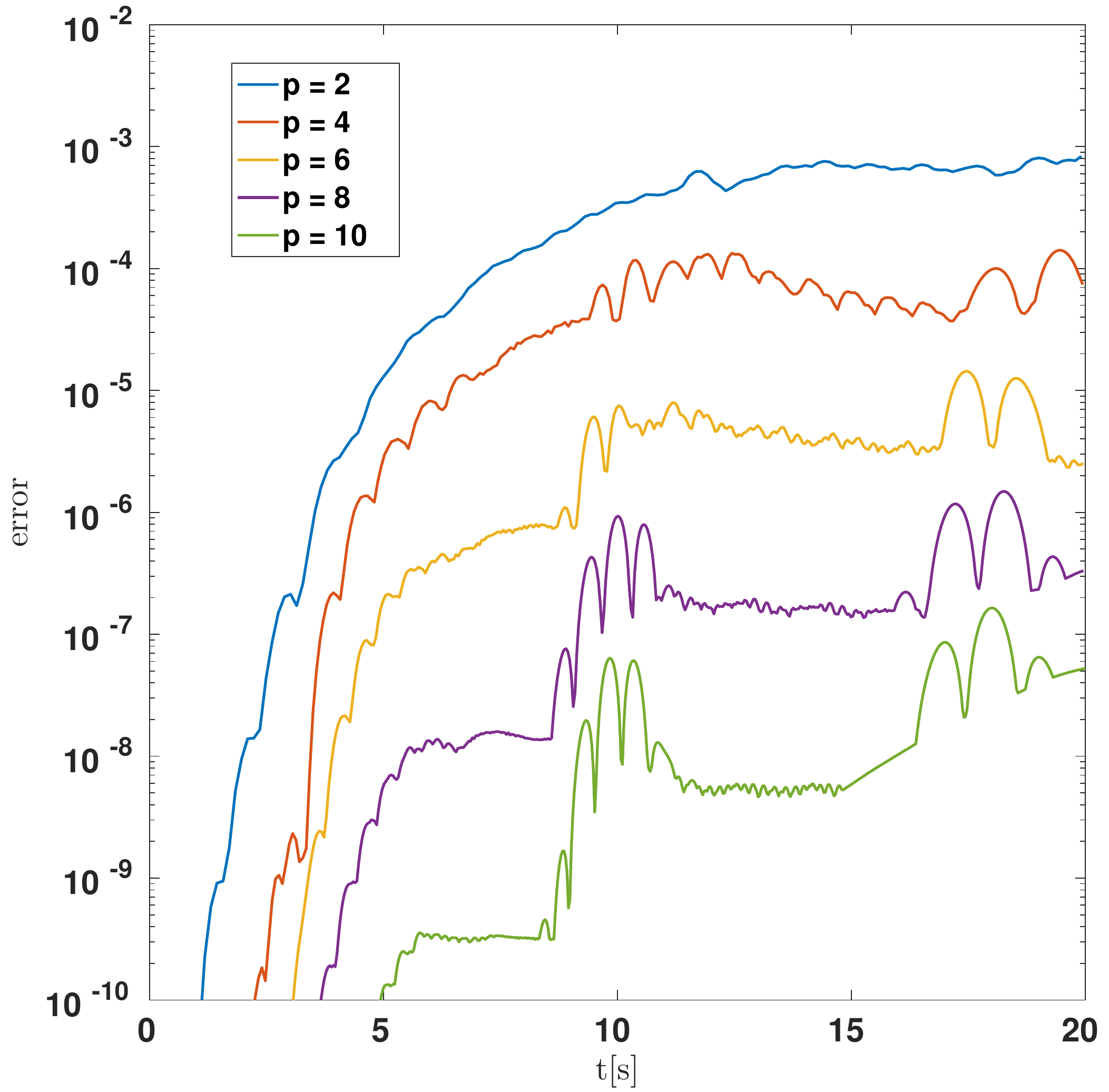}}{$p$-refinement.}%
\hspace{0.0cm}%
\stackunder[5pt]{\includegraphics[width=0.49\linewidth, height=0.45\linewidth]{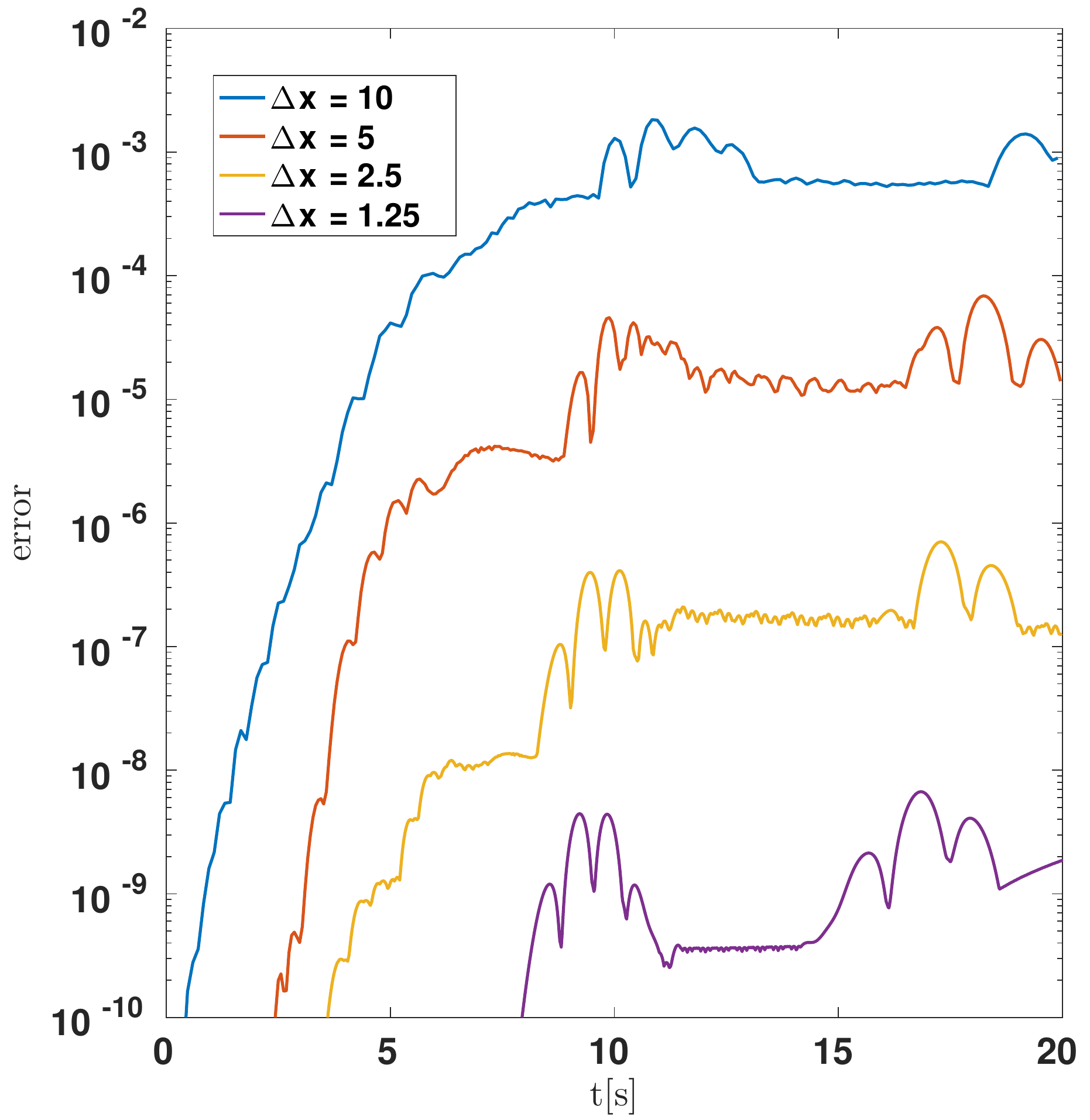}}{$h$-refinement.}%
     \end{subfigure}
    \caption{High order accurate convergence PML errors, with $\theta_\xi = 1$ in a half-space.}
    \label{fig:Time_series_error__HalfPlane}
\end{figure}

%
%
\subsection{3D elastodynamics with the PML}
We will present numerical simulations of elastic waves with the PML in 3D for 3 benchmark  problems of different interests and levels of difficulty. The results demonstrate the importance of the analysis performed in the previous sections and the effectiveness of the PML in the absorption of seismic waves of different types. The benchmark problems are i) a homogeneous whole-space (HWS) problem, ii) a homogeneous   half-space (HHS) problem and the classical layer over a homogeneous half-space (LOH1) problem. The  sketch of a typical model setup with the PML for the LOH1 benchmark is shown in Figure 7. The benchmark problems HHS and LOH1, are proposed by the Seismic wave Propagation and Imaging in Complex media: a European network (SPICE) code validation project, \cite{Kristekova_etal2006,Kristekova_etal2009}. These benchmark problems are designed to quantify and assess the accuracy of simulation codes for seismic surface and interface waves.

\begin{figure}[h!]
\centering
\includegraphics[width=.625\linewidth]{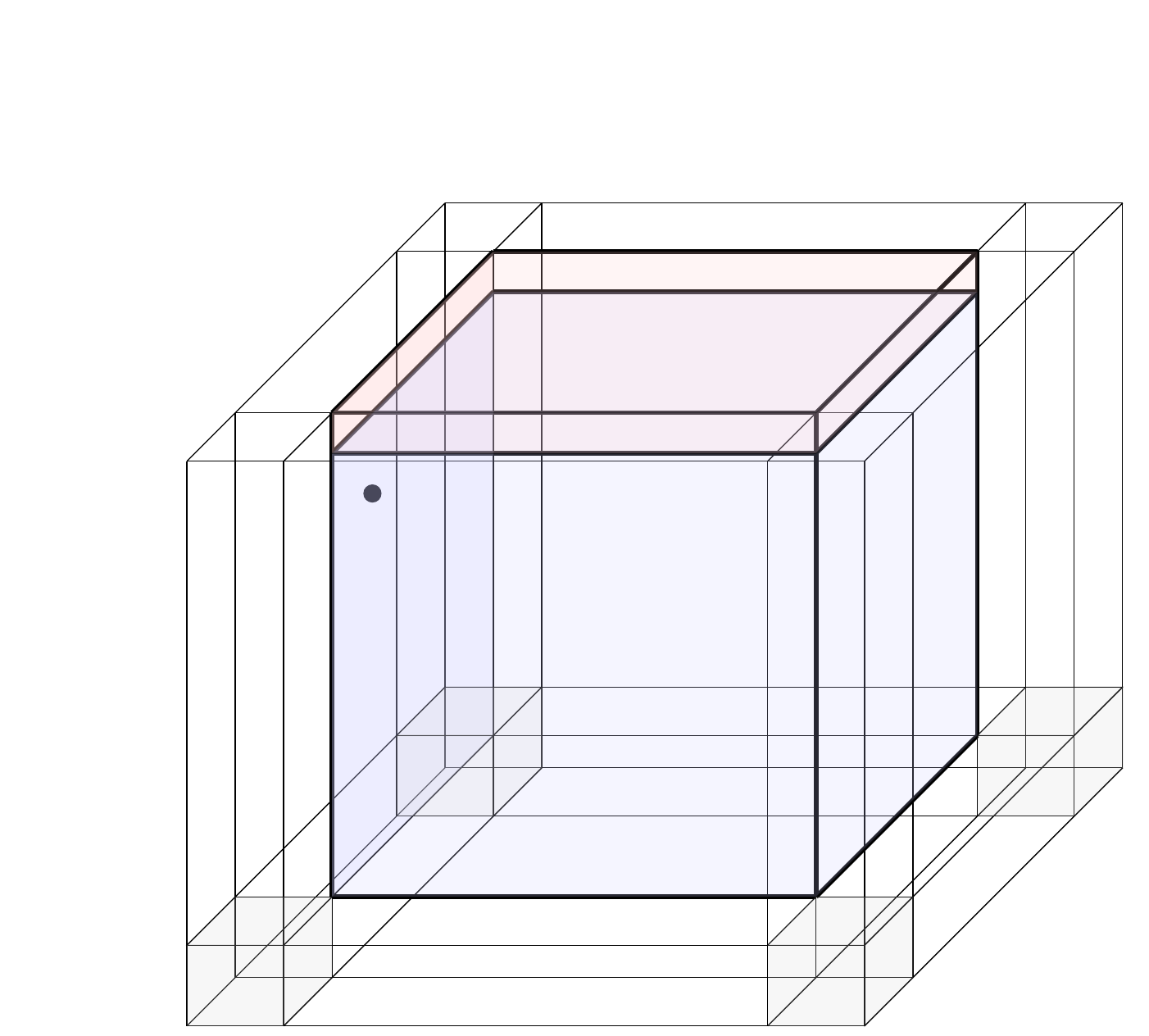}
\caption{The model setup including the PML for the LOH1 benchmark. The black dot  at the top left corner depicts the position of the source.}
\label{fig:LOH1_PML_Setup}
\end{figure}

We will consider seismic sources defined by the singular moment tensor point source 
{
\begin{align}\label{eq:pointsource}
\mathbf{f}(x,y,z,t) &=   \mathbf{M} \delta_x(x-x_0)\delta_y(y-y_0)\delta_z(z-z_0)g(t), \quad 
\mathbf{M} 
=
\begin{pmatrix}
M_{xx} & M_{xy} &M_{xz} \\
M_{xy} & M_{yy} &M_{yz} \\
M_{xz} & M_{yz} &M_{zz} 
\end{pmatrix}.
\end{align}
}
Here, $ \mathbf{M} $ is the symmetric second order moment tensor, $ \delta_{\eta}(\eta)$ are the one dimensional Dirac delta function, $(x_0, y_0, z_0)$ is the source location, and $g(t)$ is the source time function. We will consider two different source mechanisms: an explosive source modelling an underground explosion and a double couple source modelling an earthquake. 

In order to perform numerical simulations, we will consider a bounded computational domain by truncating the unbounded extent of the domain.
We discretize the domain uniformly with 25 DG elements in each spatial direction.
At an artificial boundary, to simulate unbounded domain we will include the PML in the 3 elements closest to the boundary, and close the external PML boundaries with the classical ABC, by setting the incoming characteristics to zero. 

From the analysis performed in the previous sections, and 2D numerical experiments presented in the last subsection, we know that the PML stabilization parameters $\theta_\xi = 1$ are critical for numerical stability of the PML. In all experiments we will  use $\theta_\xi = 1$  and set the PML relative error $\mathrm{tol} = 0.1 \%$, which is rather a modest error tolerance.
\subsubsection{The homogeneous whole-space problem}
Here, we consider solutions of  a 3D homogeneous isotropic elastic solid in a  whole space, and an explosive point source, 
with $\mathbf{M} = \mathrm{M}_0\mathbf{I}$, where $\mathbf{I}$ is the $3$-by-$3$ identity matrix, that is    $\mathrm{M}_{xx} = \mathrm{M}_{yy} = \mathrm{M}_{zz} =\mathrm{M}_0$, and  $\mathrm{M}_0 = 10^{18} ~\ \text{Nm}$ is the moment magnitude. 
The source time function is given by the Gaussian
{
\begin{align}\label{eq:pointsource_compressive}
  g(t) = \frac{1}{\sigma_0\sqrt{2\pi}}e^{-\frac{(t-t_0)^2}{2\sigma_0^2}}, \quad  \sigma_0 = 0.1149~\text{s}, \quad t_0 = 0.7~\text{s}.
\end{align}
}
The density and wave speeds of the medium are 
\begin{align}\label{eq:homo_data}
\rho = 2670~\ \text{kg/m}^3, \quad c_p = 6000~\ \text{m/s}, \quad c_s = 3464~\ \text{m/s}.
\end{align}
The whole space problem has an analytical solution, see for example \cite{Achenbach1973,Petersson_etal2016}.

\begin{figure}[h!]
\begin{subfigure}
    \centering
\stackunder[5pt]{\includegraphics[width=0.475\textwidth]{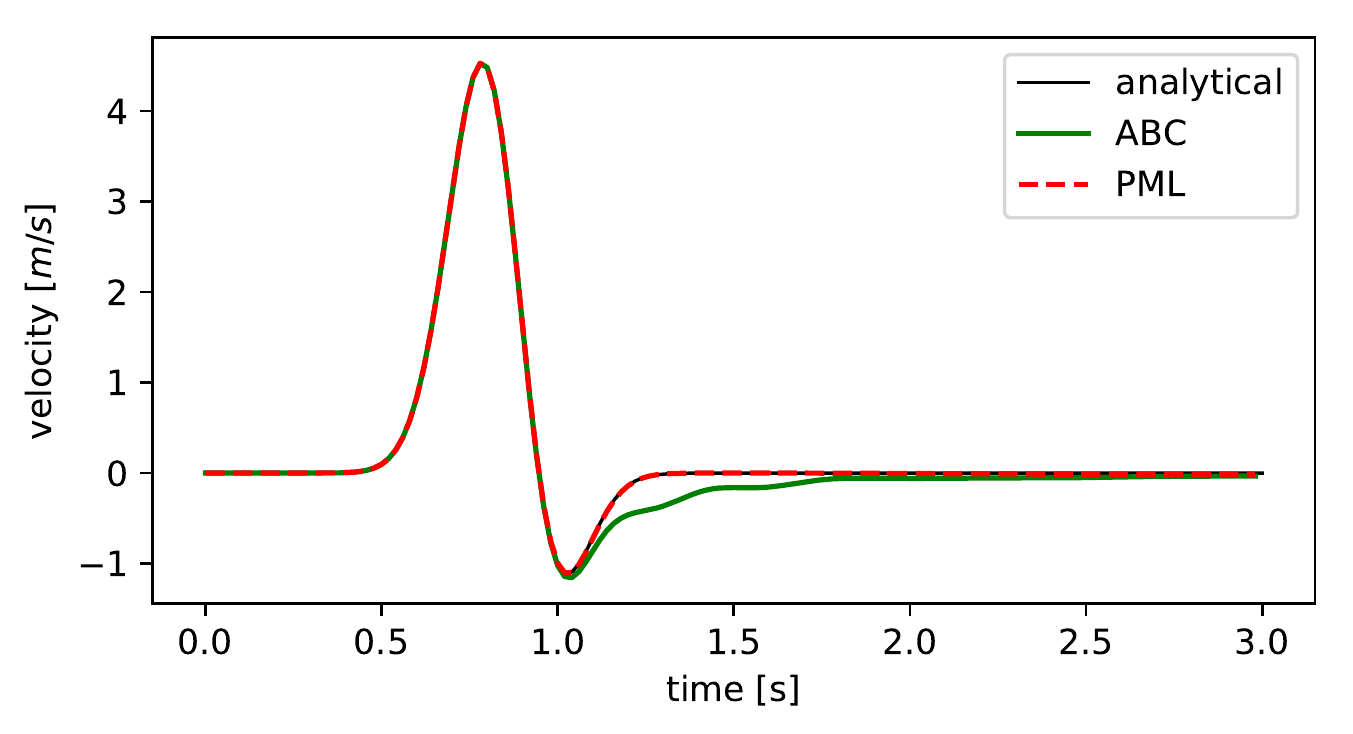}}{Receiver 1:   1 km from the source.}%
\hspace{0.0cm}%
\stackunder[5pt]{\includegraphics[width=0.5\textwidth]{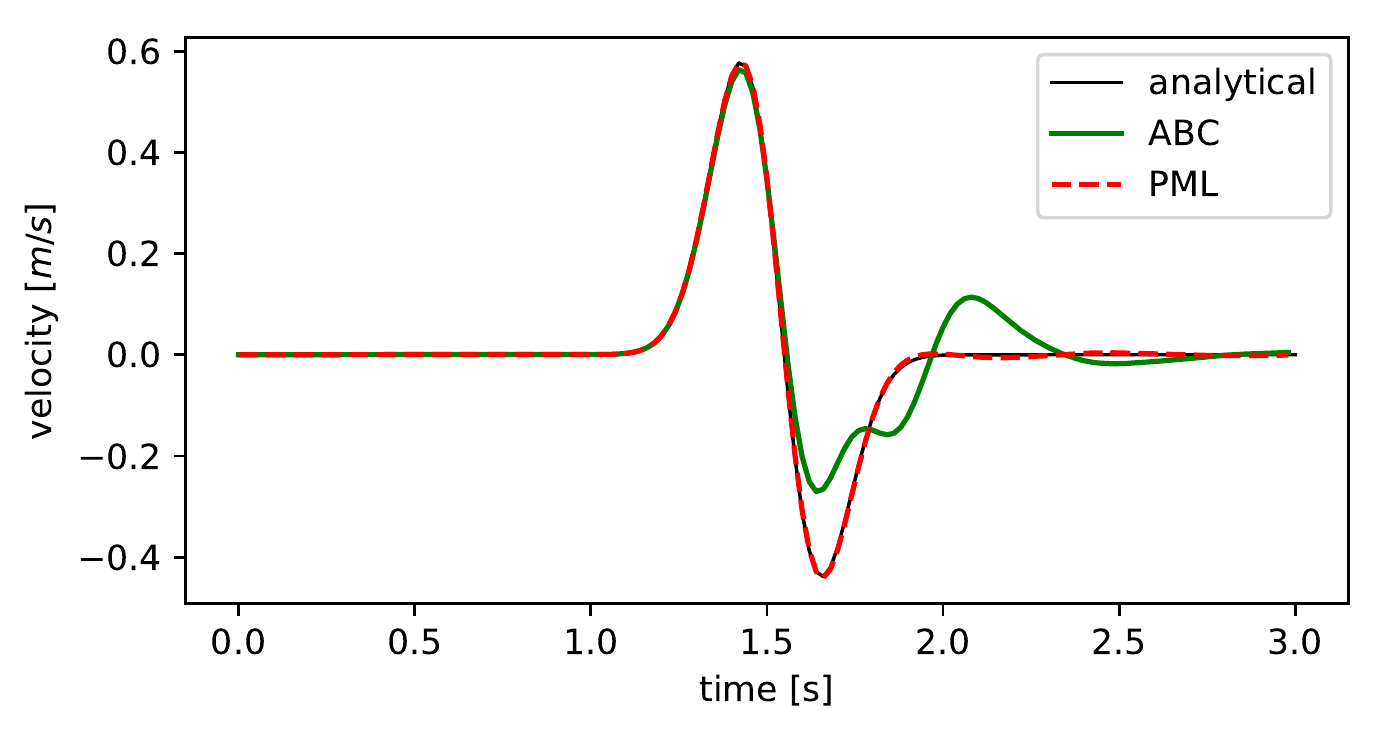}}{Receiver 2:   5 km from the source.}%
     \end{subfigure}
    \caption{The whole space problem. Comparing the ABC and PML solutions with the analytical solution.}
    \label{fig:wholespace}
\end{figure}

The computational domain is the cube $(x,y,z) = [0, 10~\ \text{km}]\times[0, 10~\ \text{km}]\times[0, 10~\ \text{km}]$, and the source is located at $(x_0, y_0, z_0) = (3.4~\ \text{km}, 5~\ \text{km}, 5~\ \text{km})$.  The domain is surrounded by the PML to absorb outgoing waves.  We consider degree $N =5$ polynomial approximation, and run the simulation until the final time $t = 3$ s. We place a receiver, Receiver 1 at $(x_r, y_r, z_r) = (4.4~\ \text{km}, 5~\ \text{km}, 5~\ \text{km})$,  1 km from the source, and another receiver, Receiver 2 at $(x_r, y_r, z_r) = (8.4~\ \text{km}, 5~\ \text{km}, 5~\ \text{km})$,  $5~\ \text{km}$ from the source, where the solution is sampled. Note that while Receiver 1 is closer to the source, Receiver 2 is much closer to the artificial boundary, at $x = 10~$km.

The solutions are displayed in Figure \ref{fig:wholespace}. Note that initially the ABC and PML solutions match the analytical solution very well. However, at later times the ABC solution  is polluted by numerical reflections arriving from the artificial boundaries. The PML solution remains accurate and matches the analytical solution excellently for the entire simulation duration.


\subsubsection{The homogeneous half-space (HHS) problem}\label{sec:hhs1}
We consider  a homogenous isotropic elastic half-space with a planar free-surface, at $x = 0$, where tractions vanish.
We place a double couple moment tensor point source at depth, $x = 0.693~\ $km below the free surface.
The source time function is define by
\begin{align}\label{eq:loh1:pointsource}
 g(t) = \frac{t}{T^{2}}e^{-{t}/{T}},  \quad T = 0.1 ~\ \text{s},
\end{align}
and the moment tensor $\mathbf{M}$, is zero almost everywhere except for the shear component $M_{yz}=M_{zy}=M_0$, and $M_0 =10^{18}$~Nm is the moment magnitude.
The constant  material parameters are given by
\begin{align}\label{eq:hhs_data}
\rho = 2700~\ \text{kg/m}^3, \quad c_p = 6000~\ \text{m/s}, \quad c_s = 3464~\ \text{m/s}.
\end{align}
The HHS problem has an analytical solution, see \cite{Kristekova_etal2006,Kristekova_etal2009}.
Note that the domain is unbounded at depth and in the tangential directions.

In order to perform numerical simulations we consider the bounded computational cube $(x,y,z) = [0, 16.333~\text{km}]\times[-2.287~\text{km}, 14.046~\text{km}]\times[-2.287~\text{km}, 14.046~\text{km}]$, and the source is located at $(x_0, y_0, z_0) = (0.693~\text{km}, 0, 0)$.  The truncated boundaries of the domain were surrounding by the PML to absorb outgoing waves. The setup is analogous to Figure \ref{fig:LOH1_PML_Setup}, with the only difference being the constant medium parameters used here.  Note that the free-surface boundary extends into the PML. We consider degree $N =5$ polynomial approximation, and run the simulation until the final time $t = 5$ s. 
We place 9 receivers  on the free-surface, at $x = 0$, where the solutions are sampled. Table \ref{tbl:receivers_hhs1} shows the positions of the receivers relative to the epicenter.
\begin{table}[h!]
  \caption{Receiver Positions of the homogenous half-space and layer over half-space problem relative to the epicenter.}
  \centering{
    \begin{tabular}{|c|c|c|c|c|c|c|c|c|c|}
      \hline
      Receiver &1&2&3&4&5&6&7&8&9  \\ \hline
      y[km]& 0     & 0     & 0       & 0.490  & 3.919 & 7.348 & 0.577 & 4.612 & 8.647 \\\hline
      z[km]& 0.693 & 5.542 & 10.392  & 0.490  & 3.919 & 7.348 & 0.384 & 3.075 & 5.764 \\ \hline
    \end{tabular}
  }
  \label{tbl:receivers_hhs1}
\end{table}

\begin{figure}[h!]
\begin{subfigure}
    \centering
\stackunder[5pt]{\includegraphics[width=0.245\textwidth]{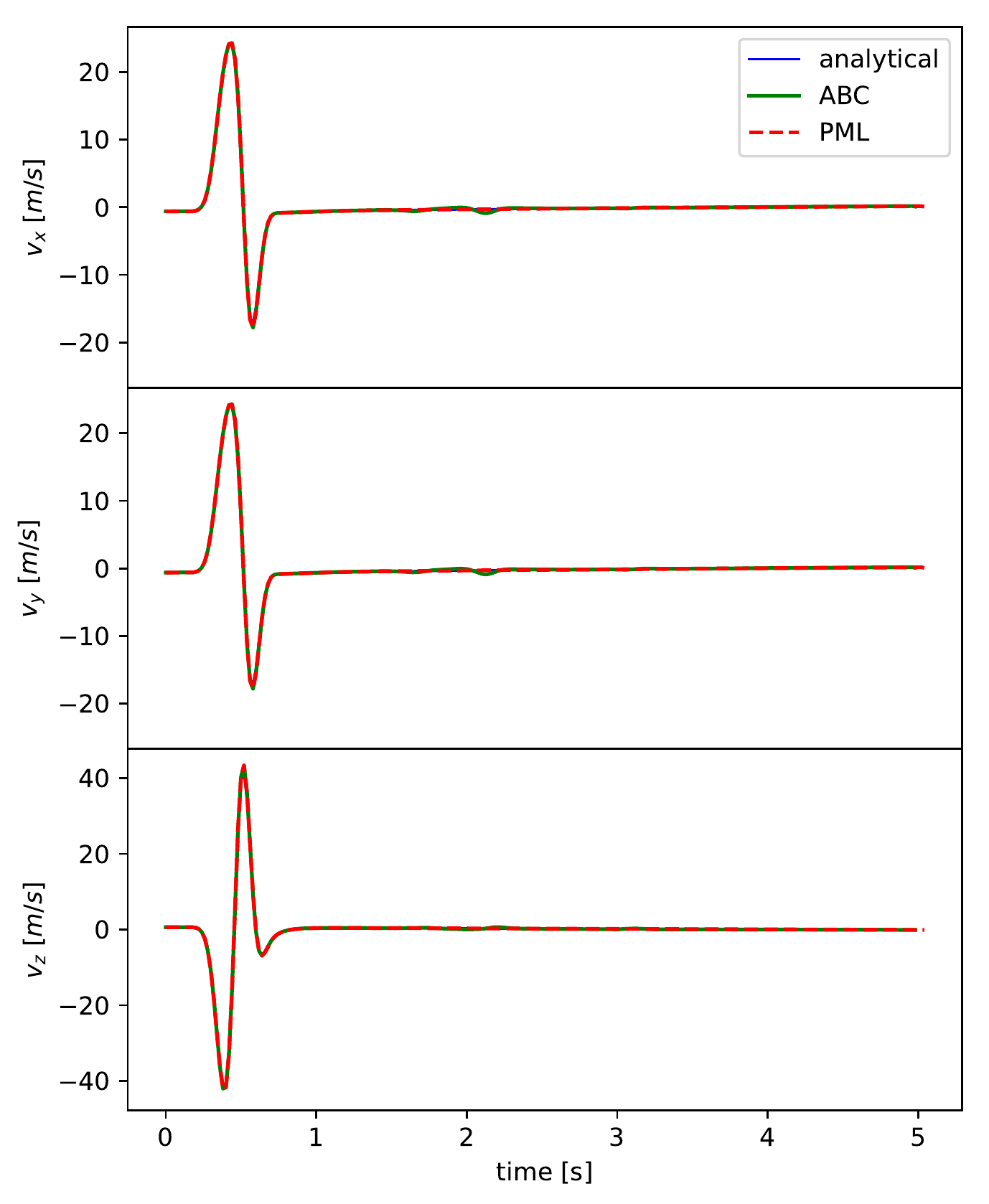}}{Receiver 4}%
\hspace{0.0cm}%
\stackunder[5pt]{\includegraphics[width=0.245\textwidth]{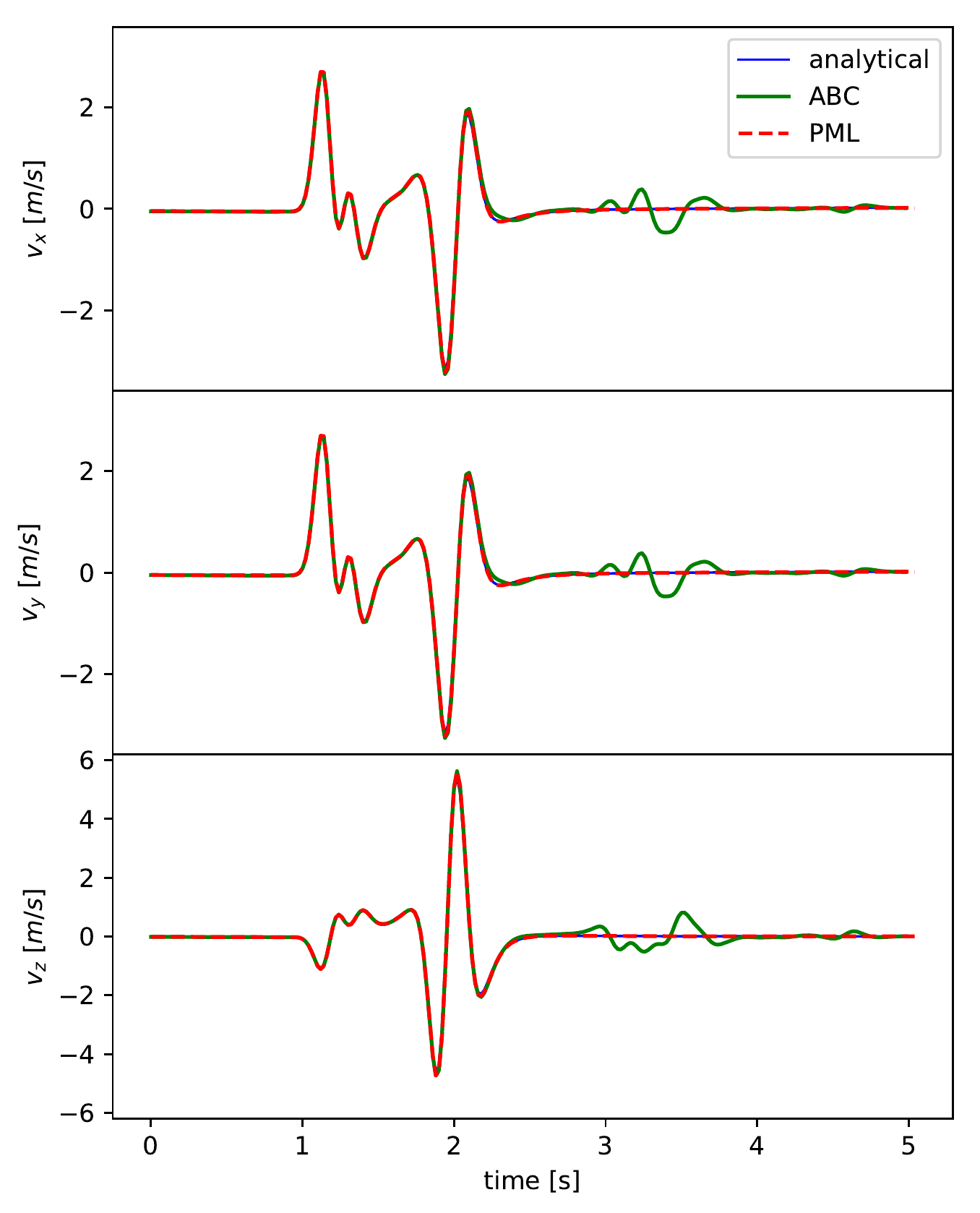}}{Receiver 5}%
\hspace{0.0cm}%
\stackunder[5pt]{\includegraphics[width=0.245\textwidth]{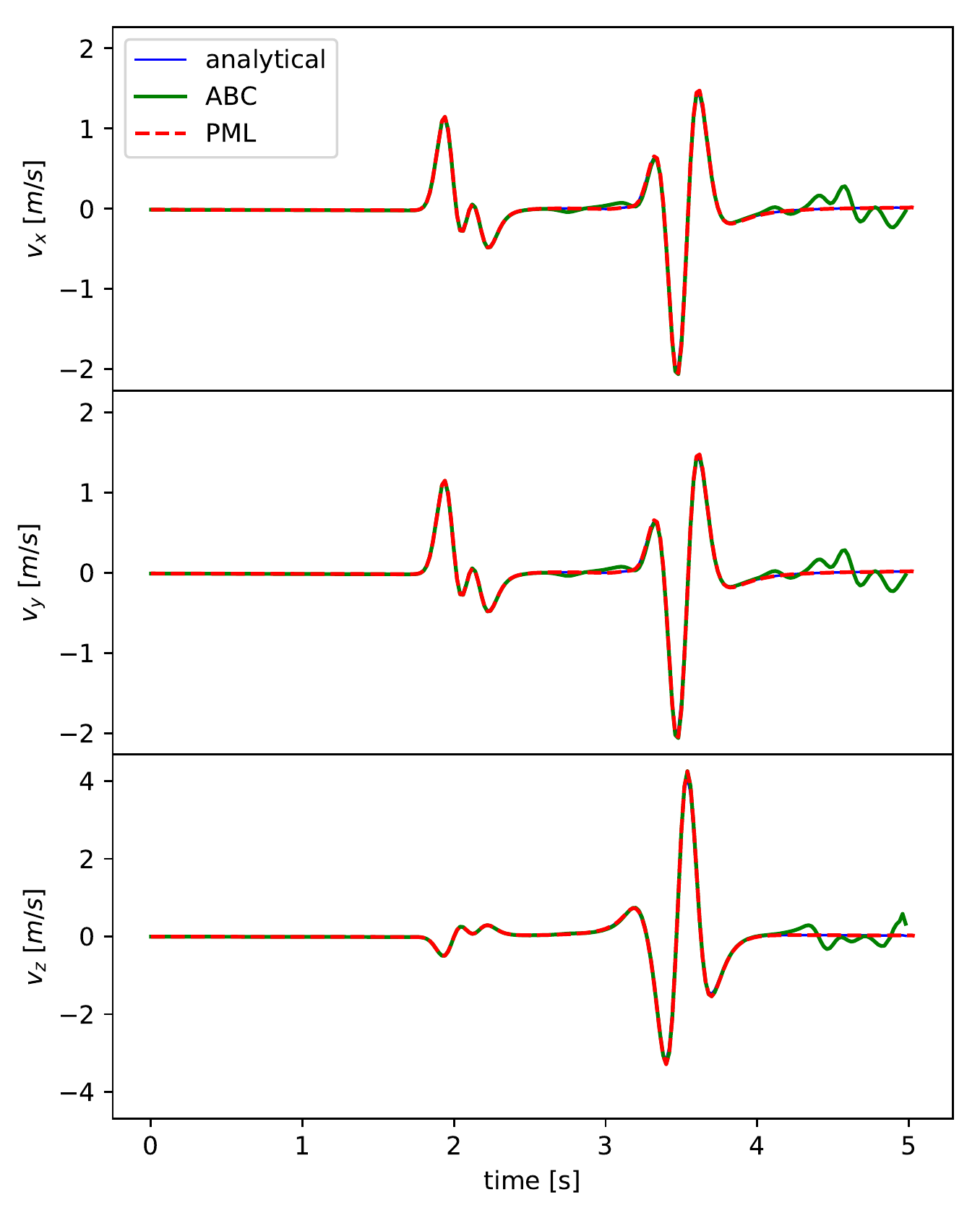}}{Receiver 6}%
\hspace{0.0cm}%
\stackunder[5pt]{\includegraphics[width=0.245\textwidth]{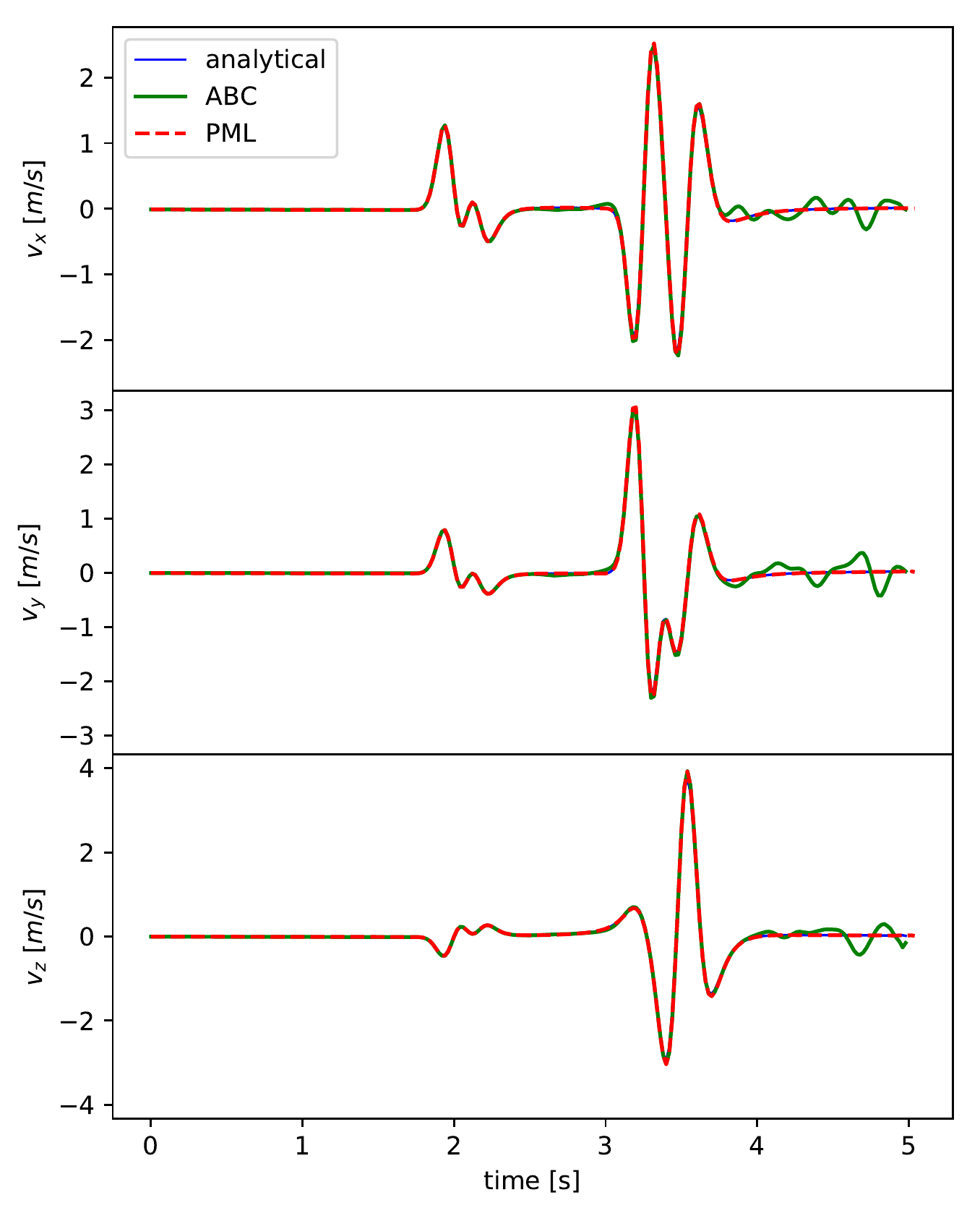}}{Receiver 9}%
     \end{subfigure}
    \caption{The homogeneous half space problem}
    \label{fig:halfspace}
\end{figure}

The solutions are displayed in Figure \ref{fig:halfspace}, for Receivers 4, 5, 6 and 9. The seismograms for the remaining Receivers are placed in Appendix \ref{sec:other_stations}.  Note that initially the ABC and PML solutions match the analytical solution very well. However, at later times the ABC solution  is polluted by numerical reflections arriving from the artificial boundaries. These artificial reflections can never diminish with $p-$ or $h-$refinement.
The PML solution remains accurate and matches the analytical solution excellently for the entire simulation duration. 

\subsubsection{Layer over a homogeneous half-space (LOH1) problem}
We will now consider the LOH1 benchmark problem, which has an analytical solution, see [29,30]. 
The LOH1 setup is  similar to the HHS problem. The only differences are the position of the 
source, discontinuous medium parameters and the duration of the simulation. The LOH1 problem 
consists of a planar free surface and an internal planar interface between a thin homogeneous 
soft layer and hard half-space. The material properties in the medium are given by 
%
\begin{align}\label{eq:hhs_data}
  x < 1\textrm{~km: }\rho = 2600~\ \text{kg/m}^3, \quad c_p = 4000~\ \text{m/s}, \quad c_s = 2000~\ \text{m/s}, \\
  x \ge 1\textrm{~km: } \rho = 2700~\ \text{kg/m}^3, \quad c_p = 6000~\ \text{m/s}, \quad c_s = 3464~\ \text{m/s}.
\end{align}
We place a double couple moment tensor point source at depth, $x = 2~$km below the free surface, 
and the source time function is  given by \eqref{eq:loh1:pointsource}. 
Note that the domain is unbounded at depth and in the tangential direction, and that the thin 
soft layer acts as a waveguide. Thus, in addition to body waves, surface waves and interface 
waves, guided waves are also supported.
%

As above, to perform numerical simulations we consider the bounded computational cube $(x,y,z) = [0, 16.333~\text{km}]\times[-2.287~\text{km}, 14.046~\text{km}]\times[-2.287~\text{km}, 14.046~\text{km}]$, and the source is located at $(x_0, y_0, z_0) = (2~\text{km}, 0, 0)$.  The truncated boundaries of the domain were surrounding by the PML to absorb outgoing waves, see  Figure \ref{fig:LOH1_PML_Setup}. 
Note 
that the free-surface boundary, at $x =0$, and the material discontinuity, at $x =1~$km, extend into the PML. We consider degrees $N =3, 5$ polynomial approximations, and run the simulation until the final time $t = 9~$s. 

\begin{figure}[h!]
\begin{subfigure}
    \centering
\stackunder[5pt]{\includegraphics[width=0.245\textwidth]{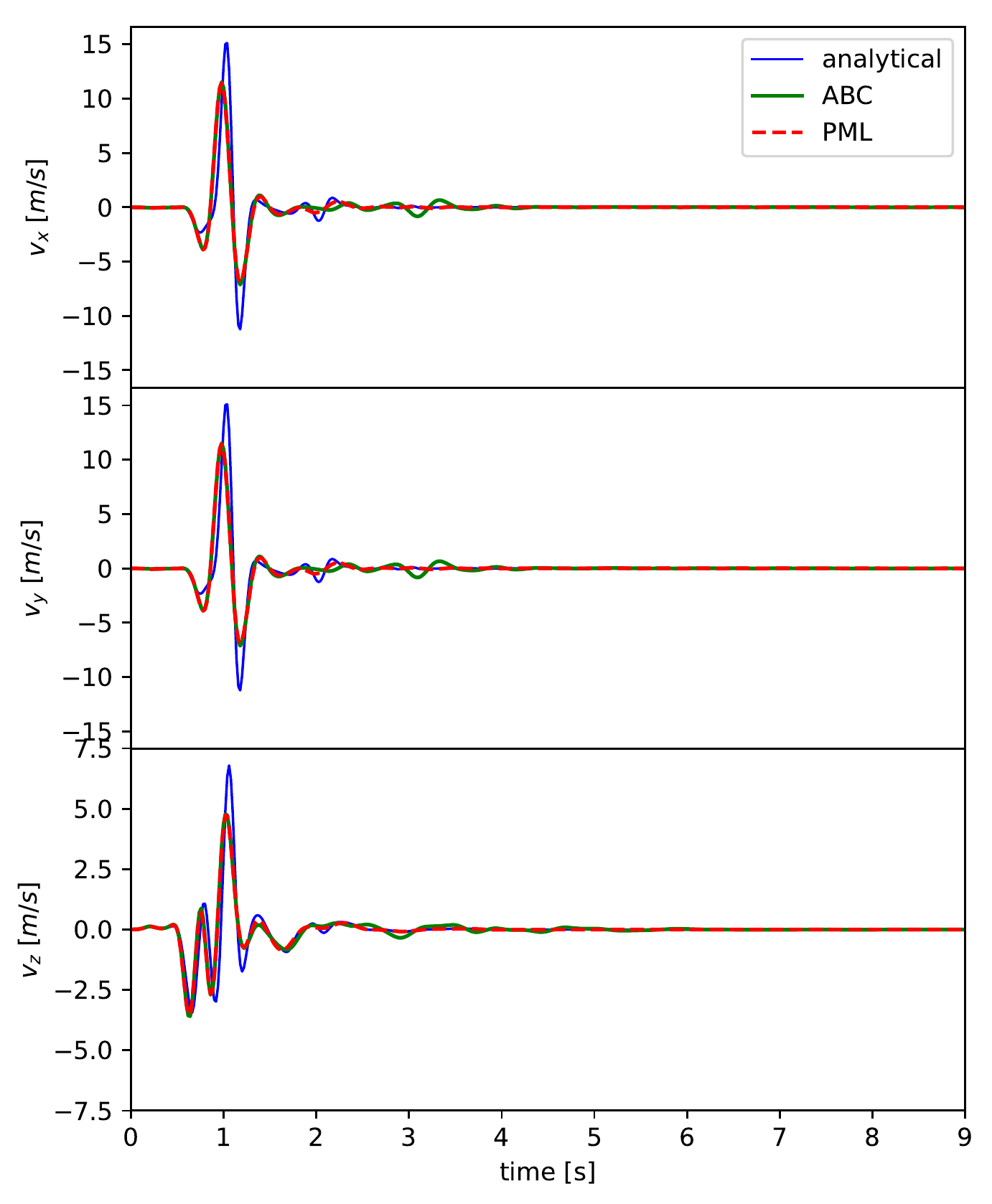}}{Receiver 4}%
\hspace{0.0cm}%
\stackunder[5pt]{\includegraphics[width=0.245\textwidth]{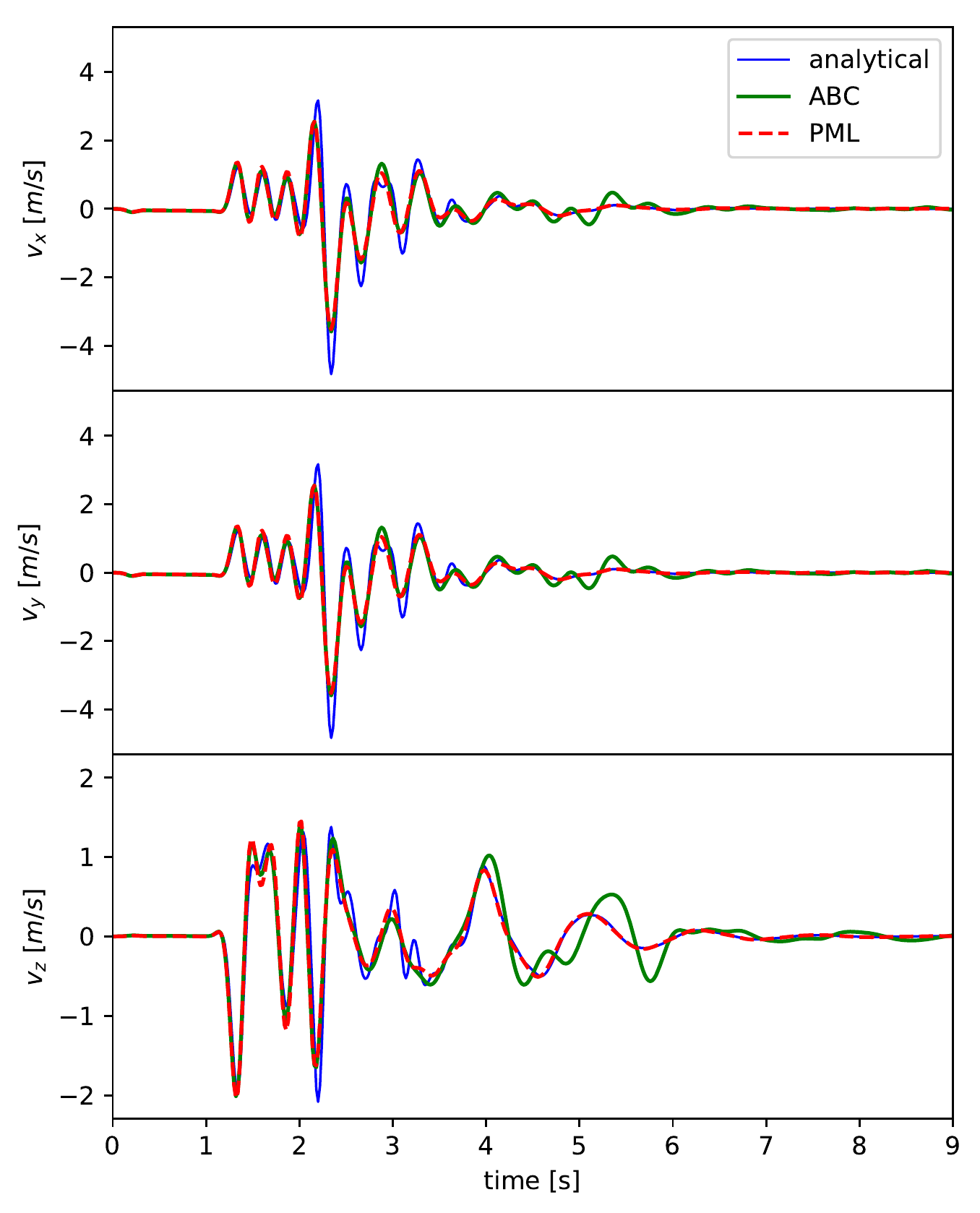}}{Receiver 5}%
\hspace{0.0cm}%
\stackunder[5pt]{\includegraphics[width=0.245\textwidth]{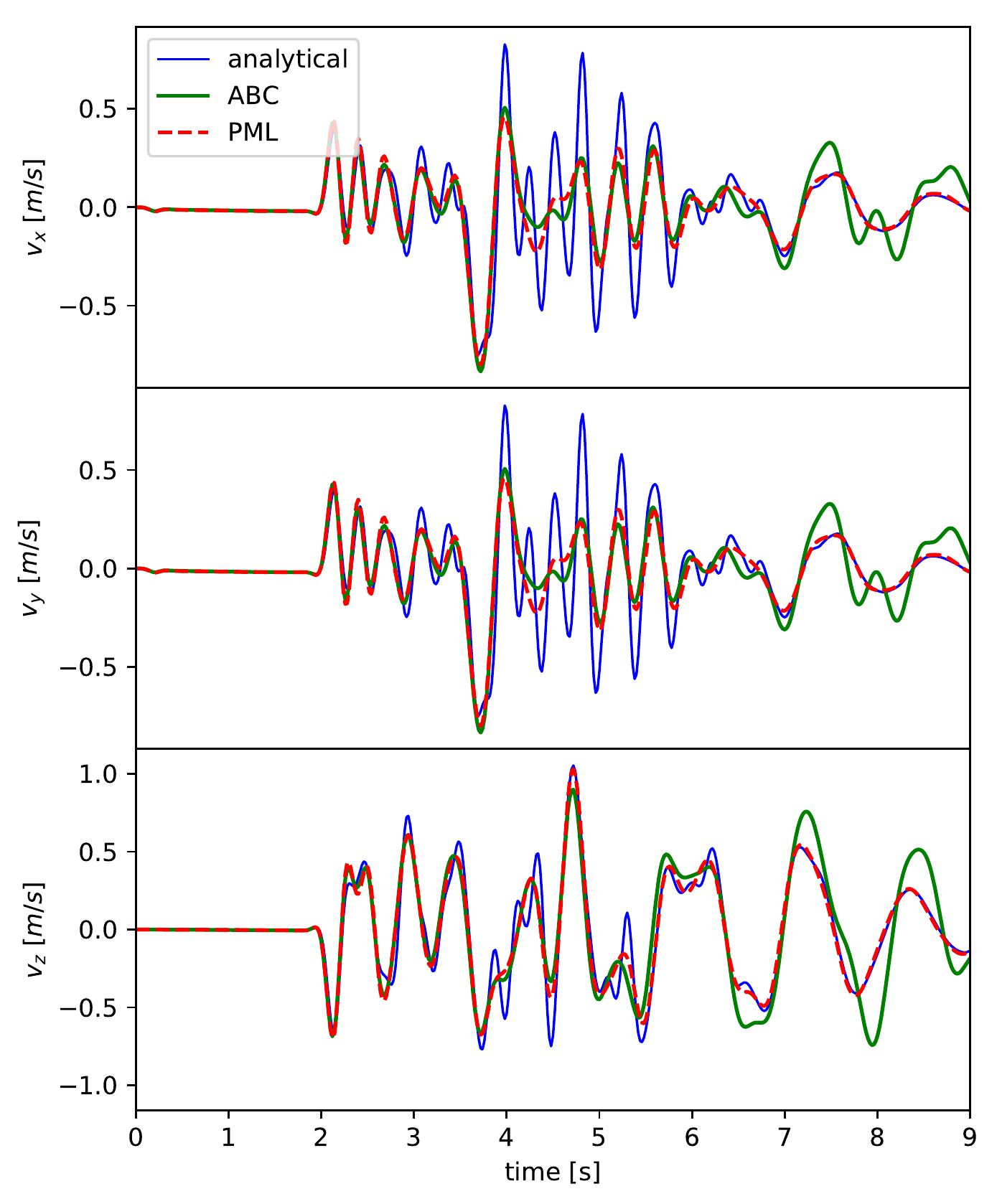}}{Receiver 6}%
\hspace{0.0cm}%
\stackunder[5pt]{\includegraphics[width=0.245\textwidth]{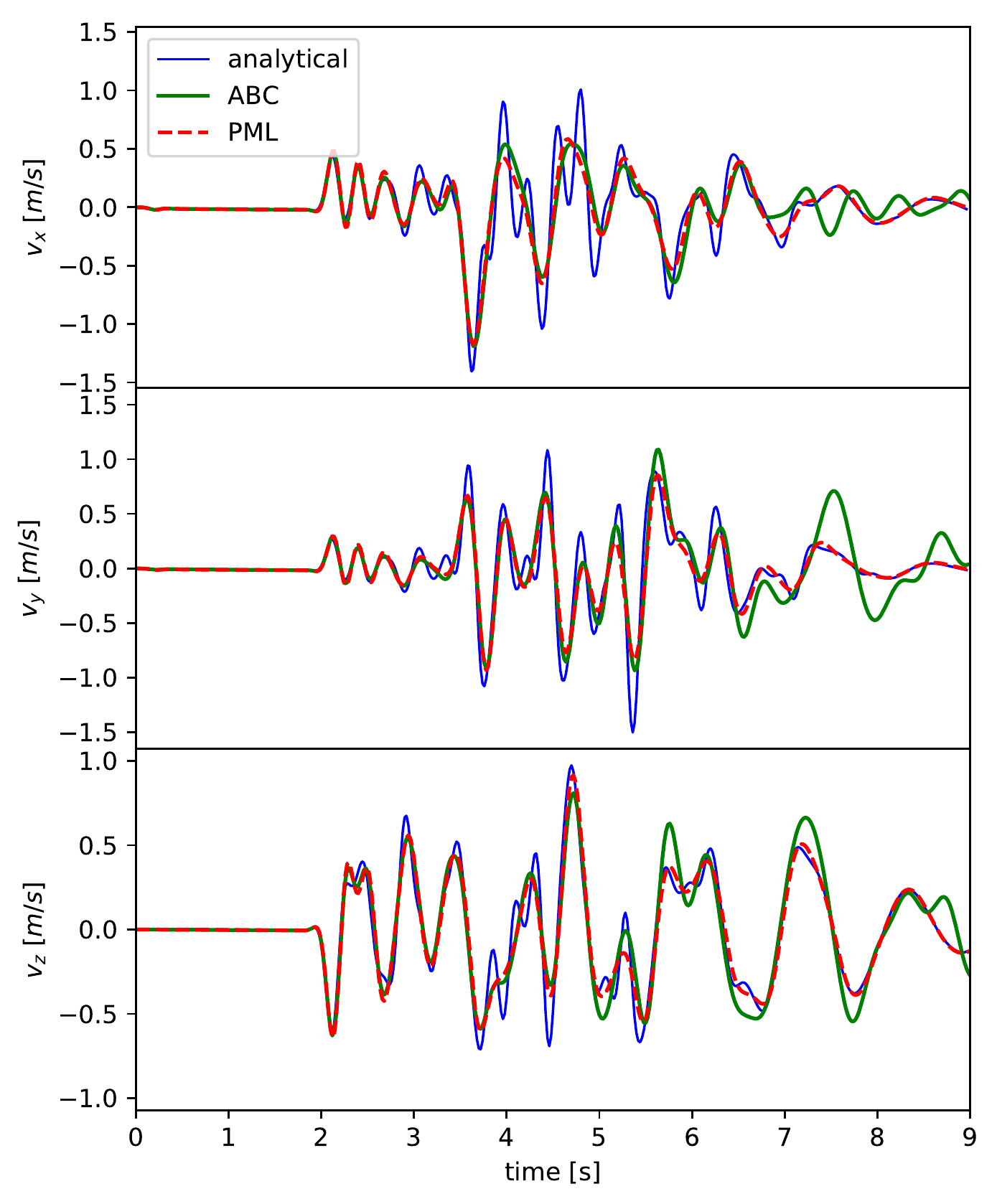}}{Receiver 9}%
     \end{subfigure}
    \caption{The LOH1 benchmark problem with degree $N = 3$ polynomial approximation}
    \label{fig:loh1_o3}
\end{figure}

\begin{figure}[h!]
\begin{subfigure}
    \centering
\stackunder[5pt]{\includegraphics[width=0.245\textwidth]{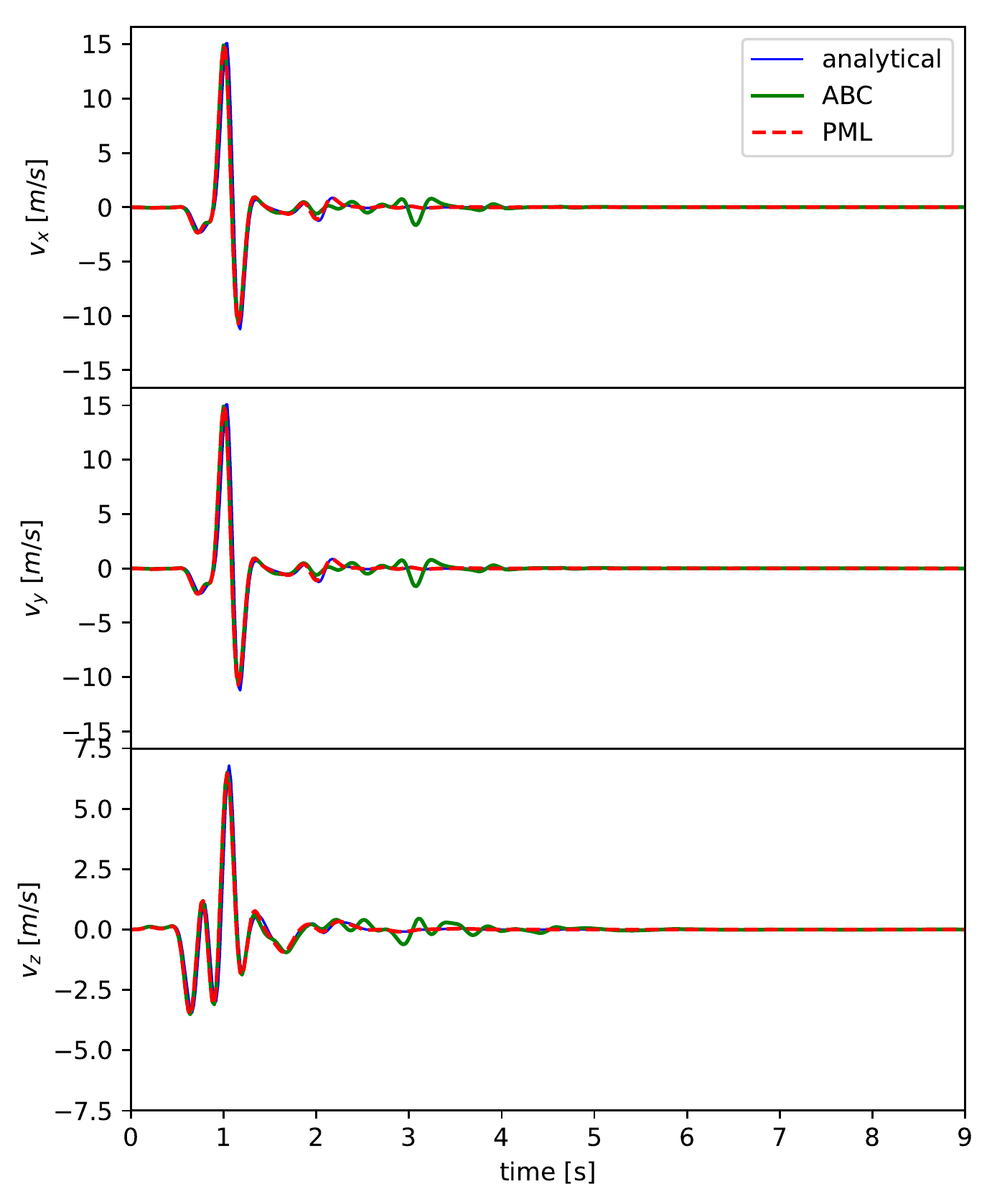}}{Receiver 4}%
\hspace{0.0cm}%
\stackunder[5pt]{\includegraphics[width=0.245\textwidth]{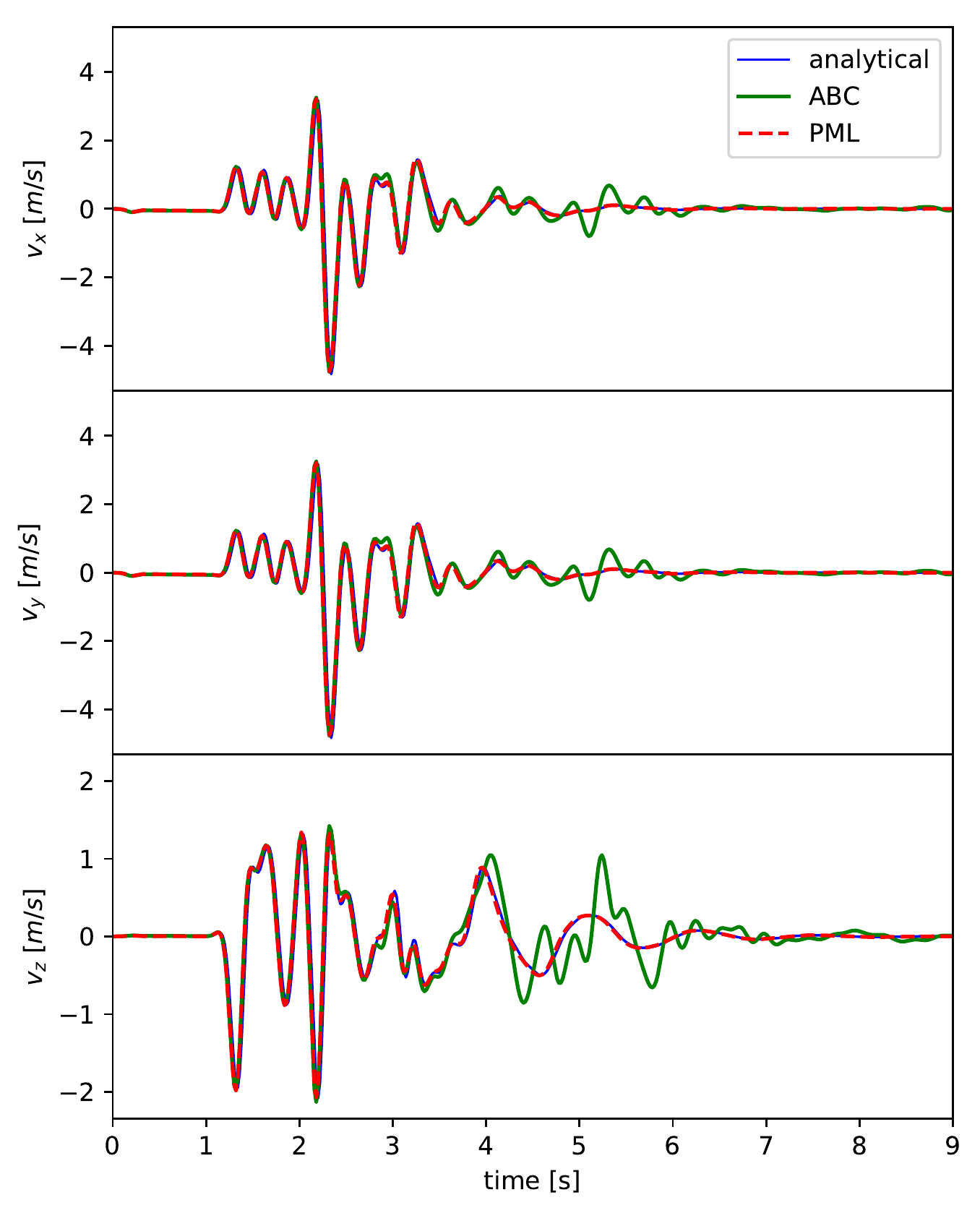}}{Receiver 5}%
\hspace{0.0cm}%
\stackunder[5pt]{\includegraphics[width=0.245\textwidth]{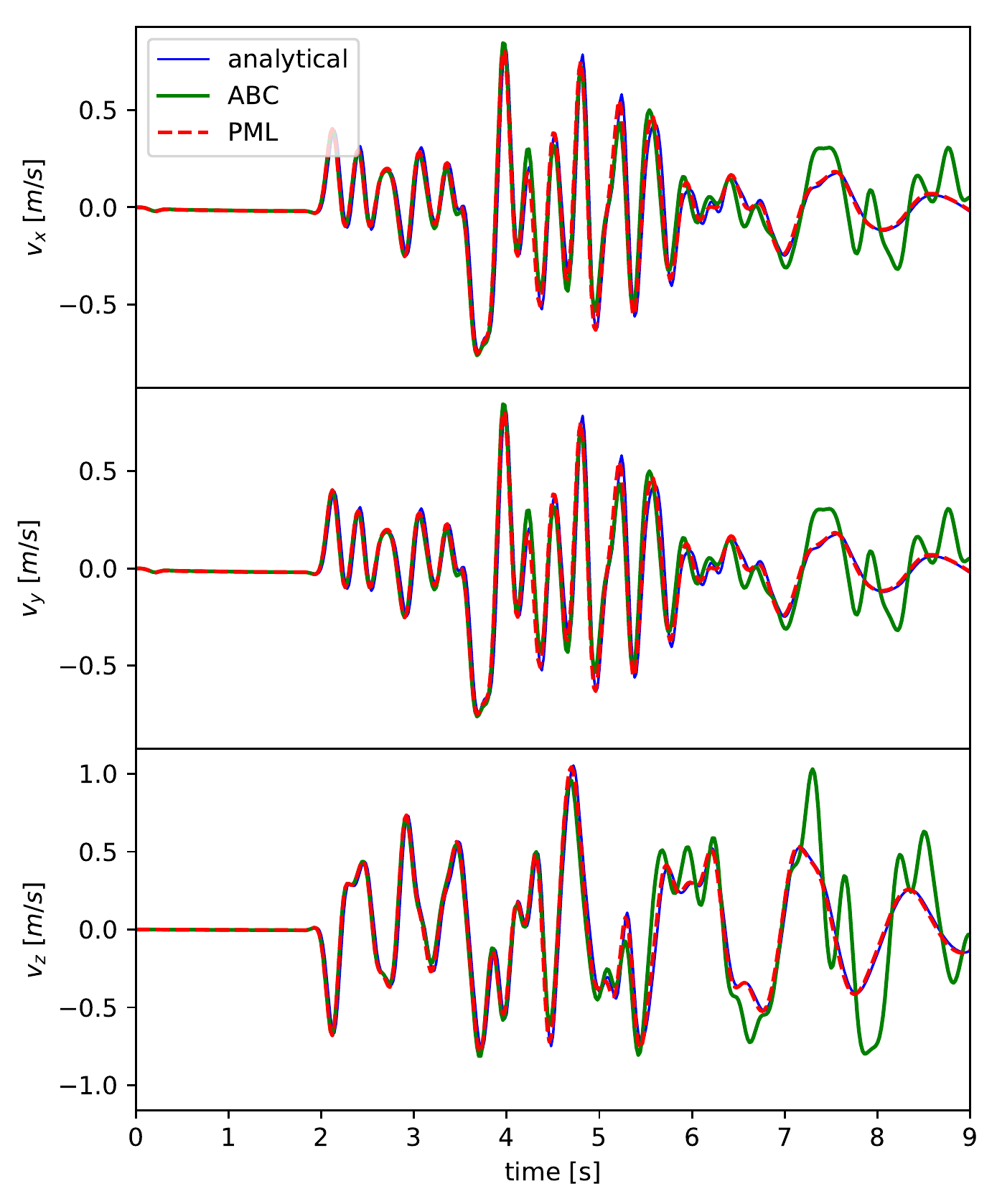}}{Receiver 6}%
\hspace{0.0cm}%
\stackunder[5pt]{\includegraphics[width=0.245\textwidth]{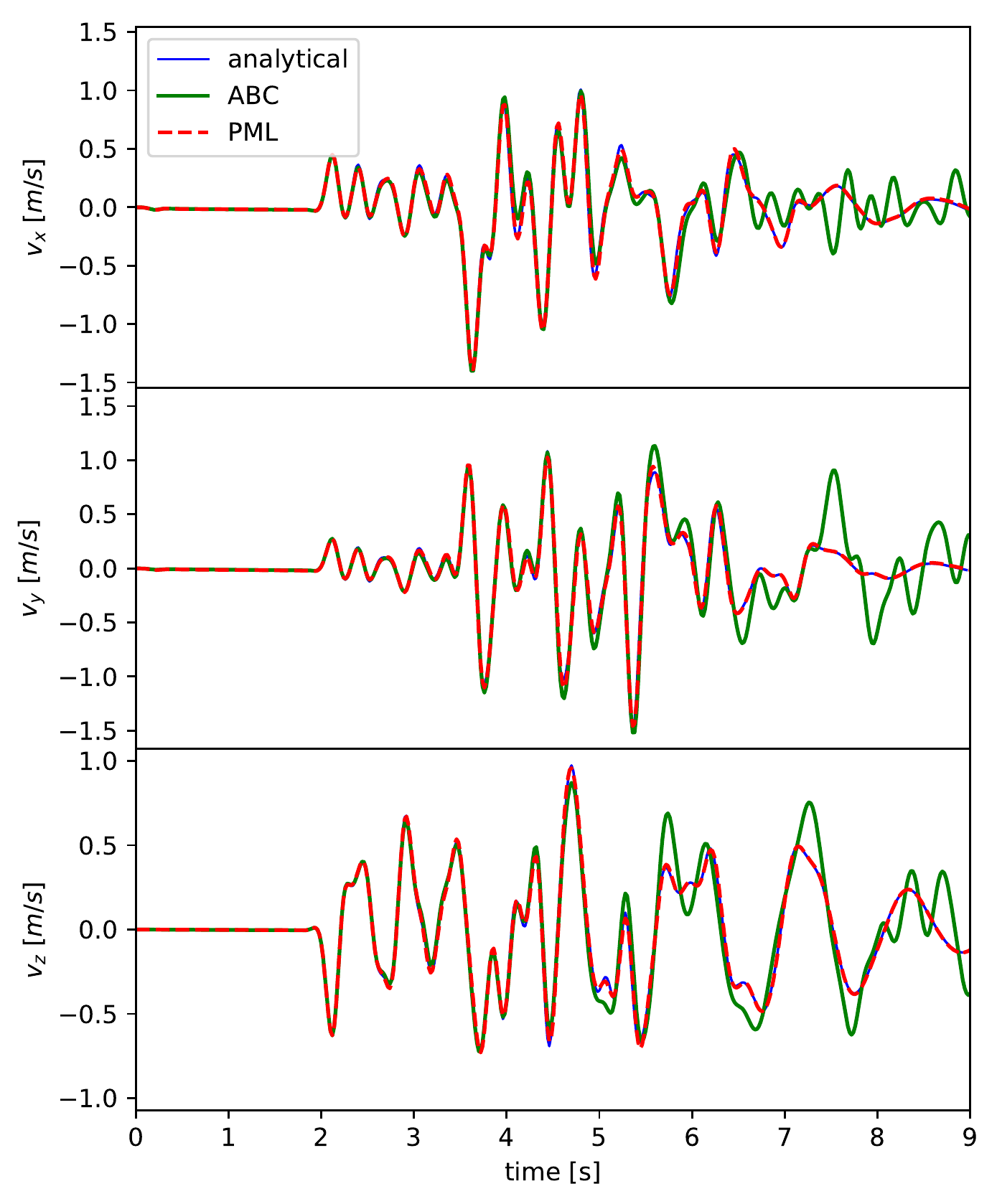}}{Receiver 9}%
     \end{subfigure}
    \caption{The LOH1 benchmark problem with degree $N = 5$ polynomial approximation. The PML solution converges spectrally to the exact solution.}
    \label{fig:loh1_o5}
\end{figure}
The solutions are displayed in Figure \ref{fig:loh1_o3} for $N =3$, and  in Figure \ref{fig:loh1_o5} for $N =5$, for Receivers 4, 5, 6 and 9. The solutions for the remaining Receivers are placed in Appendix \ref{sec:other_stations}. Note that initially the ABC and PML solutions match the analytical solution very well. However, at later times the ABC solution  is polluted by numerical reflections arriving from the artificial boundaries. 
Note in particular, for $ N = 3$, there are numerical errors which are due to numerical resolution, and can be eliminated by $p-$ or $h-$refinement. As expected, for $N = 5$ these errors diminish spectrally, the PML solutions match the analytical solution excellently, and remains accurate for the entire simulation duration. For the ABC, the dominant errors are the errors introduced by artificial reflections, these can never diminish with $p-$ or $h-$refinement.
\section{Concluding remarks}
We have developed a stable  DG method for PMLs truncating  3D and 2D linear elastic wave equation.
In a series of  papers  \cite{Duruthesis2012,DuKrSIAM,KDuru2016,DuruKozdonKreiss2016} we established theoretically the stability of the PML subject to linear well-posed boundary conditions, and developed provably stable finite difference schemes for the PML IBVP. In this paper, we extend some of these results to the DG approximation of the elastic wave equation, in first order form, and develop new theories.

We began by constructing continuous energy  estimates, in the time domain for the elastic wave equation, and in the Laplace space for the PML, by   assuming  variations only in one space dimension,  in 2D for a PML edge problem, and in 3 space dimension for a PML corner problem.
We developed a DG method for the linear wave equation  using physically motivated numerical fluxes and penalty parameters, which are compatible with all well-posed, internal and external, boundary conditions. When the PML damping vanishes, by construction, our choice of penalty parameters yield an upwind scheme and a discrete energy estimate analogous to the continuous energy estimate.  

To ensure numerical stability of the discretization when PML damping is present, it is necessary to extend the  numerical DG fluxes, and the numerical  inter-element and boundary  procedures, to the PML auxiliary differential equations. This is crucial  for deriving discrete energy estimates analogous to the continuous energy estimates.   
Numerical solutions are evolved in time using the high order arbitrary  derivative (ADER)  time stepping  scheme of the same order of accuracy with the spatial discretization.  By combining the DG spatial approximation with the high order ADER  time stepping  scheme and the accuracy of the PML we obtain a spectrally  accurate wave propagation solver in the time domain.  Numerical experiments are presented verifying numerical stability and high order accuracy. More importantly, the numerical experiments  show that the stabilizing PML flux fluctuations are necessary for numerical stability.   

Further, we present numerical experiments in 3D verifying the analysis, accuracy and demonstrating the effectiveness of the PML in seismolgical applications.
The 2D and 3D production code, for elastic waves simulation, have been implemented in ExaHyPE (www.exahype.eu), a simulation engine for hyperbolic PDEs, on adaptive Cartesian meshes, for exa-scale supercomputers. This software, ExaHyPE, is open source and publicly available.

%
%
\begin{acknowledgements}
The work presented in this paper was enabled by funding from the European Union's Horizon 2020 research and innovation program under grant agreement No 671698 (ExaHyPE). 
\newline
{}\hfill{\includegraphics[angle=90, width=0.15\textwidth]{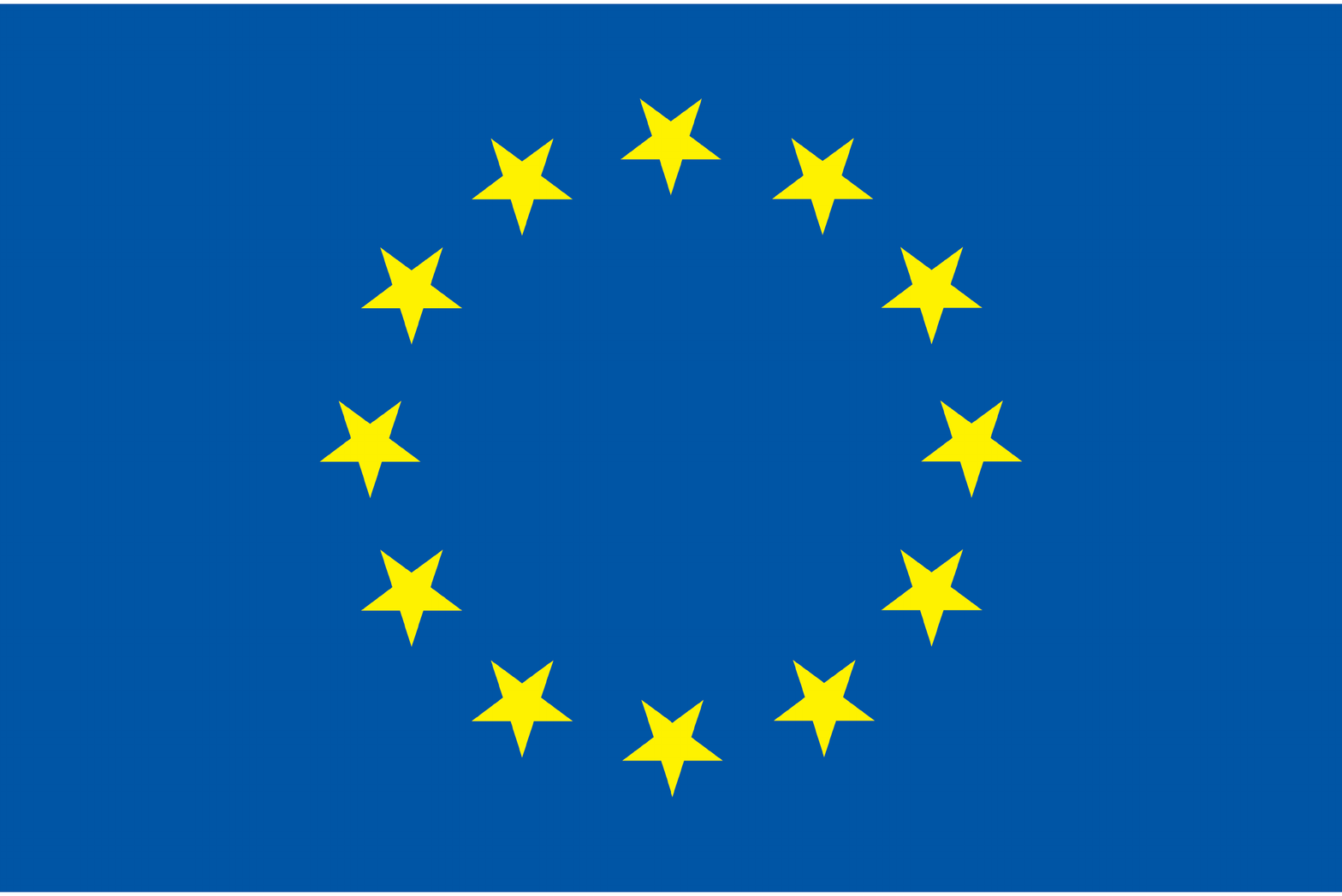}}

A.-A.G.\ acknowledges additional support by the German Research Foundation (DFG) (projects no.~KA 2281/4-1, GA 2465/2-1, GA 2465/3-1), by BaCaTec (project no.~A4) and BayLat, by KONWIHR -- the Bavarian Competence Network for Technical and Scientific High Performance Computing (project NewWave), by KAUST-CRG (GAST, grant no.~ORS-2016-CRG5-3027 and FRAGEN, grant no.~ORS-2017-CRG6 3389.02), by the European Union's Horizon 2020 research and innovation program (ChEESE, grant no.~823844 and TEAR, grant no.~852992).
\newline
{}\hfill{\includegraphics[angle=0, width=0.2\textwidth]{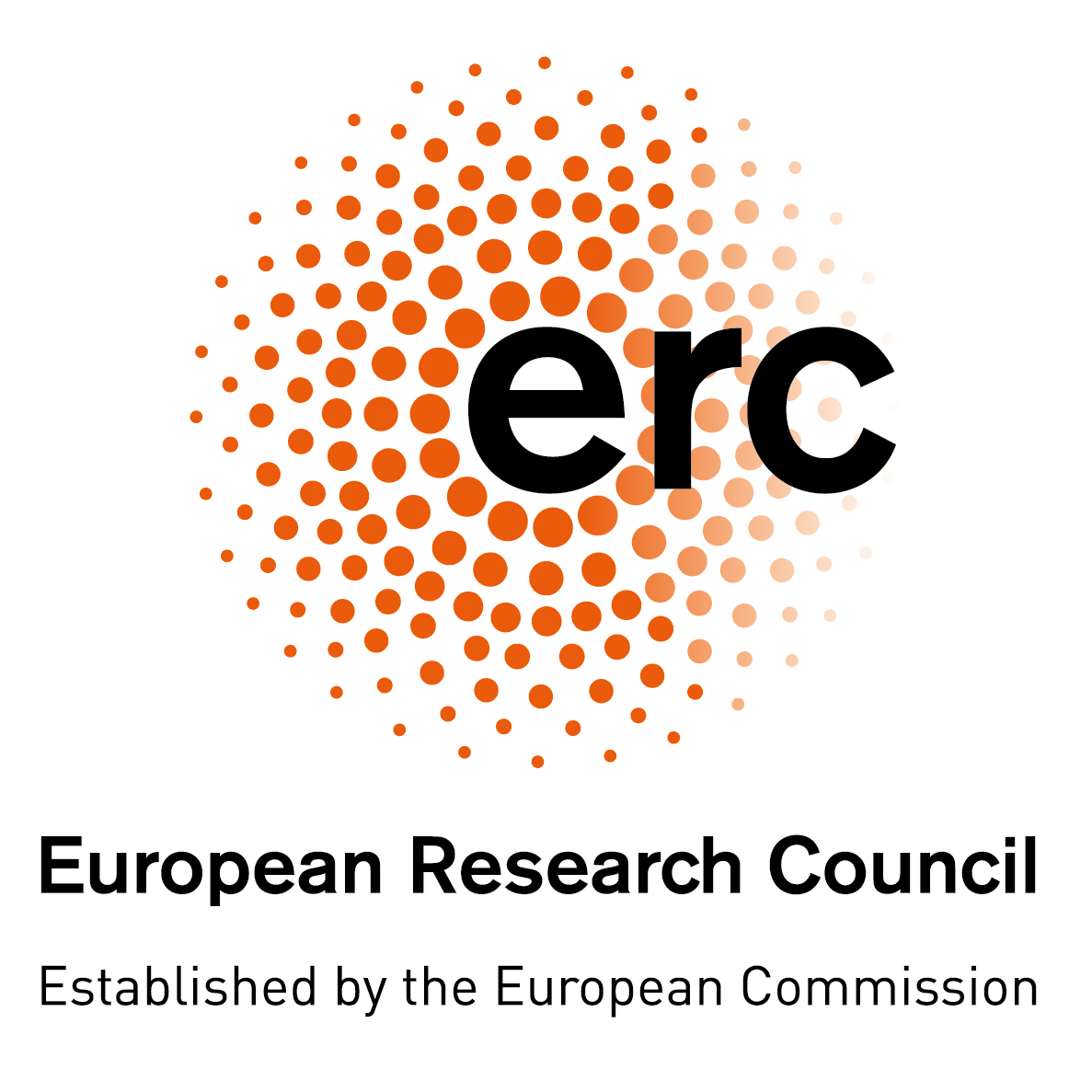}}
\end{acknowledgements}


%
%

\appendix
\section{Other Receivers}\label{sec:other_stations}

\begin{figure}[h!]
\begin{subfigure}
    \centering
\stackunder[5pt]{\includegraphics[width=0.2\textwidth]{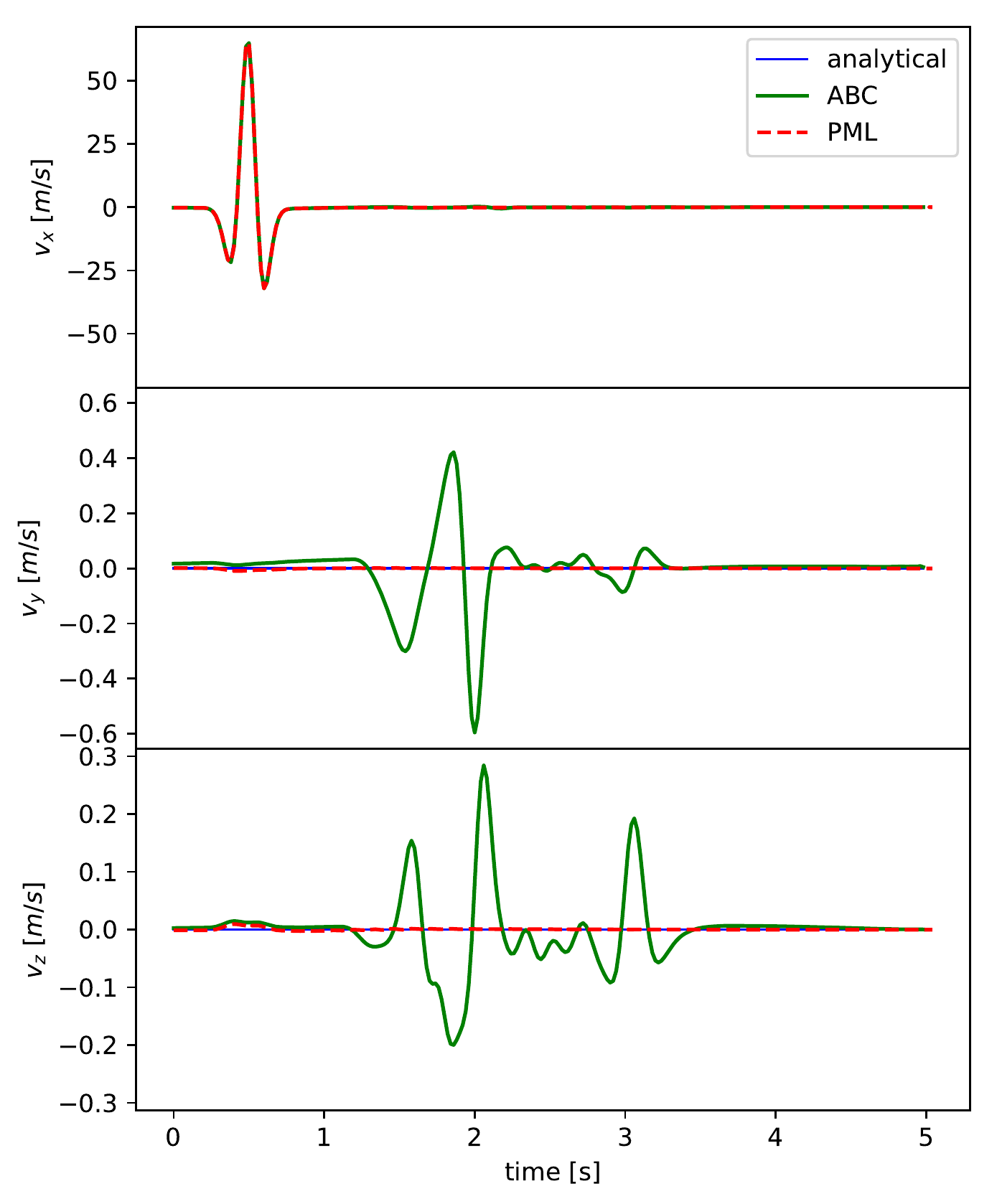}}{Receiver 1}%
\hspace{0.0cm}%
\stackunder[5pt]{\includegraphics[width=0.2\textwidth]{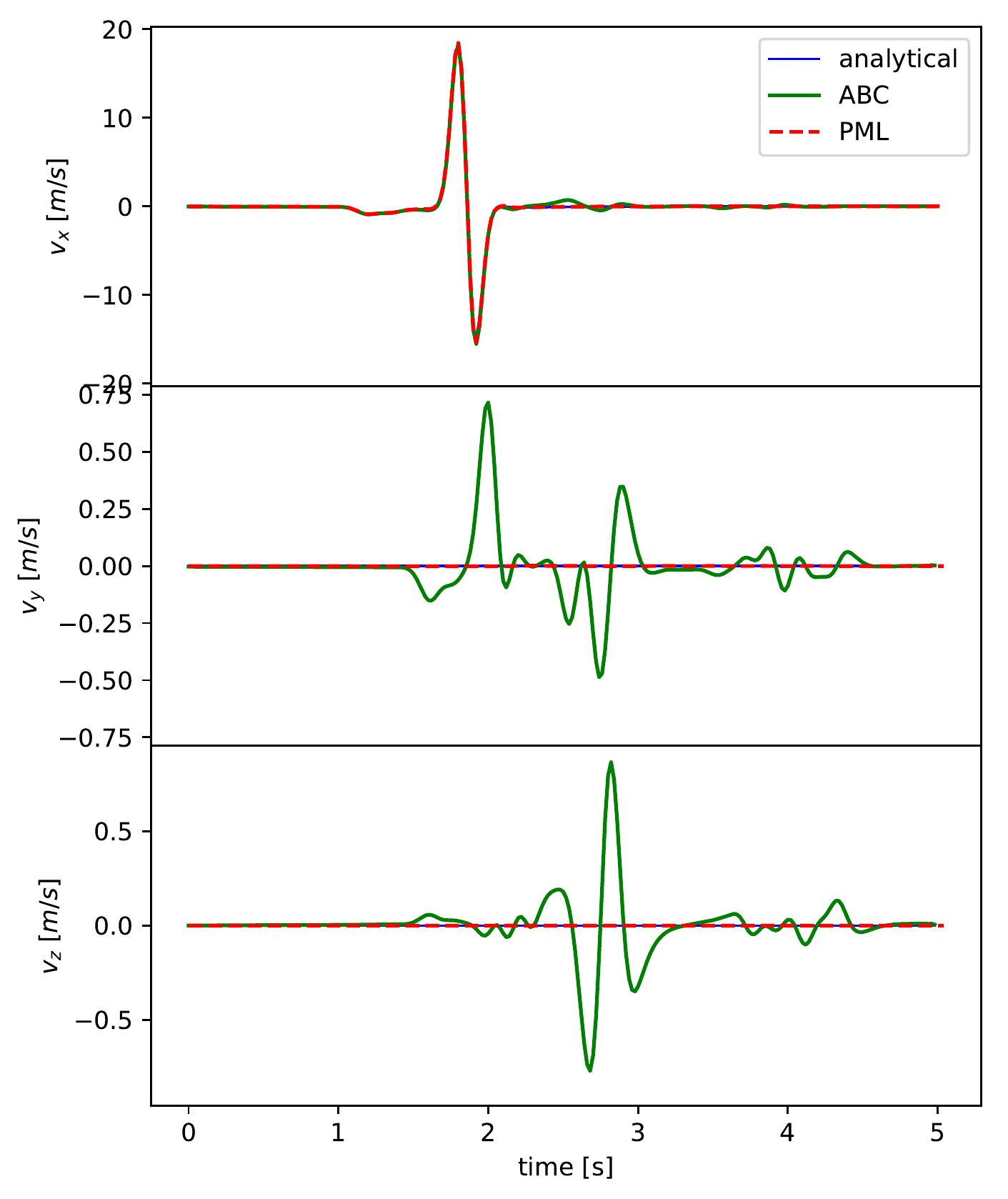}}{Receiver 2}%
\hspace{0.0cm}%
\stackunder[5pt]{\includegraphics[width=0.2\textwidth]{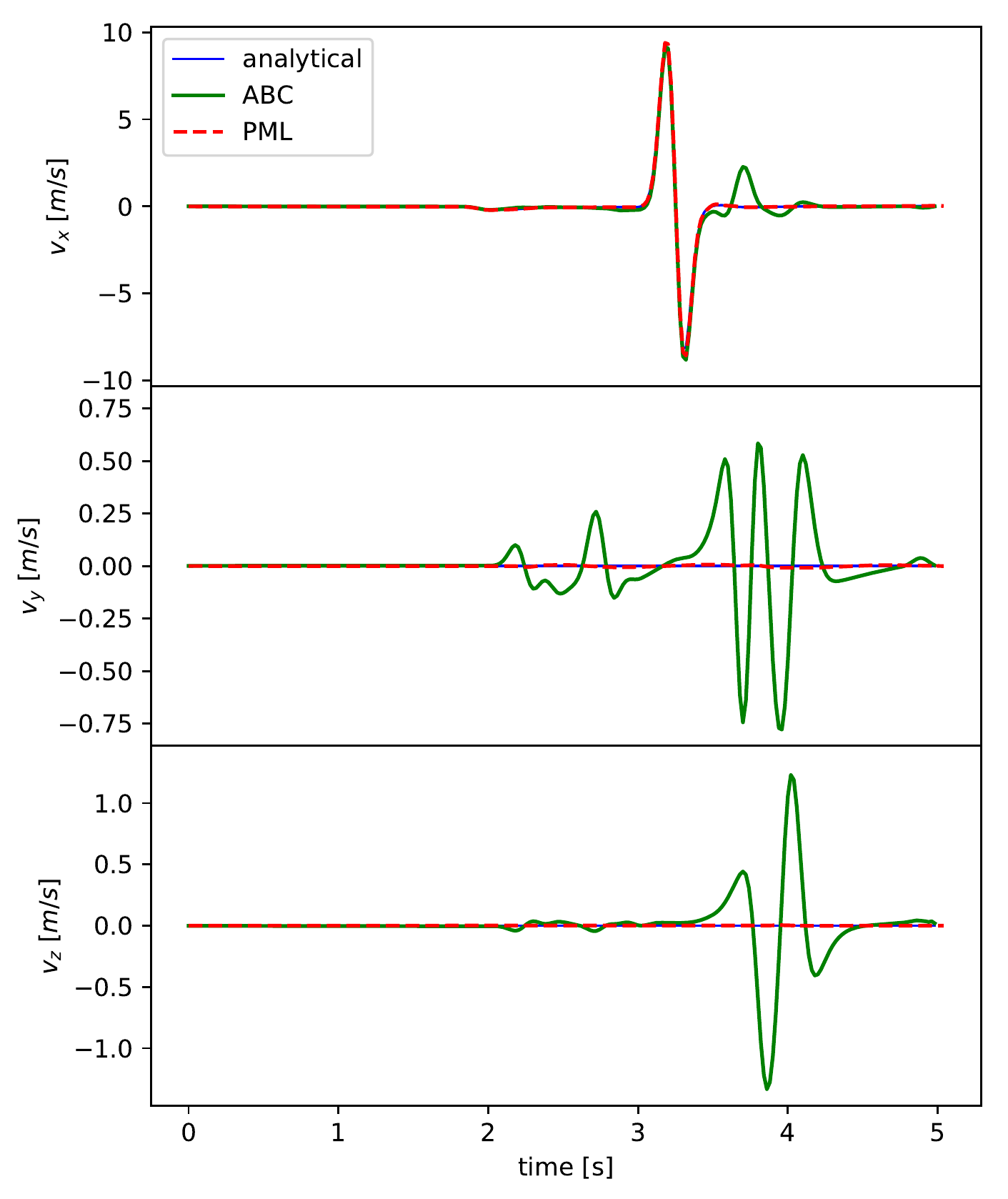}}{Receiver 3}%
\hspace{0.0cm}%
\stackunder[5pt]{\includegraphics[width=0.2\textwidth]{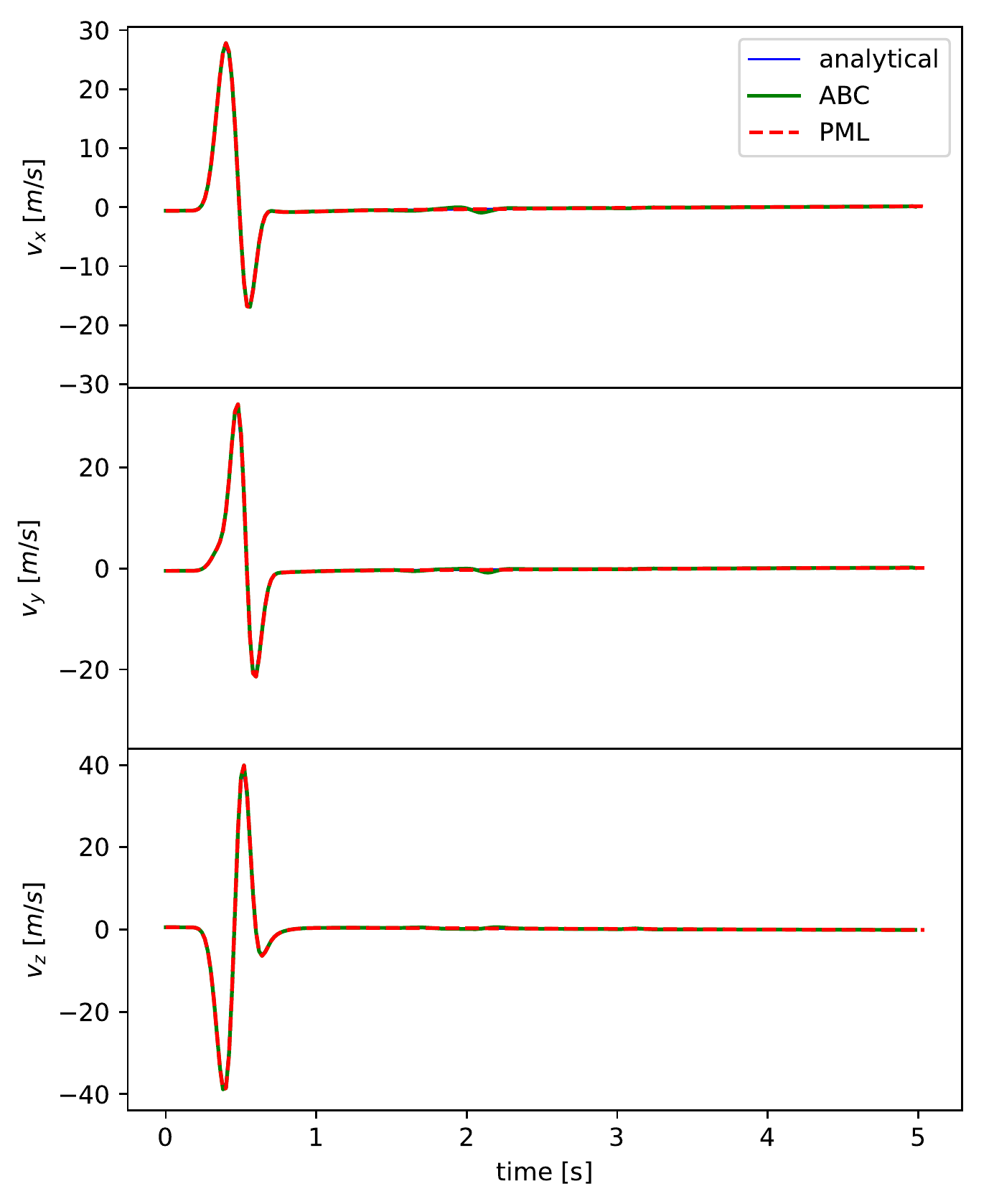}}{Receiver 7}%
\hspace{0.0cm}%
\stackunder[5pt]{\includegraphics[width=0.2\textwidth]{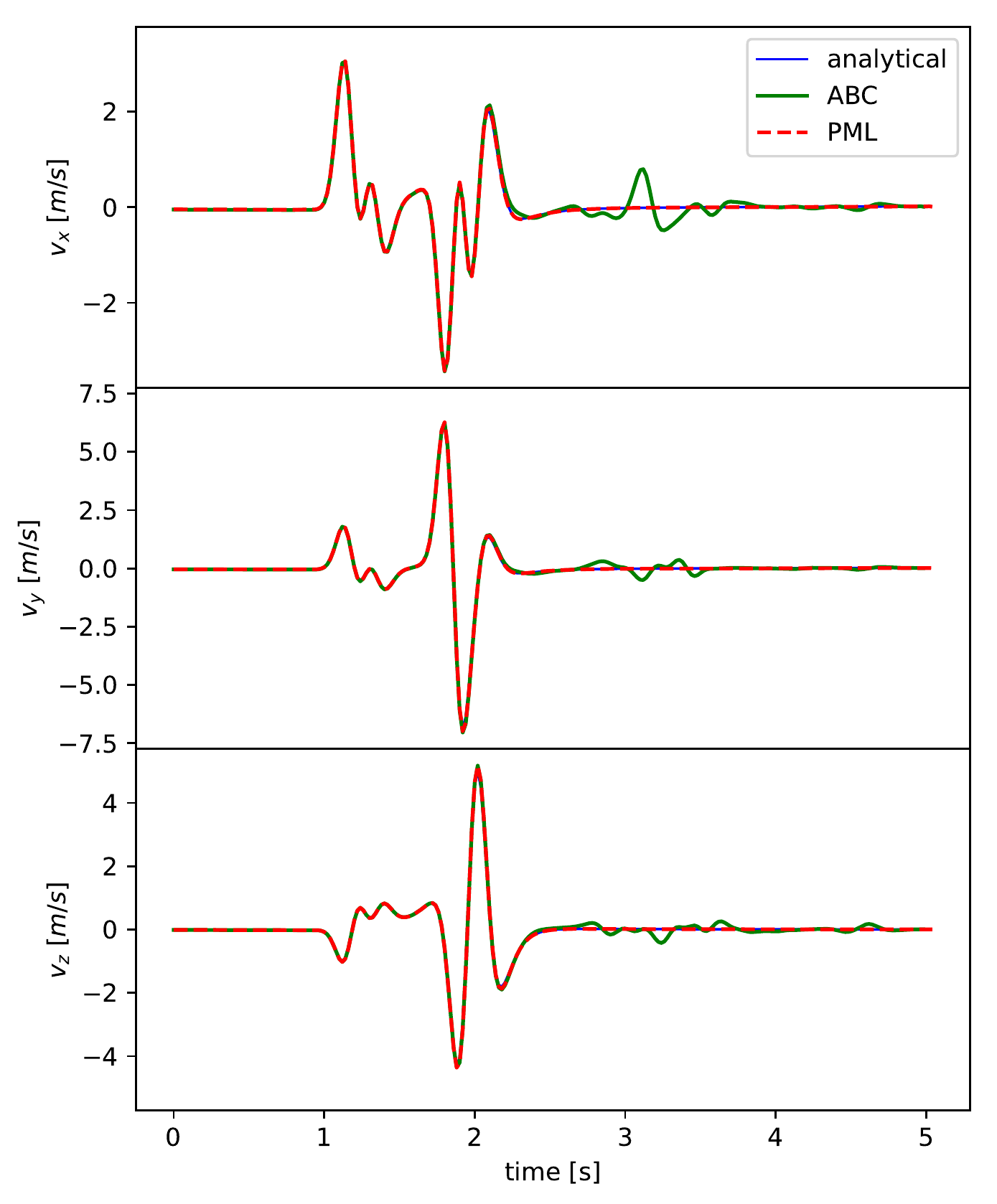}}{Receiver 8}%
     \end{subfigure}
    \caption{The homogeneous half space problem}
    \label{fig:halfspace_other_Receivers}
\end{figure}

\begin{figure}[h!]
\begin{subfigure}
    \centering
\stackunder[5pt]{\includegraphics[width=0.2\textwidth]{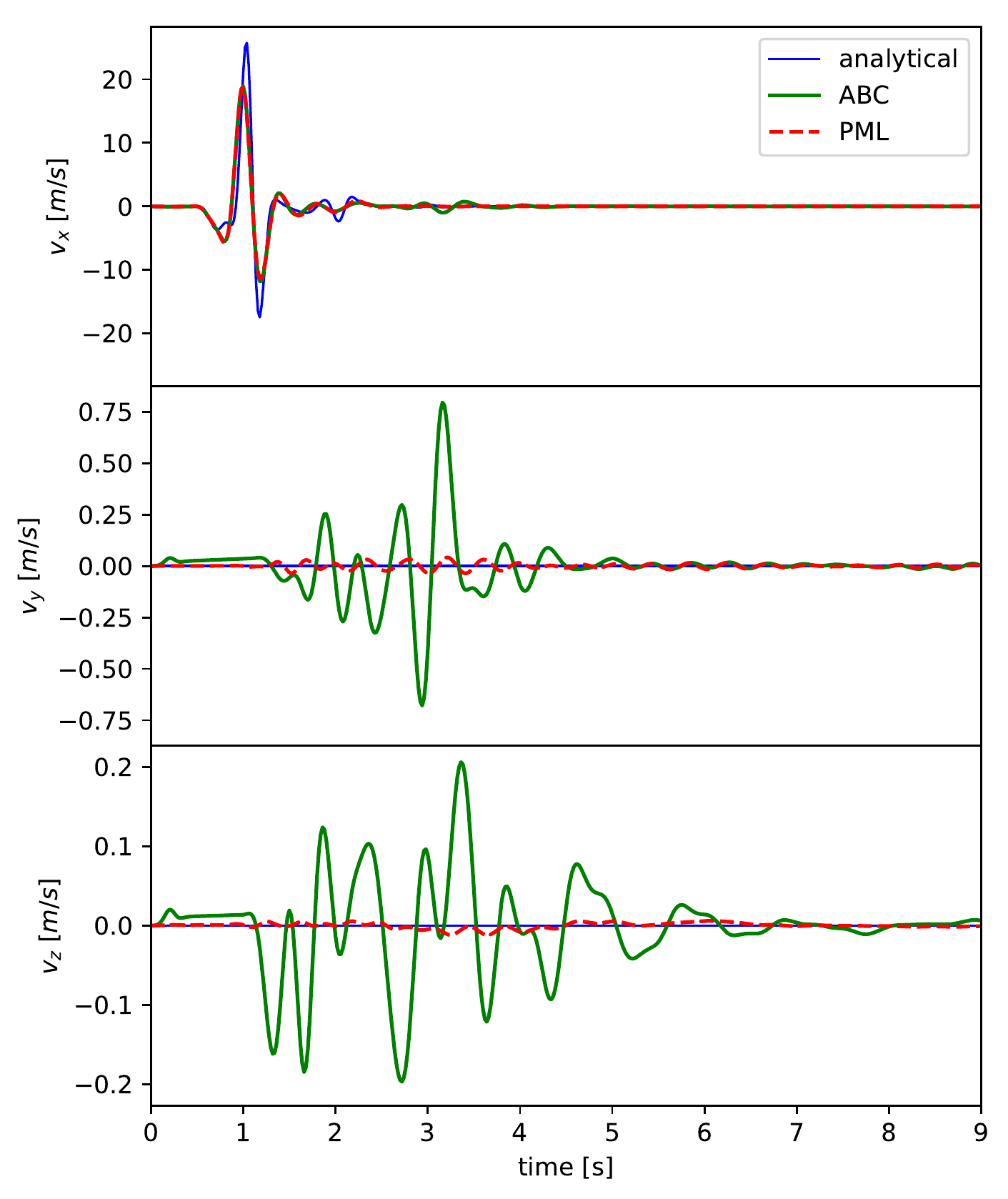}}{Receiver 1}%
\hspace{0.0cm}%
\stackunder[5pt]{\includegraphics[width=0.2\textwidth]{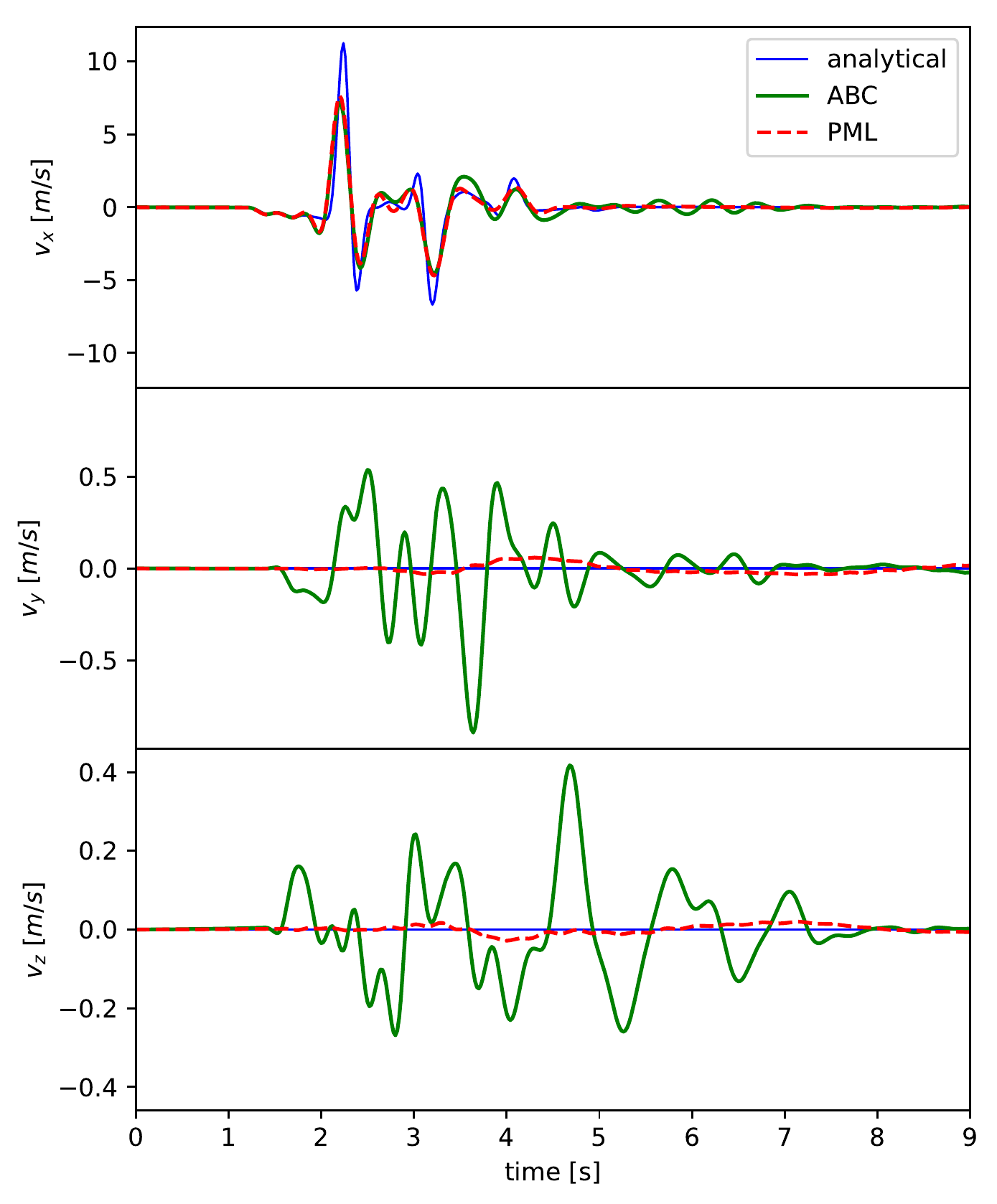}}{Receiver 2}%
\hspace{0.0cm}%
\stackunder[5pt]{\includegraphics[width=0.2\textwidth]{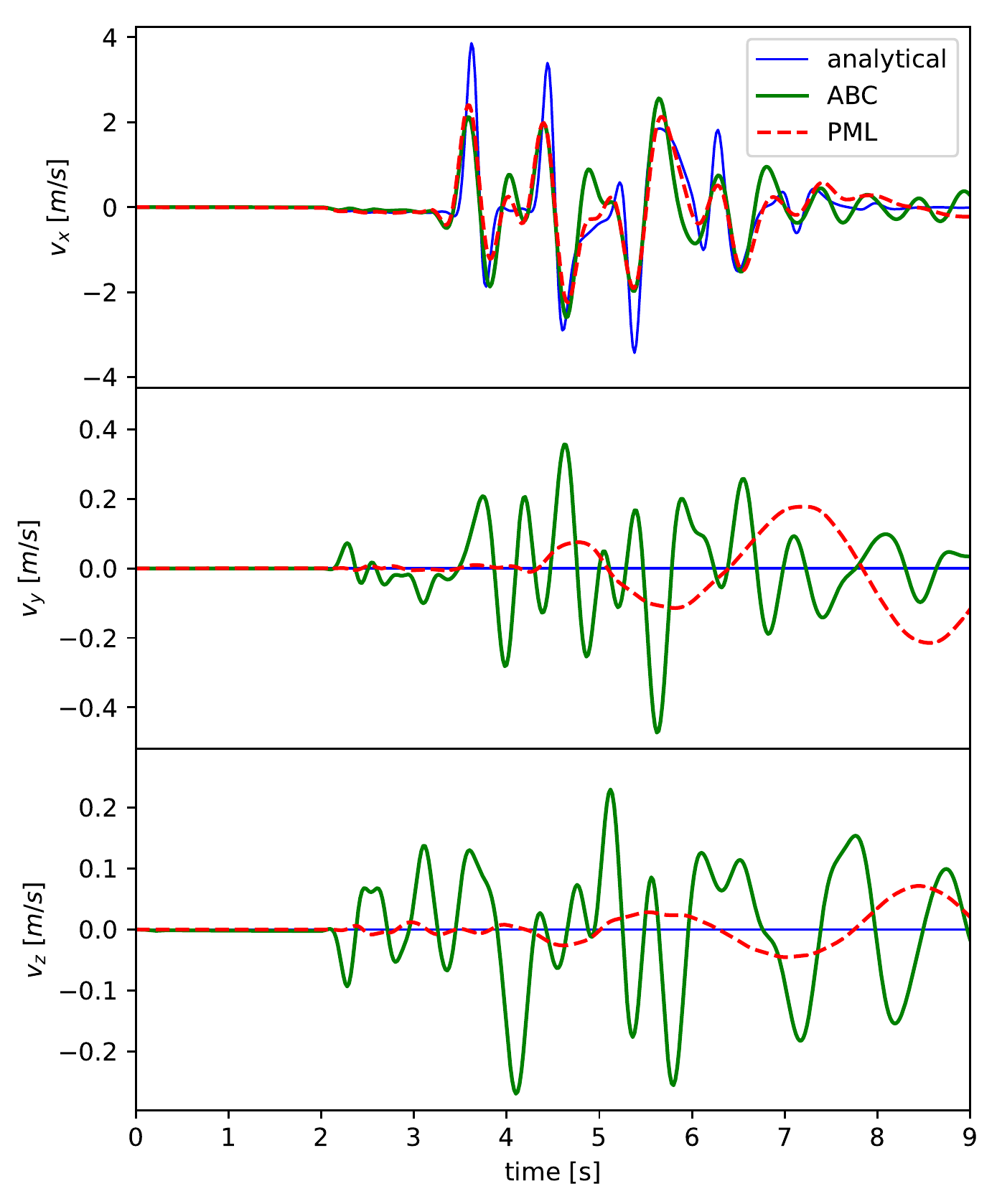}}{Receiver 3}%
\hspace{0.0cm}%
\stackunder[5pt]{\includegraphics[width=0.2\textwidth]{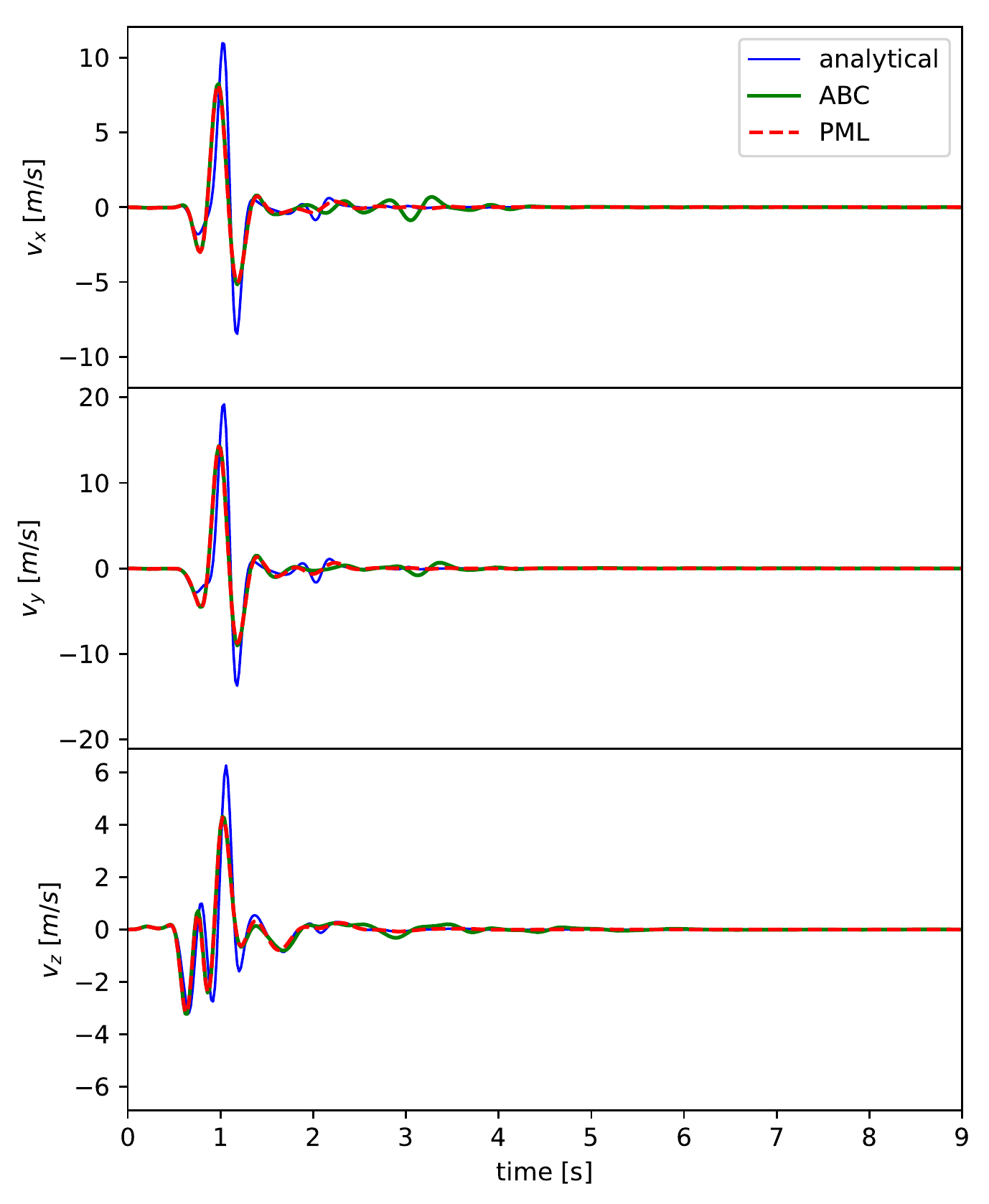}}{Receiver 7}%
\hspace{0.0cm}%
\stackunder[5pt]{\includegraphics[width=0.2\textwidth]{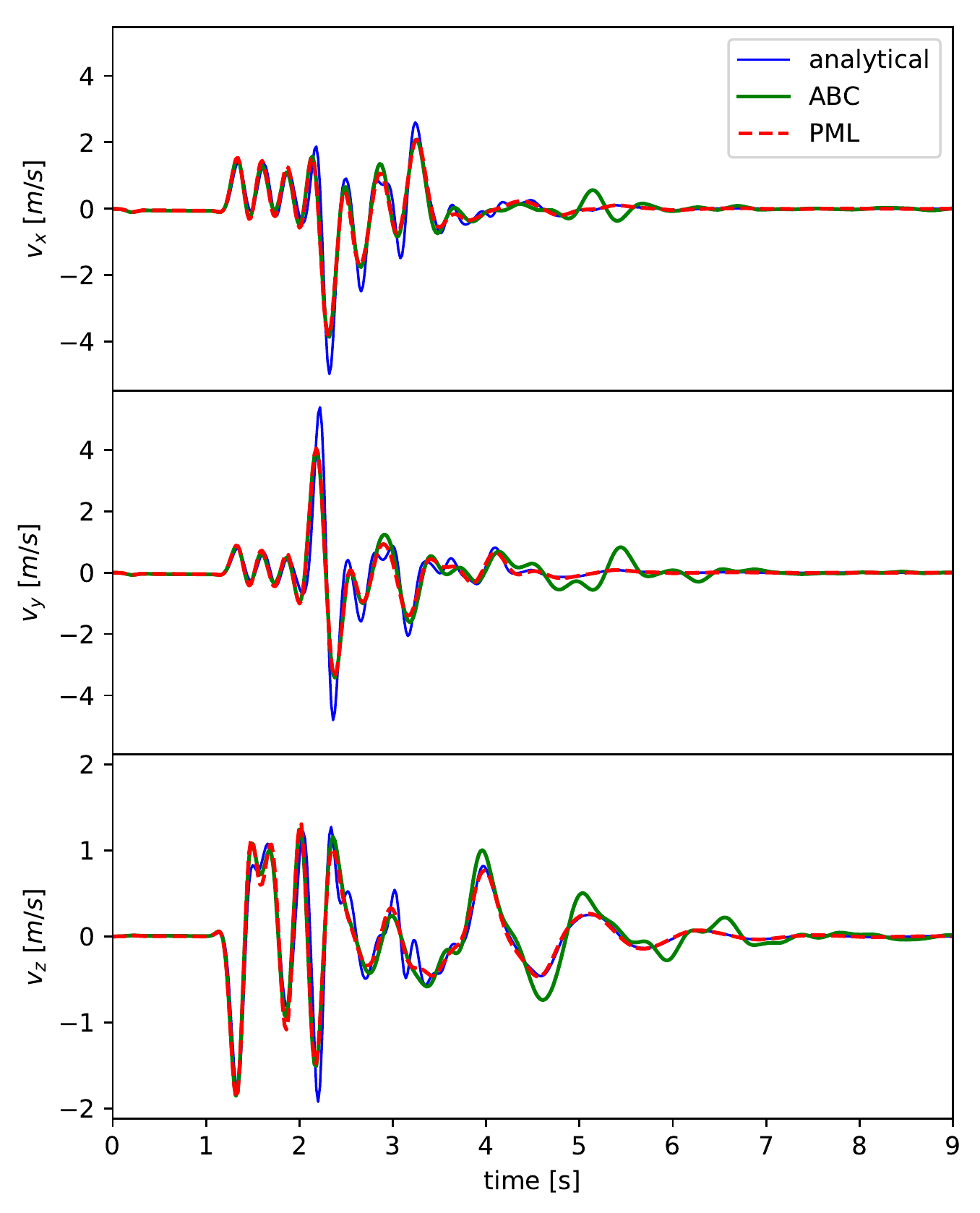}}{Receiver 8}%
     \end{subfigure}
    \caption{The LOH1 benchmark problem with degree $N = 3$ polynomial approximation}
    \label{fig:loh1_o3_other_Receivers}
\end{figure}

\begin{figure}[h!]
\begin{subfigure}
    \centering
\stackunder[5pt]{\includegraphics[width=0.2\textwidth]{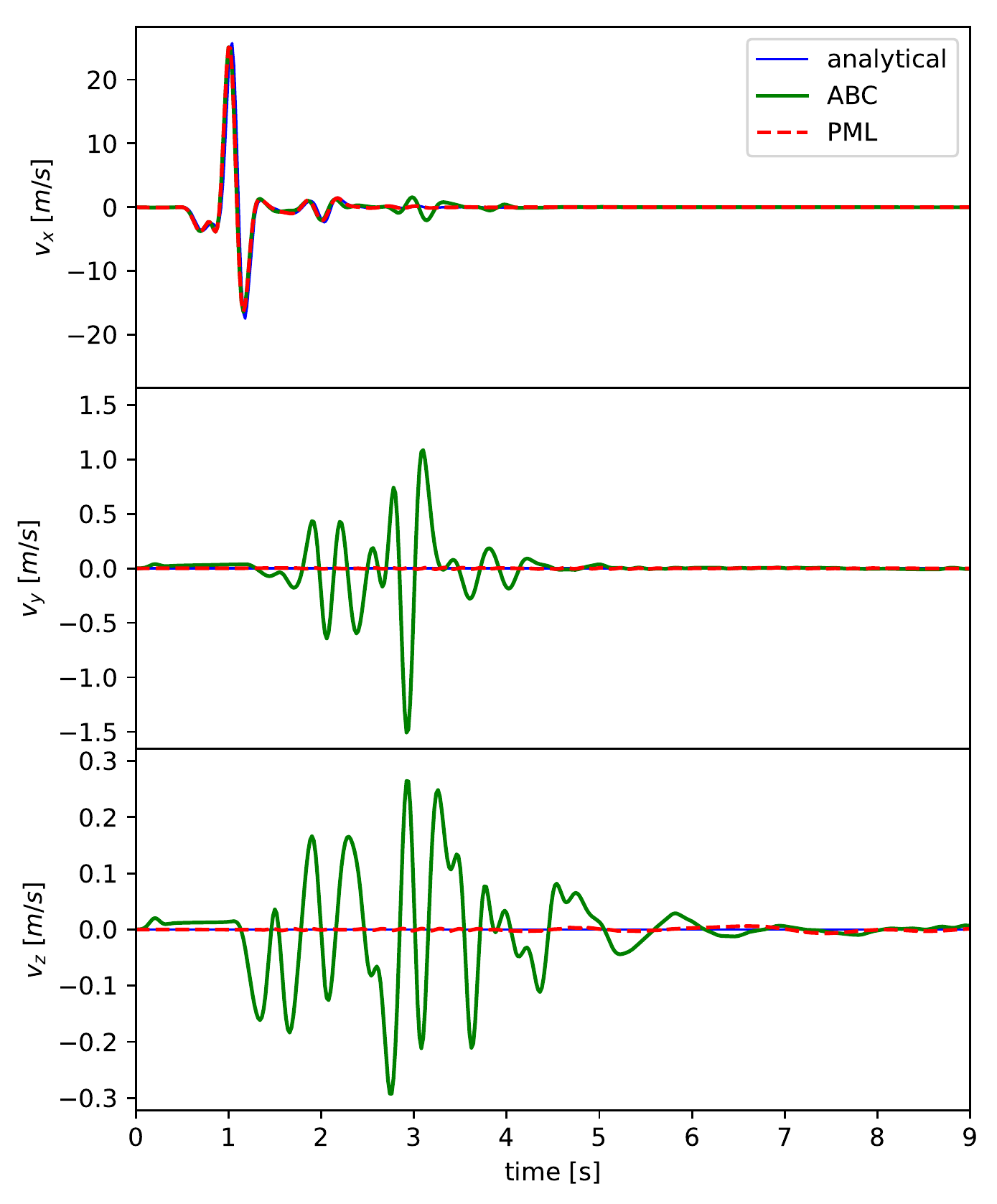}}{Receiver 1}%
\hspace{0.0cm}%
\stackunder[5pt]{\includegraphics[width=0.2\textwidth]{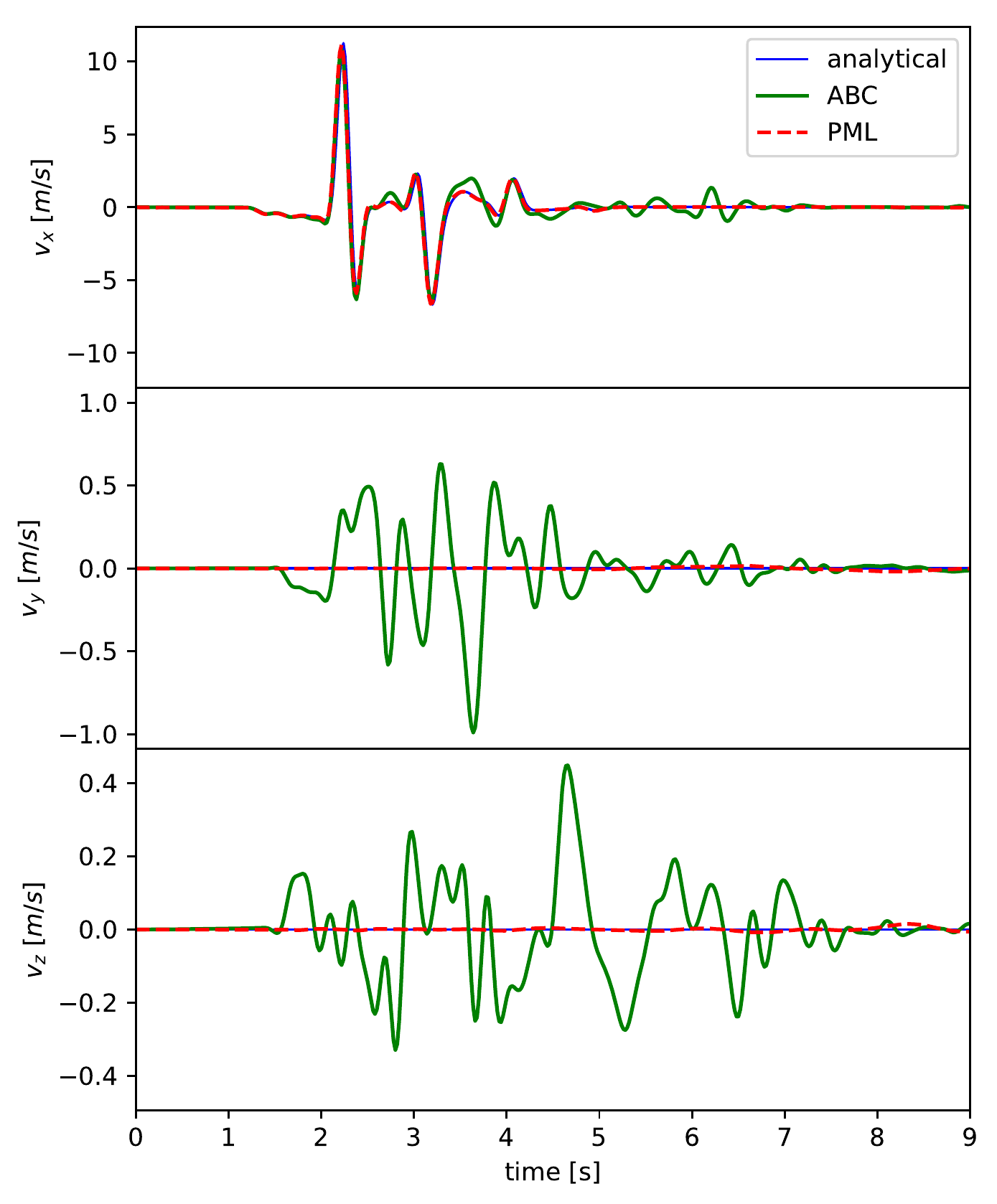}}{Receiver 2}%
\hspace{0.0cm}%
\stackunder[5pt]{\includegraphics[width=0.2\textwidth]{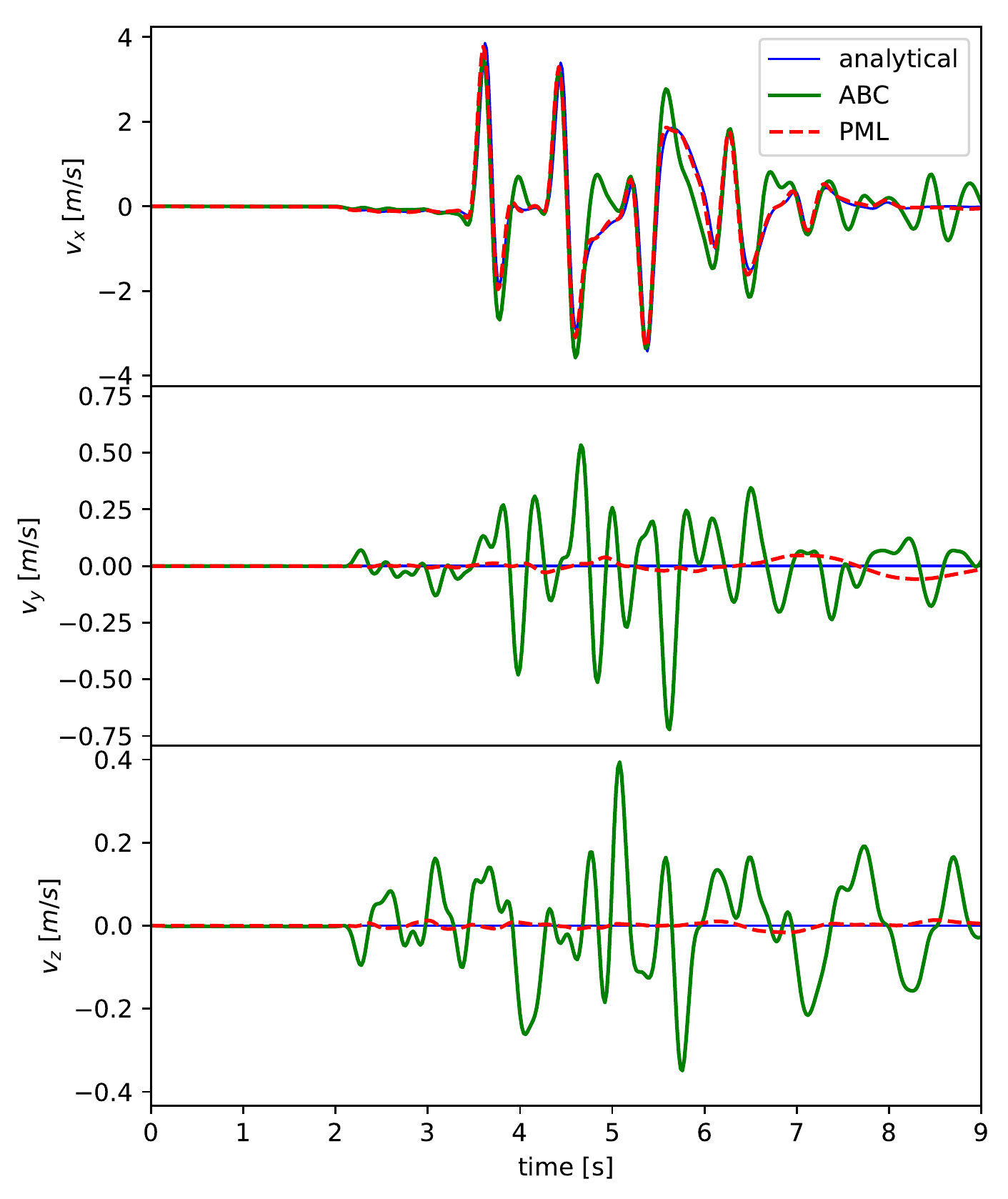}}{Receiver 3}%
\hspace{0.0cm}%
\stackunder[5pt]{\includegraphics[width=0.2\textwidth]{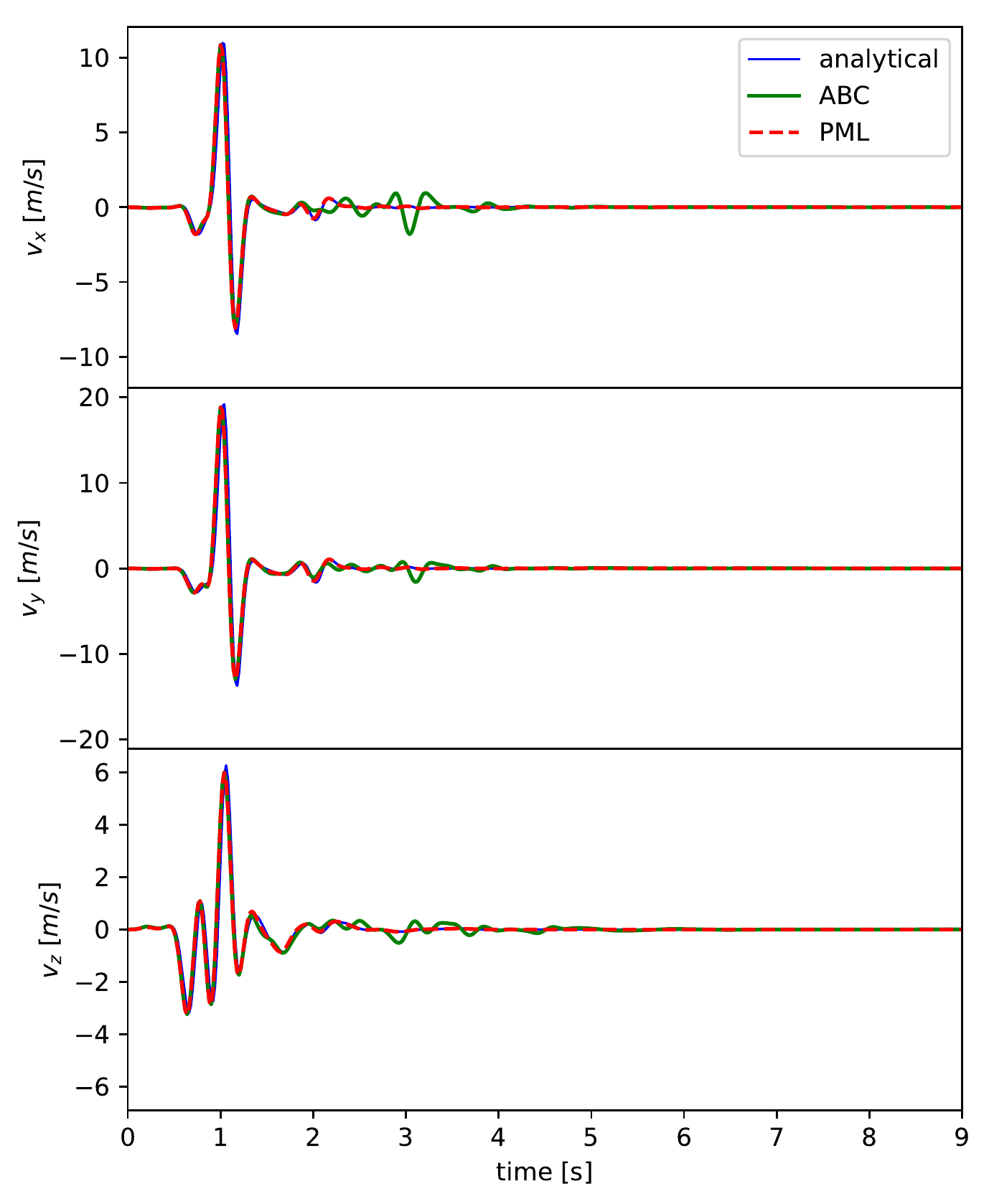}}{Receiver 7}%
\hspace{0.0cm}%
\stackunder[5pt]{\includegraphics[width=0.2\textwidth]{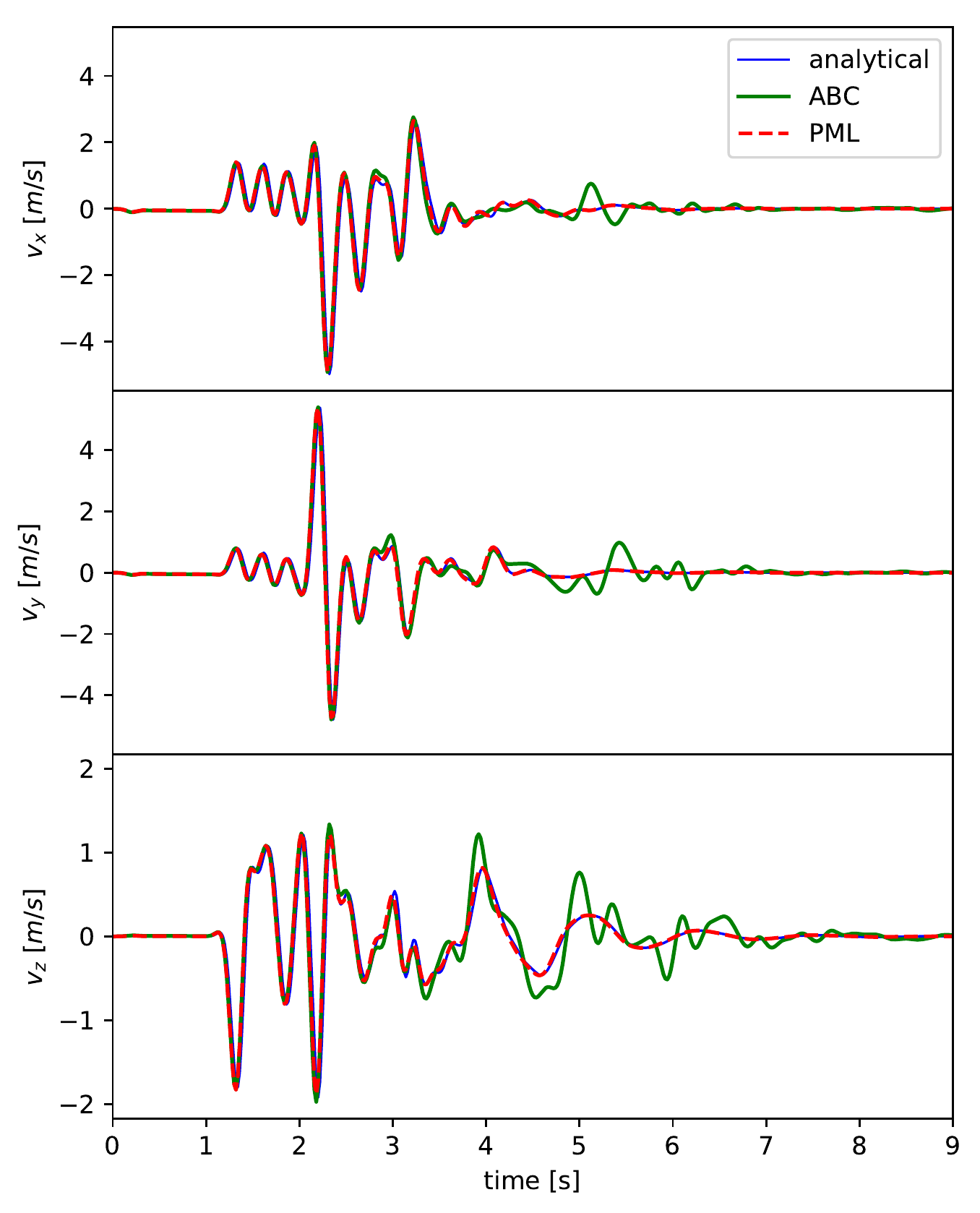}}{Receiver 8}%
     \end{subfigure}
    \caption{The LOH1 benchmark problem with degree $N = 5$ polynomial approximation}
    \label{fig:loh1_o5_other_Receivers}
\end{figure}

\section{ Source terms}\label{sec:source_terms}
We will derive a representation of the modified source terms in the physical space.
Let $t\ge 0$ denote time, and $f(t)$ and $g(t)$  are real functions of exponential order. We denote their corresponding Laplace transforms $\widetilde{f}(s) = \mathcal{L}\left(f(t)\right)$, $\widetilde{g}(s) = \mathcal{L}\left(g(t)\right)$.
We also define the convolution operator $*$, 
\[
g(t)*f(t) = \int_{0}^{t} g(t-\tau)f(\tau) d\tau.
\]
Note that $\mathcal{L}\left(g(t)*f(t)\right) = \widetilde{f}(s)\widetilde{g}(s)$.
To succeed we will use the Dirac  delta distribution $\delta(t)$
\[
\int_{0}^{+\infty} \delta(t) dt = 1,
\]
with
\[
\delta(t)*f(t) = f(t), \quad \delta^{\prime}(t)*f(t) = f^{\prime}(t) + \delta(t)f(0).
\]
We also recall 
\[
\mathcal{L}\left(\delta^{n}(t)\right) = s^n, \quad n = 0, 1, 2, \cdots  ,
\]
where  $\delta^{n}(t)$ is the $n$-th derivative of $\delta(t)$.

We have
\[
\mathcal{L}\left(\delta^{\prime}(t)*f(t) \right) = s\widetilde{f}(s).
\]
Consider the modified source terms in the Laplace domain
\begin{align}
\widetilde{\mathbf{F}} =     \mathbf{P}\left(S_x \widetilde{\mathbf{F}}_{Q}  - \sum_{\xi = x,y, z} \frac{d_{\xi}S_x }{s + \alpha_{\xi} + d_{\xi}} \widetilde{\mathbf{F}}_{w_\xi}\right).
\end{align}
If $d_{\xi} = d_x = d \ge 0$ and $\alpha_{\xi} = \alpha_{x} = \alpha \ge 0$, then 
\begin{align}
S_x \widetilde{\mathbf{F}}_{Q}=    \left(1 + \frac{d}{s + \alpha} \right)\widetilde{\mathbf{F}}_{Q}, \quad   \frac{S_x }{s + \alpha_{\xi} + d_{\xi}} \widetilde{\mathbf{F}}_{w_\xi} =  \frac{1 }{s + \alpha} \widetilde{\mathbf{F}}_{w_\xi},
\end{align}
and
{
\small
\begin{align}\label{eq:physical_data}
\mathcal{L}^{-1}\left(S_x \widetilde{\mathbf{F}}_{Q}\right) =    \left(\delta(t) + {d} e^{-\alpha t} \right) * {\mathbf{F}}_{Q}\left(x,y,z, t\right), 
\quad
  \mathcal{L}^{-1}\left( \frac{1 }{s + \alpha} \widetilde{\mathbf{F}}_{w_\xi}\right) =    e^{-\alpha t} * {\mathbf{F}}_{w_\xi} \left(x,y,z, t\right), \quad \Re{s} > - \alpha.
\end{align}
}
\section{Proof of Theorem \ref{theo:PML_2D_Edge_Laplace}}\label{sec:proof_theorem2}
\begin{proof}
  As before, multiply  \eqref{eq:weak_Laplace_2D} with $\widetilde{\mathbf{Q}}^{\dagger}$ from the left and integrate over the whole spatial domain.
   Integration-by-parts gives
  {
  \begin{equation}\label{eq:weak_Laplace_2D_1}
\begin{split}
\int_{\Omega}\left( \left(sS_x\right)\widetilde{\mathbf{Q}}^{\dagger} \mathbf{P}^{-1} \widetilde{\mathbf{Q}} \right)dxdy dz &= \sum_{\xi = x, y}\frac{1}{2}\int_{\Omega}\left(\left[\widetilde{\mathbf{Q}}^\dagger\left(\mathbf{A}_{\xi}\frac{\partial{\widetilde{\mathbf{Q}}}}{\partial \xi}\right)-\left(\mathbf{A}_{\xi}\frac{\partial{\widetilde{\mathbf{Q}}}}{\partial \xi}\right)^\dagger \mathbf{Q}\right] \right)dxdydz \\
&+  \sum_{\xi = x, y}\int_{\widetilde{\Gamma}}\left(\left[\widetilde{\mathbf{Q}}^{\dagger}{A}_\xi \widetilde{\mathbf{Q}}\right] \Big|_{-1}^{1}\right)\frac{dxdydz}{d\xi} 
+ \int_{\Omega}\left(\widetilde{\mathbf{Q}}^{\dagger} \mathbf{P}^{-1}\widetilde{\mathbf{F}} \right)dxdydz.
  \end{split}
  \end{equation}
  }
  Adding  \eqref{eq:weak_Laplace_2D_1} to its  complex conjugate, the spatial derivative terms vanish, yielding
  {
  \begin{equation}\label{eq:weak_Laplace_2D_2}
\begin{split}
\|\sqrt{\Re(s S_x)}\widetilde{\mathbf{Q}} (\cdot ,\cdot ,\cdot , s)\|_{P}^2 = \frac{1}{2}\left[\left(\widetilde{\mathbf{Q}},\widetilde{\mathbf{F}} \right)_{{P}}  + \left(\widetilde{\mathbf{F}},\widetilde{\mathbf{Q}} \right)_{{P}}\right] +  \sum_{\xi = x, y}\int_{\widetilde{\Gamma}}\left(\left[\widetilde{\mathbf{Q}}^{\dagger}{A}_x \widetilde{\mathbf{Q}}\right] \Big|_{-1}^{1}\right)\frac{dxdydz}{d\xi} .
  \end{split}
  \end{equation}
  }
  Using Cauchy-Schwarz inequality and the fact \eqref{eq:BT_Negative} in the right hand side of \eqref{eq:weak_Laplace_2D_2} completes the proof.
 \end{proof}
   
\section{Proof of Theorem \ref{theo:PML_3D_Corner_Laplace}}\label{sec:proof_theorem3}
\begin{proof}
  As above, multiply  \eqref{eq:weak_Laplace_Corner} with $\widetilde{\mathbf{Q}}^{\dagger}$ from the left and integrate over the whole spatial domain.
   Integration-by-parts gives
  {
  \small
  \begin{equation}\label{eq:weak_Laplace_3D_1}
\begin{split}
\int_{\Omega}\left( \left(sS_x\right)\widetilde{\mathbf{Q}}^{\dagger} \mathbf{P}^{-1} \widetilde{\mathbf{Q}} \right)dxdy dz &= \sum_{\xi = x, y,z}\frac{1}{2}\int_{\Omega}\left(\left[\widetilde{\mathbf{Q}}^\dagger\left(\mathbf{A}_{\xi}\frac{\partial{\widetilde{\mathbf{Q}}}}{\partial \xi}\right)-\left(\mathbf{A}_{\xi}\frac{\partial{\widetilde{\mathbf{Q}}}}{\partial \xi}\right)^\dagger \mathbf{Q}\right] \right)dxdydz \\
&+  \sum_{\xi = x, y,z}\int_{\widetilde{\Gamma}}\left(\left[\widetilde{\mathbf{Q}}^{\dagger}{A}_\xi \widetilde{\mathbf{Q}}\right] \Big|_{-1}^{1}\right)\frac{dxdydz}{d\xi} 
+ \int_{\Omega}\left(\widetilde{\mathbf{Q}}^{\dagger} \mathbf{P}^{-1}\widetilde{\mathbf{F}} \right)dxdydz.
  \end{split}
  \end{equation}
  }
  Adding the complex conjugate of the product, the spatial derivative terms vanish, yielding
  {\small
  \begin{equation}\label{eq:weak_Laplace_3D_2}
\begin{split}
\|\sqrt{\Re(s S_x)}\widetilde{\mathbf{Q}} (\cdot ,\cdot ,\cdot , s)\|_{P}^2 = \frac{1}{2}\left[\left(\widetilde{\mathbf{Q}},\widetilde{\mathbf{F}} \right)_{{P}}  + \left(\widetilde{\mathbf{F}},\widetilde{\mathbf{Q}} \right)_{{P}}\right] +  \sum_{\xi = x, y, z}\int_{\widetilde{\Gamma}}\left(\left[\widetilde{\mathbf{Q}}^{\dagger}{A}_\xi \widetilde{\mathbf{Q}}\right] \Big|_{-1}^{1}\right)\frac{dxdydz}{d\xi} , \quad \Re{s} > 0.
  \end{split}
  \end{equation}
  }
  Using Cauchy-Schwarz inequality and the fact \eqref{eq:BT_Negative} in the right hand side of \eqref{eq:weak_Laplace_3D_2} completes the proof.
 \end{proof}

\section{Proof of Theorem \ref{theo:PML_2D_Edge_Laplace_Disc}}\label{sec:proof_theorem8} 
We introduce the boundary terms
 \begin{align*}
 \mathcal{BT}_s\left(\widehat{\widetilde{v}}_{\eta}^{-}, \widehat{\widetilde{T}}_{\eta}^{-}\right)  = &  \frac{\Delta{x}}{2}  \frac{\Delta{z}}{2} \sum_{i = 1}^{P+1}\sum_{k = 1}^{P+1}\sum_{\eta = x,y,z}\left(\left( \widehat{\widetilde{T}}_\eta^{-*} \widehat{\widetilde{v}}_\eta^{-}  \right)\Big|_{r = 1} - \left( \widehat{\widetilde{T}}_\eta^{-*} \widehat{\widetilde{v}}_\eta^{-}  \right)\Big|_{r = -1}\right)_{i k} h_{i} h_{k} \\
&- \frac{\Delta{y}}{2}  \frac{\Delta{z}}{2}\sum_{i = 1}^{P+1}\sum_{k = 1}^{P+1}\sum_{\eta = x,y,z}\left(\left( \widehat{\widetilde{T}}_\eta^{-*} \widehat{\widetilde{v}}_\eta^{-}   \right)\Big|_{q = -1}\right)_{i k} h_{i} h_{k} \le 0,
 \end{align*}
 \begin{align*}
\mathcal{BT}_s\left(\widehat{\widetilde{v}}_{\eta}^{+}, \widehat{\widetilde{T}}_{\eta}^{+}\right)  = &  \frac{\Delta{x}}{2}  \frac{\Delta{z}}{2} \sum_{i = 1}^{P+1}\sum_{k = 1}^{P+1}\sum_{\eta = x,y,z}\left(\left( \widehat{\widetilde{T}}_\eta^{+*} \widehat{\widetilde{v}}_\eta^{+}  \right)\Big|_{r = 1} - \left( \widehat{\widetilde{T}}_\eta^{+*} \widehat{\widetilde{v}}_\eta^{+}  \right)\Big|_{r = -1}\right)_{i k} h_{i} h_{k} \\
&+ \frac{\Delta{y}}{2}  \frac{\Delta{z}}{2}\sum_{i = 1}^{P+1}\sum_{k = 1}^{P+1}\sum_{\eta = x,y,z}\left(\left( \widehat{\widetilde{T}}_\eta^{+*} \widehat{\widetilde{v}}_\eta^{+}   \right)\Big|_{q = 1}\right)_{i k} h_{i} h_{k} \le 0,
 \end{align*}
  the interface term
  \begin{align*}
  \mathcal{IT}_s\left(\widehat{\widetilde{v}}^{\pm}, \widehat{\widetilde{T}}^{\pm} \right) = - \frac{\Delta{y}}{2}  \frac{\Delta{z}}{2}\sum_{i = 1}^{P+1}\sum_{k = 1}^{P+1}\sum_{\eta = x,y,z}\left({\widehat{\widetilde{T}}}_\eta  \lJump{{\widehat{\widetilde{v}}}_\eta \rJump}\right)_{i k} h_{i} h_{k} \equiv 0,
\end{align*}
   and the fluctuation term
    \begin{align*}
  \mathcal{F}_{luc}\left(\widetilde{G},  Z\right) = - \sum_{\xi = x, y}\frac{\Delta{x}}{2} \frac{\Delta{y}}{2}  \frac{\Delta{z}}{2} \frac{2}{\Delta{\xi}}\sum_{\eta = x, y, z}\sum_{i = 1}^{P+1}\sum_{k = 1}^{P+1}\left(\left(\frac{1}{Z_{\eta}}|\widetilde{G}_\eta |^2\right)\Big|_{ -1} + \left(\frac{1}{Z_{\eta}}|\widetilde{G}_\eta |^2\right)\Big|_{ 1} \right)_{i, k} \le 0.
  \end{align*}
\begin{proof}
From the left,  multiply \eqref{eq:Laplace_disc_elastic_pml_1A_2D} with $\widetilde{\mathbf{Q}}^\dagger {\mathbf{H}}$, we have
  {
 \small
 \begin{equation}\label{eq:Laplace_disc_elastic_pml_1A_2D_proof_1}
\begin{split}
\left(sS_x\right)\widetilde{\mathbf{Q}}^\dagger {\mathbf{H}}\mathbf{P}^{-1}  \widetilde{\mathbf{Q}} = &   \sum_{\xi = x, y}\frac{1}{2}\widetilde{\mathbf{Q}}^\dagger \left(\mathbf{A}_{\xi}\left(\mathbf{H}\mathbf{D}_{\xi}\right) - \left(\mathbf{H}\mathbf{D}_{\xi}\right)^T\mathbf{A}_{\xi} \right)\widetilde{\mathbf{Q}}  + \widetilde{\mathbf{Q}}^\dagger {\mathbf{H}}\mathbf{P}^{-1}  \widetilde{\mathbf{F}} \\
+ & \widetilde{\mathbf{Q}}^\dagger\mathbf{H}\mathbf{H}_{\xi}^{-1}\left(\frac{1}{2}\mathbf{A}_{x}\left( \mathbf{B}_{\xi}\left(1,1\right) -  \mathbf{B}_{\xi}\left(-1,-1\right)\right)\widetilde{\mathbf{Q}} -\left({\mathbf{e}_{\xi}(-1)} \widetilde{\mathbf{FL}}_{\xi} +  {\mathbf{e}_{\xi}(1)} \widetilde{\mathbf{FR}}_{\xi}  \right) \right)
%
  \end{split}
  \end{equation}
  }
Using the identity \eqref{eq:identity_001} gives 
  {
 \small
 \begin{equation}\label{eq:Laplace_disc_elastic_pml_1A_2D_proof_2}
\begin{split}
&\left(sS_x\right)\widetilde{\mathbf{Q}}^\dagger {\mathbf{H}}\mathbf{P}^{-1}  \widetilde{\mathbf{Q}} =    \sum_{\xi = x, y}\frac{1}{2}\widetilde{\mathbf{Q}}^\dagger \left(\mathbf{A}_{\xi}\left(\mathbf{H}\mathbf{D}_{\xi}\right) - \left(\mathbf{H}\mathbf{D}_{\xi}\right)^T\mathbf{A}_{\xi} \right)\widetilde{\mathbf{Q}} +  \widetilde{\mathbf{Q}}^\dagger {\mathbf{H}}\mathbf{P}^{-1}  \widetilde{\mathbf{F}}  \\
- &\sum_{\xi = x, y}\frac{\Delta{x}}{2} \frac{\Delta{y}}{2}  \frac{\Delta{z}}{2} \frac{2}{\Delta{\xi}} \sum_{i=1}^{P+1}\sum_{k= 1}^{P+1}\left(\sum_{\eta = x,y,z} \left(\frac{1}{Z_\eta  }|\widetilde{G}_\eta |^2  - \widehat{\widetilde{T}}_\eta^* \widehat{\widetilde{v}}_\eta \right)\Big|_{ 1} +   \sum_{\eta = x,y,z} \left(\frac{1}{Z_\eta  }|\widetilde{G}_\eta |^2  + \widehat{\widetilde{T}}_\eta^* \widehat{\widetilde{v}}_\eta \right)\Big|_{-1}\right)_{ik} h_ih_k 
  \end{split}
  \end{equation}
  }
  Next we add \eqref{eq:Laplace_disc_elastic_pml_1A_2D_proof_2} to its complex conjugate, and the spatial derivative terms vanish, we have
   {
 \small
 \begin{equation}\label{eq:Laplace_disc_elastic_pml_1A_2D_proof_3}
\begin{split}
&\Re\left(sS_x\right)\widetilde{\mathbf{Q}}^\dagger {\mathbf{H}}\mathbf{P}^{-1}  \widetilde{\mathbf{Q}} = 
\frac{1}{2}\left(  \widetilde{\mathbf{Q}}^\dagger {\mathbf{H}}\mathbf{P}^{-1}  \widetilde{\mathbf{F}}  +  \widetilde{\mathbf{F}}^\dagger {\mathbf{H}}\mathbf{P}^{-1}  \widetilde{\mathbf{Q}} \right)\\
-& \sum_{\xi = x, y}\frac{\Delta{x}}{2} \frac{\Delta{y}}{2}  \frac{\Delta{z}}{2} \frac{2}{\Delta{\xi}} \sum_{i=1}^{P+1}\sum_{k= 1}^{P+1}\left(\sum_{\eta = x,y,z} \left(\frac{1}{Z_\eta  }|\widetilde{G}_\eta |^2  - \widehat{\widetilde{T}}_\eta^* \widehat{\widetilde{v}}_\eta \right)\Big|_{ 1} +   \sum_{\eta = x,y,z} \left(\frac{1}{Z_\eta  }|\widetilde{G}_\eta |^2  + \widehat{\widetilde{T}}_\eta^* \widehat{\widetilde{v}}_\eta \right)\Big|_{-1}\right)_{ik} h_ih_k 
%
  \end{split}
  \end{equation}
  }
  Using Cauchy-Schwarz inequality for the source term in \eqref{eq:Laplace_disc_elastic_pml_1A_2D_proof_3} and collecting contributions from both sides of the elements completes the proof.
\end{proof}

\section{Proof of Theorem \ref {theo:PML_3D_Corner_Laplace_Disc}}\label{sec:proof_theorem9} 
We introduce the boundary terms
 \begin{align*}
 \mathcal{BT}_s\left(\widehat{\widetilde{v}}_{\eta}^{-}, \widehat{\widetilde{T}}_{\eta}^{-}\right)  = &  \frac{\Delta{x}}{2}  \frac{\Delta{y}}{2} \sum_{i = 1}^{P+1}\sum_{k = 1}^{P+1}\sum_{\eta = x,y,z}\left(\left( \widehat{\widetilde{T}}_\eta^{-*} \widehat{\widetilde{v}}_\eta^{-}  \right)\Big|_{s = 1} - \left( \widehat{\widetilde{T}}_\eta^{-*} \widehat{\widetilde{v}}_\eta^{-}  \right)\Big|_{s = -1}\right)_{i k} h_{i} h_{k} \\
 +&  \frac{\Delta{x}}{2}  \frac{\Delta{z}}{2} \sum_{i = 1}^{P+1}\sum_{k = 1}^{P+1}\sum_{\eta = x,y,z}\left(\left( \widehat{\widetilde{T}}_\eta^{-*} \widehat{\widetilde{v}}_\eta^{-}  \right)\Big|_{r = 1} - \left( \widehat{\widetilde{T}}_\eta^{-*} \widehat{\widetilde{v}}_\eta^{-}  \right)\Big|_{r = -1}\right)_{i k} h_{i} h_{k}\\
&- \frac{\Delta{y}}{2}  \frac{\Delta{z}}{2}\sum_{i = 1}^{P+1}\sum_{k = 1}^{P+1}\sum_{\eta = x,y,z}\left(\left( \widehat{\widetilde{T}}_\eta^{-*} \widehat{\widetilde{v}}_\eta^{-}   \right)\Big|_{q = -1}\right)_{i k} h_{i} h_{k} \le 0,
 \end{align*}
 \begin{align*}
\mathcal{BT}_s\left(\widehat{\widetilde{v}}_{\eta}^{+}, \widehat{\widetilde{T}}_{\eta}^{+}\right)  = &  \frac{\Delta{x}}{2}  \frac{\Delta{y}}{2} \sum_{i = 1}^{P+1}\sum_{k = 1}^{P+1}\sum_{\eta = x,y,z}\left(\left( \widehat{\widetilde{T}}_\eta^{+*} \widehat{\widetilde{v}}_\eta^{+}  \right)\Big|_{s = 1} - \left( \widehat{\widetilde{T}}_\eta^{+*} \widehat{\widetilde{v}}_\eta^{+}  \right)\Big|_{s = -1}\right)_{i k} h_{i} h_{k} \\
+ &  \frac{\Delta{x}}{2}  \frac{\Delta{z}}{2} \sum_{i = 1}^{P+1}\sum_{k = 1}^{P+1}\sum_{\eta = x,y,z}\left(\left( \widehat{\widetilde{T}}_\eta^{+*} \widehat{\widetilde{v}}_\eta^{+}  \right)\Big|_{r = 1} - \left( \widehat{\widetilde{T}}_\eta^{+*} \widehat{\widetilde{v}}_\eta^{+}  \right)\Big|_{r = -1}\right)_{i k} h_{i} h_{k} \\
&+ \frac{\Delta{y}}{2}  \frac{\Delta{z}}{2}\sum_{i = 1}^{P+1}\sum_{k = 1}^{P+1}\sum_{\eta = x,y,z}\left(\left( \widehat{\widetilde{T}}_\eta^{+*} \widehat{\widetilde{v}}_\eta^{+}   \right)\Big|_{q = 1}\right)_{i k} h_{i} h_{k} \le 0,
 \end{align*}
  the interface term
  \begin{align*}
  \mathcal{IT}_s\left(\widehat{\widetilde{v}}^{\pm}, \widehat{\widetilde{T}}^{\pm} \right) = - \frac{\Delta{y}}{2}  \frac{\Delta{z}}{2}\sum_{i = 1}^{P+1}\sum_{k = 1}^{P+1}\sum_{\eta = x,y,z}\left({\widehat{\widetilde{T}}}_\eta  \lJump{{\widehat{\widetilde{v}}}_\eta \rJump}\right)_{i k} h_{i} h_{k} \equiv 0,
\end{align*}
   and the fluctuation term
    \begin{align*}
  \mathcal{F}_{luc}\left(\widetilde{G},  Z\right) = - \sum_{\xi = x, y, z}\frac{\Delta{x}}{2} \frac{\Delta{y}}{2}  \frac{\Delta{z}}{2} \frac{2}{\Delta{\xi}}\sum_{\eta = x, y, z}\sum_{i = 1}^{P+1}\sum_{k = 1}^{P+1}\left(\left(\frac{1}{Z_{\eta}}|\widetilde{G}_\eta |^2\right)\Big|_{ -1} + \left(\frac{1}{Z_{\eta}}|\widetilde{G}_\eta |^2\right)\Big|_{ 1} \right)_{i, k} \le 0.
  \end{align*}
\begin{proof}
From the left,  multiply \eqref{eq:Laplace_disc_elastic_pml_1A_3D} with $\widetilde{\mathbf{Q}}^\dagger {\mathbf{H}}$, we have
  {
 \small
 \begin{equation}\label{eq:Laplace_disc_elastic_pml_1A_3D_proof_1}
\begin{split}
\left(sS_x\right)\widetilde{\mathbf{Q}}^\dagger {\mathbf{H}}\mathbf{P}^{-1}  \widetilde{\mathbf{Q}} = &   \sum_{\xi = x, y, z}\frac{1}{2}\widetilde{\mathbf{Q}}^\dagger \left(\mathbf{A}_{\xi}\left(\mathbf{H}\mathbf{D}_{\xi}\right) - \left(\mathbf{H}\mathbf{D}_{\xi}\right)^T\mathbf{A}_{\xi} \right)\widetilde{\mathbf{Q}}  + \widetilde{\mathbf{Q}}^\dagger {\mathbf{H}}\mathbf{P}^{-1}  \widetilde{\mathbf{F}} \\
+ & \widetilde{\mathbf{Q}}^\dagger\mathbf{H}\mathbf{H}_{\xi}^{-1}\left(\frac{1}{2}\mathbf{A}_{x}\left( \mathbf{B}_{\xi}\left(1,1\right) -  \mathbf{B}_{\xi}\left(-1,-1\right)\right)\widetilde{\mathbf{Q}} -\left({\mathbf{e}_{\xi}(-1)} \widetilde{\mathbf{FL}}_{\xi} +  {\mathbf{e}_{\xi}(1)} \widetilde{\mathbf{FR}}_{\xi}  \right) \right)
  \end{split}
  \end{equation}
  }
Using the identity \eqref{eq:identity_001} gives 
  {
 \small
 \begin{equation}\label{eq:Laplace_disc_elastic_pml_1A_3D_proof_2}
\begin{split}
&\left(sS_x\right)\widetilde{\mathbf{Q}}^\dagger {\mathbf{H}}\mathbf{P}^{-1}  \widetilde{\mathbf{Q}} =    \sum_{\xi = x, y, z}\frac{1}{2}\widetilde{\mathbf{Q}}^\dagger \left(\mathbf{A}_{\xi}\left(\mathbf{H}\mathbf{D}_{\xi}\right) - \left(\mathbf{H}\mathbf{D}_{\xi}\right)^T\mathbf{A}_{\xi} \right)\widetilde{\mathbf{Q}} +  \widetilde{\mathbf{Q}}^\dagger {\mathbf{H}}\mathbf{P}^{-1}  \widetilde{\mathbf{F}}  \\
- &\sum_{\xi = x, y}\frac{\Delta{x}}{2} \frac{\Delta{y}}{2}  \frac{\Delta{z}}{2} \frac{2}{\Delta{\xi}} \sum_{i=1}^{P+1}\sum_{k= 1}^{P+1}\left(\sum_{\eta = x,y,z} \left(\frac{1}{Z_\eta  }|\widetilde{G}_\eta |^2  - \widehat{\widetilde{T}}_\eta^* \widehat{\widetilde{v}}_\eta \right)\Big|_{ 1} +   \sum_{\eta = x,y,z} \left(\frac{1}{Z_\eta  }|\widetilde{G}_\eta |^2  + \widehat{\widetilde{T}}_\eta^* \widehat{\widetilde{v}}_\eta \right)\Big|_{-1}\right)_{ik} h_ih_k 
  \end{split}
  \end{equation}
  }
  Next we add \eqref{eq:Laplace_disc_elastic_pml_1A_3D_proof_2} to its complex conjugate, and the spatial derivative terms vanish, we have
   {
 \small
 \begin{equation}\label{eq:Laplace_disc_elastic_pml_1A_3D_proof_3}
\begin{split}
&\Re\left(sS_x\right)\widetilde{\mathbf{Q}}^\dagger {\mathbf{H}}\mathbf{P}^{-1}  \widetilde{\mathbf{Q}} = 
\frac{1}{2}\left(  \widetilde{\mathbf{Q}}^\dagger {\mathbf{H}}\mathbf{P}^{-1}  \widetilde{\mathbf{F}}  +  \widetilde{\mathbf{F}}^\dagger {\mathbf{H}}\mathbf{P}^{-1}  \widetilde{\mathbf{Q}} \right)\\
-& \sum_{\xi = x, y}\frac{\Delta{x}}{2} \frac{\Delta{y}}{2}  \frac{\Delta{z}}{2} \frac{2}{\Delta{\xi}} \sum_{i=1}^{P+1}\sum_{k= 1}^{P+1}\left(\sum_{\eta = x,y,z} \left(\frac{1}{Z_\eta  }|\widetilde{G}_\eta |^2  - \widehat{\widetilde{T}}_\eta^* \widehat{\widetilde{v}}_\eta \right)\Big|_{ 1} +   \sum_{\eta = x,y,z} \left(\frac{1}{Z_\eta  }|\widetilde{G}_\eta |^2  + \widehat{\widetilde{T}}_\eta^* \widehat{\widetilde{v}}_\eta \right)\Big|_{-1}\right)_{ik} h_ih_k 
%
  \end{split}
  \end{equation}
  }
  Using Cauchy-Schwarz inequality for the source term in \eqref{eq:Laplace_disc_elastic_pml_1A_3D_proof_3} and collecting contributions from both sides of the elements completes the proof.
\end{proof}


\end{document}